\documentclass{cmslatex}
\usepackage[paperwidth=7in, paperheight=10in, margin=.875in]{geometry}
 \usepackage[backref,colorlinks,linkcolor=red,anchorcolor=green,citecolor=blue]{hyperref}
\usepackage{amsfonts,amssymb}
\usepackage{amsmath}
\usepackage{graphicx}
\usepackage{cite}
\usepackage{enumerate}
\usepackage{subfigure}
\usepackage{nicefrac}
\usepackage{amsfonts}
\newcommand{\Dh}{\Delta_h}
\newcommand{\nabh}{\nabla_{\! h}}

\newcommand{\g}{\mbox{\boldmath$g$}}

\newcommand{\hf}{\nicefrac{1}{2}}
\newcommand{\nrm}[1]{\left\| #1 \right\|}

\newcommand{\cipgen}[3]{\left\langle #1 , #2 \right\rangle_{#3}}

\newcommand\dt {{\Delta t}}

\newcommand{\eipx}[2]{\left[ #1 , #2 \right]_{\rm x}}
\newcommand{\eipy}[2]{\left[ #1 , #2 \right]_{\rm y}}

\newcommand{\hh}{\mbox{\boldmath$h$}}
\newcommand{\HH}{\mbox{\boldmath$H$}}

   \allowdisplaybreaks
\begin{document}
 \title{A positivity-preserving, second-order energy stable and convergent numerical scheme for a ternary system of macromolecular microsphere composite hydrogels \thanks{Received date, and accepted date (The correct dates will be entered by the editor).}}

\author{Lixiu Dong\thanks{Department of Mathematics, Faculty of Arts and Sciences, Beijing Normal University, Zhuhai 519087, P.R. China, (lxdong@bnu.edu.cn). }
\and Cheng Wang\thanks{Department of Mathematics; University of Massachusetts; North Dartmouth, MA 02747, USA, (cwang1@umassd.edu).}
\and Zhengru Zhang\thanks{Laboratory of Mathematics and Complex Systems, Ministry of Education and School of Mathematical Sciences, Beijing Normal University, Beijing 100875, P.R. China, (Corresponding Author: zrzhang@bnu.edu.cn).}}

\pagestyle{myheadings} \markboth{Second order scheme for the ternary MMC system}{L. Dong, C. Wang and Z. Zhang} \maketitle

          \begin{abstract}
          	A second order accurate numerical scheme is proposed and analyzed for the periodic three-component Macromolecular Microsphere Composite(MMC) hydrogels system, a ternary Cahn-Hilliard system with a Flory-Huggins-deGennes free energy potential. This numerical scheme with energy stability is based on the Backward Differentiation Formula(BDF) method in time derivation combining with Douglas-Dupont regularization term, combined the finite difference method in space. We provide a theoretical justification of positivity-preserving property for all the singular terms, i.e., not only the two phase variables are always between $0$ and $1$, but also the sum of the two phase variables is between $0$ and $1$, at a point-wise level. In addition, an optimal rate convergence analysis is provided in this paper, in which a higher order asymptotic expansion of the numerical solution,
          	the rough error estimate and refined error estimate techniques have to be included to accomplish such an analysis. This paper will be the first to combine the following theoretical properties for a second order accurate numerical scheme for the ternary MMC system: (i) unique solvability and positivity-preserving property; (ii) energy stability; (iii) and optimal rate convergence. A few numerical results are also presented.
        \end{abstract}

\section{Introduction}\label{intro}

Macromolecular microsphere composite (MMC) hydrogels, a class of polymeric materials, have attracted theoretical and experimental studies due to their well-defined network microstructures and high mechanical strength. 

A binary mathematical model was presented in~\cite{Zhai2012Investigation} to describe the periodic structures and the phase transitions of the MMC hydrogels based on Boltzmann entropy theory. The corresponding model leads to the MMC-TDGL equation, with a similar structure to the Cahn-Hilliard equation, but with certain singular gradient coefficients, is discussed in~\cite{Dong2019a, Lixiao2015, Li2016An, Yuan2022}. Also see the related works~\cite{Drury2003, Edlund2010, Huang2007, Ji2021, Johnson2010} for the hydrogel model. The binary Cahn-Hilliard equation with either polynomial Ginzburg-Landau or singular Flory-Huggins-type free energy models spinodal decomposition, phase separation, and coarsening in a two-phase fluid. There have been many theoretical analyses and numerical approximations for these kinds of gradient flows in the two-phase case~\cite{ chen16, chenY18, Feng2017A, diegel15a, diegel17, han15, liuY17, CHBDF2}.

For the ternary Cahn-Hilliard system, the general framework is to adopt three independent phase variables $(\phi_1,\phi_2,\phi_3)$ while enforcing a mass conservation (or ``no-voids") constraint $\phi_1+\phi_2+\phi_3=1$. See the related works~\cite{Boyer2006Numerical, Boyer2011Numerical, Yang2017Numerical,Yuan2021a}. A ternary system with Flory-Huggins-deGennes energy potential \cite{Ji2021} has been of great scientific interests, which turns out to be an improvement over the model proposed in~\cite{Zhai2012Investigation}, as it removes certain limiting assumptions. The singular Flory-Huggins-deGennes energy potential is as follows:
\begin{equation*}
	G_o(\phi_1,\phi_2,\phi_3) = \int_{\Omega} \left\{ S_o(\phi_1,\phi_2,\phi_3)+ \frac{1}{36} \sum_{i=1}^3 \frac{\varepsilon_i^2}{\phi_i} |\nabla\phi_i|^2 + H_o(\phi_1,\phi_2,\phi_3)\right\} d{\bf x},
\end{equation*}
where $S_o(\phi_1,\phi_2,\phi_3)+ H_o(\phi_1,\phi_2,\phi_3)$ is the reticular (Flory-Huggins style) free energy density:
\begin{align*}
	S_o(\phi_1,\phi_2,\phi_3) & = \frac{\phi_1}{M_0} \ln \frac{\alpha \phi_1}{M_0}+\frac{\phi_2}{N_0} \ln \frac{\beta \phi_2}{N_0} + \phi_3 \ln \phi_3,
	\\
	H_o(\phi_1,\phi_2,\phi_3) & = \chi_{12} \phi_1 \phi_2 + \chi_{13} \phi_1 \phi_3 + \chi_{23}\phi_2 \phi_3.
\end{align*}
$S_o$ is the ideal solution part and $H_o$ is the entropy of mixing part. 
The domain $\Omega \subset \mathbb{R}^2$ is assumed open, bounded, and simply connected. We focus on the 2-D case for simplicity of presentation, while the extension to the 3-D gradient flow is straightforward. The mass-conservative phase variables $\phi_1$, $\phi_2$ and $\phi_3$, represent the concentration of the macromolecular microsphere, the polymer chain, and the solvent, respectively. These three phase variables are subject to the ``no-voids" constraint $\phi_1 + \phi_2 + \phi_3= 1$. We denote by $M_0$ the relative volume of one macromolecular microsphere, and by $N_0$ the degree of polymerization of the polymer chains. The coefficient $\varepsilon_i$ is called the statistical segment length of the $i$-th component, which is always positive. The parameters $\alpha$ and $\beta$ depend on $M_0$ and $N_0$:

   \[
\alpha= \pi \Big( \Big( \frac{M_0}{\pi} \Big)^\frac12 + \frac{N_0}{2} \Big)^2, \quad
\beta= 2  \Big( \frac{M_0}{\pi} \Big)^\frac12 + N_0 .
\]
By $\chi_{12}, \chi_{13}$, and $\chi_{23}$ we denote the Huggins interaction parameters between (i) the macromolecular microspheres and polymer chains, (ii) the macromolecular microspheres and solvent, and (iii) the polymer chains and solvent, respectively. All these parameters are positive, and the following inequality is assumed to guarantee the concavity of the entropy of mixing $H_0$ term: 
\[
4\chi_{13}\chi_{23}-(\chi_{12}-\chi_{13}-\chi_{23})^2 >0.
\]
Making use of the no-voids constraint $\phi_3 = 1 - \phi_1 - \phi_2$, we can rewrite the energy functional as
\begin{align}
	\label{continuous energy1}
	G(\phi_1,\phi_2) &= \int_{\Omega} \bigg\{ S(\phi_1,\phi_2)+ \frac{\varepsilon_1^2|\nabla\phi_1|^2}{36\phi_1} + \frac{\varepsilon_2^2|\nabla\phi_2|^2}{36\phi_2}+
	\frac{\varepsilon_3^2|\nabla(1-\phi_1-\phi_2)|^2}{36(1-\phi_1-\phi_2)}
	\\
	& \quad + H(\phi_1,\phi_2) \bigg\} d{\bf x},	
	\nonumber
\end{align}
where
\begin{align*}
	S(\phi_1,\phi_2) & = \frac{\phi_1}{M_0} \ln \frac{\alpha \phi_1}{M_0}+\frac{\phi_2}{N_0} \ln \frac{\beta \phi_2}{N_0} + (1-\phi_1-\phi_2) \ln (1-\phi_1-\phi_2),
	\\
	H(\phi_1,\phi_2) &= \chi_{12} \phi_1 \phi_2 + \chi_{13} \phi_1 (1-\phi_1-\phi_2) + \chi_{23}\phi_2 (1-\phi_1-\phi_2).
\end{align*}
The ternary MMC dynamic equations are $H^{-1}$ gradient flows associated with the given energy functional \eqref{continuous energy1}:
\begin{equation}\label{MMC3term-equation}
	\partial_t \phi_1 = \mathcal{M}_1 \Delta \mu_1,\quad \partial_t \phi_2 = \mathcal{M}_2 \Delta
	\mu_2,
\end{equation}
where $\mathcal{M}_1, \mathcal{M}_2 >0$ are mobilities, which are assumed to be positive constants. The terms $\mu_1$ and $\mu_2$ are the chemical potentials with respect to $\phi_1$ and $\phi_2$,
respectively, i.e.,
\begin{align}
	\mu_1:=\delta_{\phi_1}G & = \frac{1}{M_0}\ln \frac{\alpha \phi_1}{M_0} - \ln(1-\phi_1-\phi_2) - 2 \chi_{13} \phi_1 +(\chi_{12}-\chi_{13}-\chi_{23})\phi_2
	\nonumber
	\\
	& \quad  + \chi_{13} + \frac{1}{M_0} -1 - \frac{\varepsilon_1^2|\nabla\phi_1|^2}{36\phi_1^2} - \nabla \cdot \left(\frac{\varepsilon_1^2\nabla \phi_1}{18\phi_1}\right)
	\label{mu_1}
	\\
	& \quad + \frac{\varepsilon_3^2|\nabla(1-\phi_1-\phi_2)|^2}{36(1-\phi_1-\phi_2)^2}
	+ \nabla \cdot \left(\frac{\varepsilon_3^2 \nabla (1-\phi_1-\phi_2)}{18(1-\phi_1-\phi_2)}
	\right),
	\nonumber
	\\
	\mu_2:=\delta_{\phi_2}G & = \frac{1}{N_0}\ln \frac{\beta \phi_2}{N_0} - \ln(1-\phi_1-\phi_2) - 2 \chi_{23} \phi_2 +(\chi_{12}-\chi_{13}-\chi_{23})\phi_1
	\nonumber
	\\
	& \quad  + \chi_{23} + \frac{1}{N_0} -1  - \frac{\varepsilon_2^2|\nabla\phi_2|^2}{36\phi_2^2} - \nabla \cdot \left(\frac{\varepsilon_2^2\nabla \phi_2}{18\phi_2}\right)
	\label{mu_2}
	\\
	& \quad  + \frac{\varepsilon_3^2|\nabla(1-\phi_1-\phi_2)|^2}{36(1-\phi_1-\phi_2)^2} + \nabla \cdot \left(\frac{\varepsilon_3^2 \nabla (1-\phi_1-\phi_2)}{18(1-\phi_1-\phi_2)} \right).
	\nonumber
\end{align}
For simplicity,
periodic boundary conditions are assumed. These equations would reduce to the classical ternary Cahn-Hilliard system if the gradient energy coefficients $\varepsilon_i^2/(36\phi_i)$ were replaced by $\varepsilon_i^2/2$. In any case, it is then easy to see that the energy is non-increasing for the ternary MMC model. The evolution equations \eqref{MMC3term-equation} are mass conservative; the mass fluxes are proportional to the gradients of the respective chemical potentials. Clearly the phase fields must satisfy $\phi_1>0$, $\phi_2>0$, and $1-\phi_1-\phi_2>0$ for the model to make sense physically and mathematically.  We define the following \emph{Gibbs Triangles} for use later:
\begin{equation}
	\mathcal{G} := \left\{ (\phi_1,\phi_2)\in\mathbb{R}^2 \ \middle| \ 0 < \phi_1,\, \phi_2, \ \phi_1+\phi_2 < 1 \right\}	,
	\label{eqn:Gibbs-tri}
\end{equation}
and, for $\delta \ge 0$, 
\begin{equation*}
	\mathcal{G}_\delta := \left\{ (\phi_1,\phi_2)\in\mathbb{R}^2 \ \middle| \ \delta <  \phi_1,\, \phi_2, \ \phi_1+\phi_2 <  1 -\delta \right\}	.
\end{equation*}
Of course, $\mathcal{G}_0 = \mathcal{G}$, and $\mathcal{G}_\delta\subseteq \mathcal{G}$, for each $\delta\ge0$. If $\left(\phi_1(\, \cdot\, , t),\phi_2(\, \cdot\, , t)\right)\in \mathcal{G}$, point-wise, for all $t\ge 0$, we say that the \emph{positivity-preserving property} holds for the equation. If, for some strictly positive $\delta >0$, $\left(\phi_1(\, \cdot\, , t),\phi_2(\, \cdot\, , t)\right)\in \mathcal{G}_\delta$, point-wise, for all $t\ge 0$, we say that a \emph{strict separation property} holds for the equation.

There have been some numerical works for the ternary MMC system, while most existing works have focused on first order accurate (in time) algorithms. Such as the recent literature \cite{Dong2020b,dong2022}, the authors presented a first order discrete finite difference numerical scheme based on the convex splitting method of the free energy with logarithmic potential, established a theoretical justification of the positivity property and convergence analysis. Also see the related finite element work \cite{Yuan2021a}. One well-known drawback of the first order convex splitting approach is that an extra dissipation added to ensure unconditional stability also introduces a significant amount of numerical error~\cite{Christlieb2014High}. Due to this fact, second-order energy stable methods have been highly desirable, which preserves all important theoretical features, i.e., unique solvability, positivity preserving, mass conversation, energy stability and convergence analysis.

The goal of this paper is to extend the convex-splitting framework to develop a second order in both time and space for the ternary MMC system. We propose and analyze a numerical scheme with four theoretical properties justified: unique solvability and positivity-preserving, mass conservation, energy stability and convergence analysis. This scheme is based on the 2nd BDF temporal approximation and the finite difference method in space for the ternary MMC system.  Based on the idea of convex splitting, we treat the convex part implicitly and the linear part explicitly using the second-order Adams-Bashforth extrapolation formula. In addition, a second order Douglas-Dupont regularization of the form $A_i\dt\Delta_h(\phi_i^{n+1}-\phi_i^n)$ is specifically introduced to ensure the energy stability in this paper, this technique is similar in~\cite{Feng2017A,cheng18ch,Dong2020a,CHBDF2}. Moreover, the highly nonlinear and singular nature of the surface diffusion coefficients makes the system turn to be a very challenging issue. In this paper, we will adopt similar techniques  in~\cite{chen19b,Dong2020b} to estimate the positivity property. First, the fully discrete numerical scheme is equivalent to a minimization of a strictly convex discrete energy functional, so we can transform the positivity preserving problem of the numerical solution into the problem that the minimizer of this functional could not occur on the boundary points. A more careful analysis reveals that, the convex and the singular natures of these implict nonlinear parts prevent the numerical solutions approach the singular limit values of $0$ and $1$, so that the phase variables are always between $0$ and $1$. At the same time, the sum of these two phase variables is between $0$  and $1$, at a point-wise level. Although the extra term $A_i\dt \Delta_h (\phi_i^{n+1}-\phi_i^n)$ is added into the numerical scheme, it does not matter because the logarithmic function always changes faster than the linear function as the phase variable approaches the boundary points. 
For convergence analysis, to control the explicit part of the extra regularization term, a higher order asymptotic expansion (up to third order temporal accuracy and fourth order spatial accuracy) has to be involved. To overcome the highly nonlinear and singular nature of the surface diffusion coefficients, a rough error estimate has to be performed, so that the $\ell^{\infty}$ bound for $\phi_i$ could be derived. This $\ell^\infty$ estimate yields the upper and lower bounds of the three variables, and these bounds play a crucial role in the subsequent analysis. Finally, the refined error estimate is carried out to accomplish the desired convergence result.  

The rest part of this paper is organized as follows. In Section \ref{sec:numerical scheme}, we present a finite difference scheme basd on the 2nd BDF method and the idea of convex splitting method of the energy functional. In Section \ref{sec:positivity-preserving property}, the unique solvability and the positivity preserving property of the numerical solutions are analyzed. The energy stability analysis is provided in Section \ref{sec:unconditional energy stability}. The detailed convergence analysis is given by Section \ref{sec:convergence}. Some numerical results are presented in Section \ref{sec:numerical results}. Finally, concluding remarks are made in Section \ref{sec:conclusion}.
\section{The fully discrete numerical scheme} \label{sec:numerical scheme}	

\subsection{The finite difference spatial discretization}
\label{subsec:finite difference}

We use the notation and results for some discrete functions and operators from~\cite{Feng2017A, wise10, wise09a}. Let $\Omega = (0,L_x)\times(0,L_y)$, where for simplicity, we assume $L_x =L_y =: L > 0$. Let $N\in\mathbb{N}$ be given, and define the grid spacing $h := \frac{L}{N}$, i.e., a uniform spatial mesh size is taken for simplicity of presentation. 
We define the following two uniform, infinite grids with grid spacing $h>0$: $E := \{ p_{i+\hf} \ |\ i\in {\mathbb{Z}}\}$, $C := \{ p_i \ |\ i\in {\mathbb{Z}}\}$, where $p_i = p(i) := (i-\hf)\cdot h$. Consider the following 2-D discrete $N^2$-periodic function spaces:
	\begin{eqnarray*}
	\begin{aligned}
		{\mathcal C}_{\rm per} &:= \left\{\nu: C\times C
		\rightarrow {\mathbb{R}}\ \middle| \ \nu_{i,j} = \nu_{i+\alpha N,j+\beta N}, \ \forall \, i,j,\alpha,\beta, \in \mathbb{Z} \right\},
		\\
		{\mathcal E}^{\rm x}_{\rm per} &:=\left\{\nu: E\times C \rightarrow {\mathbb{R}}\ \middle| \ \nu_{i+\frac12,j}= \nu_{i+\frac12+\alpha N,j+\beta N}, \ \forall \, i,j,\alpha,\beta \in \mathbb{Z}\right\} ,
	\end{aligned}
\end{eqnarray*}
in which identification $\nu_{i,j} = \nu(p_i,p_j)$ is taken. The space  ${\mathcal E}^{\rm y}_{\rm per}$ is analogously defined. The functions of ${\mathcal C}_{\rm per}$ are called {\emph{cell-centered functions}}, and the functions of ${\mathcal E}^{\rm x}_{\rm per}$, ${\mathcal E}^{\rm y}_{\rm per}$ are called {\emph{east-west}},  {\emph{north-south face-centered functions}}, respectively.  We also define the mean zero space $\mathring{\mathcal C}_{\rm per}:= \{\nu\in {\mathcal C}_{\rm per} \ \Big| 0 = \overline{\nu} :=  \frac{h^2}{| \Omega|} \sum_{i,j=1}^N \nu_{i,j} \}$, and denote $\vec{\mathcal{E}}_{\rm per} := {\mathcal E}^{\rm x}_{\rm per}\times {\mathcal E}^{\rm y}_{\rm per}$. The space 
$\vec{\mathcal{C}}_{\rm per}^{\mathcal{G}}$ is defined as 
	\[
\vec{\mathcal{C}}_{\rm per}^{\mathcal{G}} := \left\{(u_1,u_2) \in \mathcal{C}_{\rm per}\times \mathcal{C}_{\rm per} \ \middle| \ ({u_1}_{i,j},{u_2}_{i,j}) \in\mathcal{G} , \quad i,j\in\mathbb{Z}  \right\}, 
\]
where ${\mathcal G}$ is the Gibbs Triangle~\eqref{eqn:Gibbs-tri}. In addition, the following difference and average operators are introduced:
\begin{eqnarray*}
	&& A_x \nu_{i+\hf,j} := \frac{1}{2}\left(\nu_{i+1,j} + \nu_{i,j} \right), \quad D_x \nu_{i+\hf,j} := \frac{1}{h}\left(\nu_{i+1,j} - \nu_{i,j} \right),
	\\
	&& A_y \nu_{i,j+\hf} := \frac{1}{2}\left(\nu_{i,j+1} + \nu_{i,j} \right), \quad D_y \nu_{i,j+\hf} := \frac{1}{h}\left(\nu_{i,j+1} - \nu_{i,j} \right) ,
\end{eqnarray*}
with $A_x,\, D_x: {\mathcal C}_{\rm per}\rightarrow{\mathcal E}_{\rm per}^{\rm x}$, $A_y,\, D_y: {\mathcal C}_{\rm per}\rightarrow{\mathcal E}_{\rm per}^{\rm y}$.
Likewise,
\begin{eqnarray*}
	&& a_x \nu_{i, j} := \frac{1}{2}\left(\nu_{i+\hf, j} + \nu_{i-\hf, j} \right),	 \quad d_x \nu_{i, j} := \frac{1}{h}\left(\nu_{i+\hf, j} - \nu_{i-\hf, j} \right),
	\\
	&& a_y \nu_{i,j} := \frac{1}{2}\left(\nu_{i,j+\hf} + \nu_{i,j-\hf} \right),	 \quad d_y \nu_{i,j} := \frac{1}{h}\left(\nu_{i,j+\hf} - \nu_{i,j-\hf} \right),
\end{eqnarray*}
with $a_x,\, d_x : {\mathcal E}_{\rm per}^{\rm x}\rightarrow{\mathcal C}_{\rm per}$, and $a_y,\, d_y : {\mathcal E}_{\rm per}^{\rm y}\rightarrow{\mathcal C}_{\rm per}$. The discrete gradient $\nabh:{\mathcal C}_{\rm per}\rightarrow \vec{\mathcal{E}}_{\rm per}$ and the discrete divergence $\nabh\cdot :\vec{\mathcal{E}}_{\rm per} \rightarrow {\mathcal C}_{\rm per}$ are given by
\[
\nabh\nu_{i,j} =\left( D_x\nu_{i+\hf, j},  D_y\nu_{i, j+\hf} \right) ,
\quad
\nabh\cdot\vec{f}_{i,j} = d_x f^x_{i,j}	+ d_y f^y_{i,j} ,
\]
where $\vec{f} = (f^x,f^y)\in \vec{\mathcal{E}}_{\rm per}$. The standard 2-D discrete Laplacian, $\Delta_h : {\mathcal C}_{\rm per}\rightarrow{\mathcal C}_{\rm per}$, becomes
\begin{align*}
	\Delta_h \nu_{i,j} 
	:=  d_x(D_x \nu)_{i,j} + d_y(D_y \nu)_{i,j}
	= \frac{1}{h^2}\left( \nu_{i+1,j}+\nu_{i-1,j}+\nu_{i,j+1}+\nu_{i,j-1} - 4\nu_{i,j}\right).
\end{align*}
More generally, if $\mathcal{D}$ is a periodic \emph{scalar} function that is defined at all of the face-centered points and $\vec{f}\in\vec{\mathcal{E}}_{\rm per}$, then $\mathcal{D}\vec{f}\in\vec{\mathcal{E}}_{\rm per}$, assuming point-wise multiplication, and we may define
$\nabla_h\cdot \big(\mathcal{D} \vec{f} \big)_{i,j} = d_x\left(\mathcal{D}f^x\right)_{i,j}  + d_y\left(\mathcal{D}f^y\right)_{i,j}$.
Specifically, if $\nu\in \mathcal{C}_{\rm per}$, then $\nabla_h \cdot\left(\mathcal{D} \nabla_h  \ \ \right):\mathcal{C}_{\rm per} \rightarrow \mathcal{C}_{\rm per}$ is defined point-wise via
$\nabla_h\cdot \big(\mathcal{D} \nabla_h \nu \big)_{i,j} = d_x\left(\mathcal{D}D_x\nu\right)_{i,j}  + d_y\left(\mathcal{D} D_y\nu\right)_{i,j}$. In particular, suppose that $\nu,\phi\in{\mathcal{C}}_{\rm per}$ are grid functions and $\sigma:\mathbb{R}\to\mathbb{R}$ is a continuous function. Then we define
\[
\nabla_h\cdot \big(\sigma(\mathcal{A}_h\nu) \nabla_h \phi \big)_{i,j} := d_x\left(\sigma(A_x\nu) D_x\nu\right)_{i,j}  + d_y\left( \sigma(A_y\nu)  D_y\phi\right)_{i,j},
\]
where $\mathcal{A}_h\nu$ is understood to be a periodic function defined at the face-centered points obtained by doing appropriate east-west and north-south averages.

In addition, the following grid inner products are defined:
\begin{small}
\begin{equation*}
	\begin{aligned}
		\langle \nu , \xi \rangle &:= h^2 \sum_{i,j=1}^N  \nu_{i,j}\, \xi_{i,j}, \quad
		\nu,\, \xi\in {\mathcal C}_{\rm per},\quad
		[ \vec{f}_1 , \vec{f}_2 ] := \eipx{f_1^x}{f_2^x}	+ \eipy{f_1^y}{f_2^y} , \quad \vec{f}_i = (f_i^x,f_i^y) \in \vec{\mathcal{E}}_{\rm per},  
		\\
		\eipx{\nu}{\xi} &:= \langle a_x(\nu\xi) , 1 \rangle ,\quad \nu,\, \xi\in{\mathcal E}^{\rm x}_{\rm per}, \quad
		\eipy{\nu}{\xi}  := \langle a_y(\nu\xi) , 1 \rangle ,  \quad \nu,\, \xi\in{\mathcal E}^{\rm y}_{\rm per} .
	\end{aligned}
\end{equation*}
\end{small}	
Subsequently, we define the following norms for cell-centered functions. If $\nu\in {\mathcal C}_{\rm per}$, then $\nrm{\nu}_2^2 := \langle \nu , \nu \rangle$; $\nrm{\nu}_p^p := \langle |\nu|^p, 1 \rangle$, for $1\le p< \infty$, and $\nrm{\nu}_\infty := \max_{1\le i,j\le N}\left|\nu_{i,j}\right|$. The gradient norms are introduced as follows:
\begin{eqnarray*}
	&&
	\nrm{ \nabla_h \nu}_2^2 : = [\nabh \nu, \nabh \nu ] = \eipx{D_x\nu}{D_x\nu} + \eipy{D_y\nu}{D_y\nu} ,  \quad \mbox{for} \, \,
	\nu \in{\mathcal C}_{\rm per} ,
	\\
	&&
	\nrm{\nabla_h \nu}_p := \left( \eipx{|D_x\nu|^p}{1} + \eipy{|D_y\nu|^p}{1} \right)^{\frac1p}  ,  \quad  1\le p< \infty .
\end{eqnarray*}
The discrete $H^1$ norm is defined as $\nrm{\nu}_{H_h^1}^2 : =  \nrm{\nu}_2^2+ \nrm{ \nabla_h \nu}_2^2$.

\begin{lem}
	\label{lemma1}   \cite{wang11a, wise09a}
	Let $\mathcal{D}$ be an arbitrary periodic, scalar function defined on all of the face center points. For any $\psi, \nu \in {\mathcal C}_{\rm per}$ and any $\vec{f}\in\vec{\mathcal{E}}_{\rm per}$, the following summation by parts formulas are valid:
	\begin{equation*}
		\langle \psi , \nabla_h\cdot\vec{f} \rangle = - [ \nabla_h \psi , \vec{f} ], \quad
		\langle \psi, \nabla_h\cdot \left(\mathcal{D}\nabla_h\nu\right) \rangle
		= - [ \nabla_h \psi , \mathcal{D}\nabla_h\nu ] .
	\end{equation*}
\end{lem}	
To facilitate the analysis, we need to introduce a discrete analogue of the space $H_{per}^{-1}\left(\Omega\right)$, as outlined in~\cite{wang11a}. Suppose that $\mathcal{D}$ is a positive, periodic scalar function defined at edge-center points. For any $\phi\in{\mathcal C}_{\rm per}$, there exists a unique $\psi\in\mathring{\mathcal C}_{\rm per}$ that solves
\begin{eqnarray*}
	\mathcal{L}_{\mathcal{D}}(\psi):= - \nabla_h \cdot\left(\mathcal{D}\nabla_h \psi\right) =
	\phi - \overline{\phi} ,
\end{eqnarray*}
where $\overline{\phi} := |\Omega|^{-1}\langle \phi,1 \rangle$. We equip this space with a bilinear form: for any $\phi_1,\, \phi_2\in \mathring{\mathcal C}_{\rm per}$, define
\begin{equation*}
	\cipgen{ \phi_1 }{ \phi_2 }{\mathcal{L}_{\mathcal{D}}^{-1}} := [\mathcal{D}\nabla_h \psi_1 ,\nabla_h \psi_2],
\end{equation*}
where $\psi_i\in\mathring{\mathcal C}_{\rm per}$ is the unique solution to
\begin{equation*}
	\mathcal{L}_{\mathcal{D}}(\psi_i):= - \nabla_h \cdot\left(\mathcal{D}\nabla_h \psi_i\right)  =
	\phi_i, \quad i = 1, 2.
\end{equation*}
The following identity~\cite{wang11a} is easy to prove via summation-by-parts:
\begin{equation*}
	\cipgen{\phi_1 }{ \phi_2 }{\mathcal{L}_{\mathcal{D}}^{-1}} = \langle \phi_1,
		\mathcal{L}_{\mathcal{D}}^{-1} (\phi_2)\rangle = \langle \mathcal{L}_{\mathcal{D}}^{-1} (\phi_1),\phi_2 \rangle,
\end{equation*}
and since $\mathcal{L}_{\mathcal{D}}$ is symmetric positive definite, $\cipgen{ \ \cdot \ }{\ \cdot \ }{\mathcal{L}_{\mathcal{D}}^{-1}}$ is an inner product on $\mathring{\mathcal C}_{\rm per}$. When $\mathcal{D}\equiv 1$, we drop the subscript and write $\mathcal{L}_{1} = \mathcal{L} = -\Delta_h$, and introduce the notation $\cipgen{ \ \cdot \ }{\ \cdot \ }{\mathcal{L}_{\mathcal{D}}^{-1}} =: \cipgen{ \ \cdot \ }{\ \cdot \ }{-1,h}$. In the general setting, the norm associated to this inner product is denoted $\nrm{\phi}_{\mathcal{L}_{\mathcal{D}}^{-1}} := \sqrt{\cipgen{\phi }{ \phi }{\mathcal{L}_{\mathcal{D}}^{-1}}}$, for all $\phi \in \mathring{\mathcal C}_{\rm per}$, but, if $\mathcal{D}\equiv 1$, we write $\nrm{\, \cdot \, }_{\mathcal{L}_{\mathcal{D}}^{-1}} =: \nrm{\, \cdot \, }_{-1,h}$.

\subsection{A convex-concave decomposition of the discrete energy}

In this section, we will recall a convex-concave decomposition of the energy~\eqref{continuous energy1}. The detailed proof of the following preliminary and lemma results could be found in the work \cite{Dong2020b}.

Define $\kappa(\phi): = \frac{1}{36\phi}$. The discrete energy $G_h(\phi_1,\phi_2):\vec{\mathcal{C}}_{\rm per}^{\mathcal{G}}\rightarrow \mathbb{R}$ is introduced as
\begin{align}
	G_h(\phi_1,\phi_2) & =\langle S(\phi_1,\phi_2) + H(\phi_1,\phi_2),1\rangle
	\nonumber
	\\
	& \quad +\langle a_x(\kappa(A_x\phi_1)(D_x\phi_1)^2)+a_y(\kappa(A_y\phi_1)(D_y\phi_1)^2),\varepsilon_1^2\rangle
	\nonumber
	\\
	&\quad +\langle a_x(\kappa(A_x\phi_2)(D_x\phi_2)^2)+a_y(\kappa(A_y\phi_2)(D_y\phi_2)^2),\varepsilon_2^2\rangle
	\nonumber
	\\
	& \quad +\langle a_x(\kappa(A_x (1-\phi_1-\phi_2))(D_x(1-\phi_1-\phi_2))^2),\varepsilon_3^2\rangle
	\nonumber
	\\
	& \quad +\langle a_y(\kappa(A_y(1-\phi_1-\phi_2))(D_y(1-\phi_1-\phi_2))^2),\varepsilon_3^2\rangle.
	\label{Full-discrete-energy}
\end{align}
\begin{lem}[Existence of a convex-concave decomposition]
	\label{Full-discrete-energy-splitting}
	Suppose $(\phi_1,\phi_2) \in \vec{\mathcal{C}}_{\rm per}^{\mathcal{G}}$. The functions
	\begin{align}
		G_{h,c}(\phi_1,\phi_2) &: =\langle S(\phi_1,\phi_2),1\rangle
		\label{Discrete-energy-c}
		\\
		&  \quad + \langle a_x(\kappa(A_x\phi_1)(D_x\phi_1)^2)+a_y(\kappa(A_y\phi_1)(D_y\phi_1)^2),\varepsilon_1^2\rangle
		\nonumber
		\\
		& \quad + \langle a_x(\kappa(A_x\phi_2)(D_x\phi_2)^2)+a_y(\kappa(A_y\phi_2)(D_y\phi_2)^2),\varepsilon_2^2\rangle
		\nonumber
		\\
		& \quad +\langle a_x(\kappa(A_x(1-\phi_1-\phi_2))(D_x(1-\phi_1-\phi_2))^2),\varepsilon_3^2\rangle
		\nonumber
		\\
		& \quad +\langle a_y(\kappa(A_y(1-\phi_1-\phi_2))(D_y(1-\phi_1-\phi_2))^2),\varepsilon_3^2\rangle,	
		\nonumber
		\\
		G_{h,e}(\phi_1,\phi_2) &:= -\langle  H(\phi_1,\phi_2),1\rangle ,
		\label{Discrete-energy-e}
	\end{align}
  where $G_{h,c}$ and $G_{h,e}$ are linear combination of certain convex functions. Therefore, $G_h(\phi_1,\phi_2)=G_{h,c}(\phi_1,\phi_2)-G_{h,e}(\phi_1,\phi_2)$ is a convex-concave decomposition of the discrete energy.
\end{lem}
	\begin{proposition}
	Suppose $(\phi_1,\phi_2) \in \vec{\mathcal{C}}_{\rm per}^{\mathcal{G}}$. The variational derivatives of $G_{h,c}$ and $G_{h,e}$ with respect to $\phi_1$ and $\phi_2$ are grid functions satisfying	
	\begin{align}
		\delta_{\phi_i}G_{h,c}(\phi_1,\phi_2) & =  \frac{\partial}{\partial \phi_i} S(\phi_1,\phi_2)
		\nonumber
		\\
		& \quad +\varepsilon_i^2 a_x(\kappa'(A_x\phi_i)(D_x\phi_i)^2)-2\varepsilon_i^2 d_x(\kappa(A_x\phi_i) D_x\phi_i)
		\nonumber
		\\
		& \quad + \varepsilon_i^2 a_y(\kappa'(A_y\phi_i)(D_y\phi_i)^2)-2\varepsilon_i^2 d_y(\kappa(A_y\phi_i) D_y\phi_i)
		\nonumber
		\\
		& \quad - \varepsilon_3^2 a_x( \kappa'(A_x(1-\phi_1-\phi_2))(D_x(1-\phi_1-\phi_2))^2)
		\nonumber
		\\
		& \quad + 2\varepsilon_3^2 d_x(\kappa(A_x(1-\phi_1-\phi_2)) D_x(1-\phi_1-\phi_2) )
		\nonumber
		\\
		& \quad - \varepsilon_3^2 a_y( \kappa'(A_y(1-\phi_1-\phi_2))(D_y(1-\phi_1-\phi_2))^2)
		\nonumber
		\\
		& \quad  + 2\varepsilon_3^2 d_y(\kappa(A_y(1-\phi_1-\phi_2)) D_y(1-\phi_1-\phi_2) ),
		\nonumber
		\\
		\delta_{\phi_i}G_{h,e}(\phi_1,\phi_2) & = - \frac{\partial}{\partial \phi_i} H(\phi_1,\phi_2),
		\nonumber
	\end{align}
	for $i = 1,2$.
\end{proposition}
	\begin{lem}\label{lemma2.2}
	Suppose that $\vec{\phi},\vec{\psi}\in \vec{\mathcal{C}}_{\rm per}^{\mathcal{G}}$. Consider the canonical convex splitting of the energy $G_h(\vec{\phi})$ in~\eqref{Full-discrete-energy} into $G_h=G_{h,c}-G_{h,e}$ given by~\eqref{Discrete-energy-c}-\eqref{Discrete-energy-e}. The following inequality is available	
	\begin{align}
		G_h(\vec{\phi})-G_h(\vec{\psi}) &\le \langle \delta_{\phi_1} G_{h,c}(\vec{\phi})-\delta_{\phi_1} G_{h,e}(\vec{\psi}),\phi_1-\psi_1\rangle
		\nonumber
		\\
		& \quad +\langle \delta_{\phi_2} G_{h,c}(\vec{\phi})-\delta_{\phi_2} G_{h,e}(\vec{\psi}),\phi_2-\psi_2\rangle.
		\nonumber
	\end{align}
\end{lem}
Using the idea of the convex splitting and the backward differentiation formula, we consider the   following semi-implicit, fully discrete scheme: for $n\geq 1$, given $(\phi_1^n, \phi_2^n) \in \vec{\mathcal{C}}_{\rm per}^{\mathcal{G}}$, $(\phi_1^{n-1}, \phi_2^{n-1}) \in \vec{\mathcal{C}}_{\rm per}^{\mathcal{G}}$, find $(\phi_1^{n+1},\phi_2^{n+1}) \in \vec{\mathcal{C}}_{\rm per}^{\mathcal{G}}$ such that
\begin{small}
\begin{align}
	&\frac{3\phi_1^{n+1}-4\phi_1^n+\phi_1^{n-1}}{2\dt}  =  \mathcal{M}_1 \Dh \mu_1^{n+1}
	\label{Full-discrete-1},
	\\
	&\mu_1^{n+1} : =\delta_{\phi_1}G_{h,c}(\phi_1^{n+1},\phi_2^{n+1})- 
	\delta_{\phi_1}G_{h,e}(\hat{\phi}_1^n,\hat{\phi}_2^n) - A_1\dt\Dh(\phi_1^{n+1}-\phi_1^n)\nonumber
	\\
 &\quad \quad ~=  \frac{1}{M_0} \ln \frac{\alpha \phi_1^{n+1}}{M_0} - \ln(1-\phi_1^{n+1} -\phi_2^{n+1}) - 2 \chi_{13} \hat{\phi}_1^n +(\chi_{12}-\chi_{13}-\chi_{23}) \hat{\phi}_2^n  \nonumber
\\
& \quad \quad \quad ~
- \frac{\varepsilon_1^2}{36} {\cal A}_h \Bigl( \frac{| \nabla_h \phi_1^{n+1} |^2}{ ( {\cal A}_h \phi_1^{n+1} )^2}  \Bigr)
- \frac{\varepsilon_1^2}{18} \nabla_h \cdot \Bigl( \frac{\nabla_h \phi_1^{n+1}}{{\cal A}_h \phi_1^{n+1} } \Bigr) - A_1\dt\Dh(\phi_1^{n+1}-\phi_1^n)\nonumber
\\
& \quad \quad \quad ~
+ \frac{\varepsilon_3^2}{36} {\cal A}_h \Bigl( \frac{| \nabla_h ( 1 - \phi_1^{n+1} - \phi_2^{n+1}) |^2}{ ( {\cal A}_h (1 - \phi_1^{n+1} - \phi_2^{n+1} ) )^2}  \Bigr)   
+ \frac{\varepsilon_3^2}{18} \nabla_h \cdot \Bigl( \frac{\nabla_h (1 - \phi_1^{n+1} - \phi_2^{n+1}) }{{\cal A}_h (1 - \phi_1^{n+1} - \phi_2^{n+1}) } \Bigr)  
	\label{Full-discrete-mu1} ,
	\\
&\frac{3\phi_2^{n+1}-4\phi_2^n+\phi_2^{n-1}}{2\dt}  =  \mathcal{M}_2 \Dh \mu_2^{n+1},
	\label{Full-discrete-2}
	\\
	&\mu_2^{n+1} :=\delta_{\phi_2}G_{h,c}(\phi_1^{n+1},\phi_2^{n+1})- \delta_{\phi_2}G_{h,e}(\hat{\phi}_1^n,\hat{\phi}_2^n)- A_2\dt\Dh(\phi_2^{n+1}-\phi_2^n)\nonumber
	\\
	& \quad \quad ~= \frac{1}{N_0} \ln \frac{\beta \phi_2^{n+1}}{N_0} - \ln(1-\phi_1^{n+1} - \phi_2^{n+1}) - 2 \chi_{23} \hat{\phi}_2^n +(\chi_{12}-\chi_{13}-\chi_{23})\hat{\phi}_1^n  \nonumber
	\\
	&\quad \quad \quad ~
	- \frac{\varepsilon_2^2}{36} {\cal A}_h \Bigl( \frac{| \nabla_h \phi_2^{n+1} |^2}{ ( {\cal A}_h \phi_2^{n+1} )^2}  \Bigr)
	- \frac{\varepsilon_2^2}{18} \nabla_h \cdot \Bigl( \frac{\nabla_h \phi_2^{n+1}}{{\cal A}_h \phi_2^{n+1} } \Bigr) - A_2\dt\Dh(\phi_2^{n+1}-\phi_2^n)\nonumber
	\\
	&\quad \quad \quad ~
	+ \frac{\varepsilon_3^2}{36} {\cal A}_h \Bigl( \frac{| \nabla_h ( 1 - \phi_1^{n+1} - \phi_2^{n+1}) |^2}{ ( {\cal A}_h (1 - \phi_1^{n+1} - \phi_2^{n+1} ) )^2}  \Bigr)   
	+ \frac{\varepsilon_3^2}{18} \nabla_h \cdot \Bigl( \frac{\nabla_h (1 - \phi_1^{n+1} - \phi_2^{n+1}) }{{\cal A}_h (1 - \phi_1^{n+1} - \phi_2^{n+1}) } \Bigr) ,	\label{Full-discrete-mu2}
\end{align}
\end{small}
where
\begin{align*}
	{\cal A}_h \left( \frac{| \nabla_h u |^2}{ ( {\cal A}_h u )^2}  \right)
	& :=  a_x \left( \frac{| D_x u |^2 }{ (A_x u)^2} \right)
	+ a_y \left( \frac{| D_y u |^2}{ (A_y u)^2}  \right)  ,
	\\
	\nabla_h \cdot \left( \frac{\nabla_h u}{{\cal A}_h u }\right)
	& := d_x \left( \frac{D_x u}{A_x u} \right) + d_y \left( \frac{D_y u}{A_y u} \right)  ,
\end{align*}
for all $u\in\mathcal{C}_{\rm per}$, provided $u$ does not vanish at any grid points. And $\hat{u}^n:=2u^n-u^{n-1}$.

The initialization step comes from a combination of convex splitting and a second order numerical correction:
\begin{small}
     \begin{align}
		&\frac{\phi_1^1-\phi_1^0}{\dt}  =  \mathcal{M}_1 \Dh \mu_1^1
		\nonumber,
		\\
		&\mu_1^1 : =\frac{1}{2}\left(\delta_{\phi_1}G_{h,c}(\phi_1^1,\phi_2^1)+
		\delta_{\phi_1}G_{h,c}(\phi_1^0,\phi_2^0)\right) + \frac{\partial}{\partial \phi_1} H(\phi_1^0,\phi_2^0) + \frac{\dt}{2}\frac{\partial^2}{\partial \phi_1^2} H(\phi_1^0,\phi_2^0)(\phi_1)^0_t\nonumber,
		\\
		&\frac{\phi_2^1-\phi_2^0}{\dt}  =  \mathcal{M}_2 \Dh \mu_2^1,
		\nonumber
		\\
		&\mu_2^1 :=\frac{1}{2}\left(\delta_{\phi_2}G_{h,c}(\phi_1^1,\phi_2^1)+ \delta_{\phi_2}G_{h,c}(\phi_1^0,\phi_2^0)\right) + \frac{\partial}{\partial \phi_2} H(\phi_1^0,\phi_2^0)+ \frac{\dt}{2}\frac{\partial^2}{\partial \phi_2^2} H(\phi_1^0,\phi_2^0)(\phi_2)^0_t.\label{scheme-CH_initial}
	\end{align}
\end{small}
The local truncation error of this initialization step is second order, which matches the overall second-order accuracy of the scheme and is consistent with the high order consistency analysis, as will be shown in later sections. In addition, this initialization step method satisfies the positivity-preserving property and energy stability. 

\begin{rem} 
The construction of a second order accurate, positivity-preserving and energy stable numerical scheme for the ternary MMC system turns out to be more challenging than the first order accurate algorithm~\cite{Dong2020b, dong2022}. Because of the complicated structure of the nonlinear and singular surface diffusion energy, as well as its functional derivatives, a Crank-Nicolson style approximation could hardly ensure both the positivity-preserving and energy stability properties. In turn, such a numerical effort has to be focused on the BDF style approach. With the BDF2 approximation, the nonlinear and singular terms could be treated in a similar manner as in the first order numerical method, while the computation of the concave and expansive terms becomes more tricky. Because of the negative eigenvalues in the concave expansive terms, an explicit treatment is necessary for the sake of both the unique solvability and energy stability. In the first order numerical method, an explicit treatment to the concave terms is able to ensure a dissipation of the associated energy; however, a direct application of second order Adams-Bashforth extrapolation for the concave terms would not enforce such an energy stability at a theoretical level. To remedy this numerical effort, we have to add artificial regularization terms, for both $\phi_1$ and $\phi_2$, to establish such a theoretical analysis of energy stability, as will be demonstrated in the later section. Moreover, since a multi-step approach is applied in the second order accurate scheme, the initialization step turns out to be more challenging, and a careful computation in the initial step, as given by~\eqref{scheme-CH_initial}, is needed to ensure the theoretical properties at the initial time step. 
\end{rem} 

	\section{Unique solvability and positivity-preserving property}
\label{sec:positivity-preserving property}
\setcounter{equation}{0}

The proof of the following lemma can be found in~\cite{chen19b}.
\begin{lem}
	\label{MMC-positivity-Lem-0}  
	Suppose that $\phi_1$, $\phi_2 \in \mathcal{C}_{\rm per}$, with $\langle\phi_1 - \phi_2,1\rangle = 0$, that is, $\phi_1 - \phi_2\in \mathring{\mathcal{C}}_{\rm per}$, and assume that $\nrm{\phi_1}_\infty < 1$, $\nrm{\phi_2}_\infty \le M$. Then,   we have the following estimate:
	\[
	\nrm{(-\Delta_h)^{-1} (\phi_1 - \phi_2)}_\infty \le C_1 ,
	\]
	where $C_1>0$ depends only upon $M$ and $\Omega$. In particular, $C_1$ is independent of the mesh size $h$.
\end{lem}

In fact, in the ternary MMC model, all the phase variables have to stay within $(0, 1)$, due to the positivity-preserving property, i.e., $0 < \phi_1 , \, \phi_2 , \, 1 - \phi_1 - \phi_2 < 1$, at a point-wise level. Therefore, we could take $M=1$ to justify an application of this lemma, and appropriate functional space could be set to enforce such a point-wise bound. The following theorem is the main result of this section.
	\begin{theorem}
	\label{MMC-positivity}
	Given $(\phi_1^k,\phi_2^k)\in\vec{\mathcal{C}}_{\rm per}^{\mathcal{G}},\, k=n-1,n$, and $(\overline{\phi_1^n}, \overline{\phi_2^n}), (\overline{\phi_1^{n-1}}, \overline{\phi_2^{n-1}})\in\mathcal{G}$, then there exists a unique
	solution $(\phi_1^{n+1},\phi_2^{n+1})\in\vec{\mathcal{C}}_{\rm per}^{\mathcal{G}}$ to
	\eqref{Full-discrete-1}-\eqref{Full-discrete-mu2}, with $\overline{\phi_1^n} =
	\overline{\phi_1^{n+1}}$ and $\overline{\phi_2^n} = \overline{\phi_2^{n+1}}$.
\end{theorem}
\subsection{The equivalent form of solving \eqref{Full-discrete-1}-\eqref{Full-discrete-mu2}}

For bookkeeping, we introduce the following notation:
\[
\delta_{\phi_1}G_{h,c}(\phi_1,\phi_2) = \sum_{\ell = 1}^9 Q_\ell(\phi_1,\phi_2),	
\]
where
\begin{align*}
	Q_1(\phi_1,\phi_2) & := \frac{\partial}{\partial \phi_1} S(\phi_1,\phi_2),
	\\
	Q_2(\phi_1,\phi_2) &:= \varepsilon_1^2 a_x(\kappa'(A_x\phi_1)(D_x\phi_1)^2),
	\\
	Q_3(\phi_1,\phi_2) &:= -2\varepsilon_1^2 d_x(\kappa(A_x\phi_1) D_x\phi_1),
	\\
	Q_4(\phi_1,\phi_2) & := \varepsilon_1^2 a_y(\kappa'(A_y\phi_1)(D_y\phi_1)^2),
	\\
	Q_5(\phi_1,\phi_2) & := -2\varepsilon_1^2 d_y(\kappa(A_y\phi_1) D_y\phi_1) ,
	\\
	Q_6(\phi_1,\phi_2) & :=-\varepsilon_3^2 a_x(\kappa'(A_x(1-\phi_1-\phi_2))(D_x(1-\phi_1-\phi_2))^2),
	\\
	Q_7(\phi_1,\phi_2) & :=  2\varepsilon_3^2 d_x(\kappa(A_x(1-\phi_1-\phi_2)) D_x(1-\phi_1-\phi_2) ),
	\\
	Q_8(\phi_1,\phi_2) & := -\varepsilon_3^2 a_y(\kappa'(A_y(1-\phi_1-\phi_2))(D_y(1-\phi_1-\phi_2))^2) ,
	\\
	Q_9(\phi_1,\phi_2) & := 2\varepsilon_3^2 d_y(\kappa(A_y(1-\phi_1-\phi_2)) D_y(1-\phi_1-\phi_2) ).
\end{align*}
The numerical solution of \eqref{Full-discrete-1}-\eqref{Full-discrete-mu2} is a minimizer of the following discrete energy functional:
\begin{align*}
	\mathcal{J}_h^n(\phi_1,\phi_2) & = \frac{1}{12\mathcal{M}_1\dt}\|3\phi_1-4\phi_1^n+\phi_1^{n-1}\|_{-1,h}^2+\frac{1}{12\mathcal{M}_2\dt} \|3\phi_2-4\phi_2^n+\phi_2^{n-1}\|_{-1,h}^2		\\
	& \quad+\langle S(\phi_1,\phi_2) ,1\rangle
	+\langle a_x(\kappa(A_x\phi_1)(D_x\phi_1)^2) +a_y(\kappa(A_y\phi_1)(D_y\phi_1)^2),\varepsilon_1^2\rangle
	\\
	& \quad  + \langle a_x(\kappa(A_x\phi_2)(D_x\phi_2)^2) +a_y(\kappa(A_y\phi_2)(D_y\phi_2)^2),\varepsilon_2^2\rangle
	\\
	& \quad  + \langle a_x(\kappa(A_x(1-\phi_1-\phi_2))(D_x(1-\phi_1-\phi_2))^2)
	\\
	& \quad\quad +a_y(\kappa(A_y(1-\phi_1-\phi_2))(D_y(1-\phi_1-\phi_2))^2) , \varepsilon_3^2 \rangle
	\\
	& \quad +\langle \frac{\partial}{\partial \phi_1} H(\hat{\phi}_1^n,\hat{\phi}_2^n),\phi_1\rangle +\langle \frac{\partial}{\partial \phi_2} H(\hat{\phi}_1^n,\hat{\phi}_2^n),\phi_2\rangle
	\\
	& \quad
	+\frac{A_1\dt}{2}\|\nabla_h(\phi_1-\phi_1^n)\|_2^2+\frac{A_2\dt}{2}\|\nabla_h(\phi_2-\phi_2^n)\|_2^2,
\end{align*}
over the admissible set
\[
A_h := \left\{ (\phi_1,\phi_2) \in \vec{\mathcal{C}}_{\rm per}^{\mathcal{G}} \ \middle| \  \langle\phi_1,1\rangle = |\Omega| \overline{\phi_1^0}, \quad \langle\phi_2,1\rangle= |\Omega| \overline{\phi_2^0} \right\} \subset \mathbb{R}^{2N^2}.
\]
It is clear that $\mathcal{J}_h^n$ is a strictly convex functional.
 
\subsection{Proof by contradiction}

Now, consider the following closed domain:
\begin{align*}
	A_{h,\delta} & := \Bigl\{  (\phi_1,\phi_2) \in \mathcal{C}_{\rm per}\times \mathcal{C}_{\rm per} \ \Big| \ \phi_1,\phi_2 \ge g(\delta),~ \delta \le \phi_1+\phi_2 \le 1-\delta, \Bigr.
	\\
	& \Bigl. \hspace{1.75in} \langle\phi_1,1\rangle = |\Omega| \overline{\phi_1^0}, ~  \langle\phi_2,1\rangle= |\Omega| \overline{\phi_2^0} \Bigr\} \subset \mathbb{R}^{2N^2} ,
\end{align*}
where $g (\delta) >0$ will be given later.
Define the hyperplane
\begin{equation*}
	V := \left\{ (\phi_1,\phi_2) \ \middle| \  \langle\phi_1,1\rangle = |\Omega| \overline{\phi_1^0}, ~  \langle\phi_2,1\rangle= |\Omega| \overline{\phi_2^0}\right\} \subset \mathbb{R}^{2N^2}.
\end{equation*}
Since $A_{h,\delta}$ is a bounded, compact, and convex subset of $V$, there exists (not necessarily unique) a minimizer of $\mathcal{J}_h^n(\phi_1,\phi_2)$ over $A_{h,\delta}$. The key point of the positivity analysis is that, such a minimizer could not occur at a boundary point of $A_{h,\delta}$, if $\delta$ and $g(\delta)$ are sufficiently small.
Assume the minimizer of $\mathcal{J}_h^n(\phi_1,\phi_2)$ over $A_{h,\delta}$ occurs at a boundary point of $A_{h,\delta}$.
\subsubsection{The minimizer $(\phi_1^{\star},\phi_2^{\star})\in A_{h,\delta}$ could not occur at $\phi_1^{\star}=g(\delta)$.}\label{subsubsection3.2.1}

We suppose the minimizer $(\phi_1^{\star},\phi_2^{\star})\in A_{h,\delta}$, satisfies $(\phi_1^{\star})_{\vec{\alpha}_0}=g(\delta)$, for some grid point $\vec{\alpha}_0:=(i_0,j_0)$. Assume that $\phi_1^{\star}$ reaches its maximum value at the grid point $\vec{\alpha}_1:=(i_1,j_1)$.  It is obvious that $(\phi_1^{\star})_{\vec{\alpha}_1}\geq \overline{\phi_1^{\star}}=\overline{\phi_1^0}$. 
A careful calculation gives the following directional derivative
	\begin{align*}
		&d_s \mathcal{J}_h^n(\phi_1^\star+s\psi,\phi_2^\star)|_{s=0} \nonumber\\
		=&\frac{1}{2\mathcal{M}_1\Delta
			t}\langle (-\Delta_h)^{-1}\left(3\phi_1^\star-4\phi_1^n+\phi_1^{n-1} \right),\psi\rangle - A_1\Delta t\langle \Delta_h(\phi_1^\star-\phi_1^n),\psi\rangle
		\\
		&\quad+ \langle \delta_{\phi_1}G_{h,c}(\phi_1^{\star},\phi_2^{\star}),\psi\rangle+
		\langle \frac{\partial}{\partial \phi_1}H(\hat{\phi}_1^n,\hat{\phi}_2^n),\psi \rangle,
\end{align*}
for any $\psi \in \mathring{\mathcal{C}}_{\rm per }$. Let us pick the direction
\[
\psi_{i,j} = \delta_{i,i_0}\delta_{j,j_0} - \delta_{i,i_1}\delta_{j,j_1},
\]
where $\delta_{i,j}$ is the Dirac delta function. Note that $\psi$ is of mean zero. The derivative may be expressed as
\begin{align}
	\frac{1}{h^2}d_s \mathcal{J}_h^n(\phi_1^\star+s\psi,\phi_2^\star)|_{s=0}	
	& = \frac{1}{2\mathcal{M}_1\Delta
		t}(-\Delta_h)^{-1}\left(3\phi_1^\star-4\phi_1^n+\phi_1^{n-1} \right)_{\vec{\alpha}_0}
	\label{MMC-Positive}
	\\
	& \quad  -
	\frac{1}{2\mathcal{M}_1\Delta
		t}(-\Delta_h)^{-1}\left(3\phi_1^\star-4\phi_1^n+\phi_1^{n-1} \right)_{\vec{\alpha}_1}
	\nonumber
	\\
	& \quad + Q_\ell(\phi_1^{\star},\phi_2^{\star})_{\vec{\alpha}_0}
	- Q_\ell(\phi_1^{\star},\phi_2^{\star})_{\vec{\alpha}_1}
	\nonumber
	\\
	& \quad +\frac{\partial}{\partial \phi_1}H(\hat{\phi}_1^n,\hat{\phi}_2^n)|_{\vec{\alpha}_0}
	-\frac{\partial}{\partial \phi_1}H(\hat{\phi}_1^n,\hat{\phi}_2^n)|_{\vec{\alpha}_1}
	\nonumber
	\\
	& \quad - A_1\Delta t\left(\Delta_h(\phi_1^\star-\phi_1^n)_{\vec{\alpha}_0}-\Delta_h(\phi_1^\star-\phi_1^n)_{\vec{\alpha}_1}\right).
	\nonumber
\end{align}
For the first and second terms appearing in the right hand side of \eqref{MMC-Positive}, we apply
Lemma \ref{MMC-positivity-Lem-0} and obtain
\begin{small}
	\begin{equation}
		- \frac{8C_1}{\mathcal{M}_1} \le
		\frac{1}{\mathcal{M}_1}(-\Delta_h)^{-1}\left(3\phi_1^\star-4\phi_1^n+\phi_1^{n-1} \right)_{\vec{\alpha}_0}
		-\frac{1}{\mathcal{M}_1}(-\Delta_h)^{-1}\left(3\phi_1^\star-4\phi_1^n+\phi_1^{n-1} \right)_{\vec{\alpha}_1}
		\le  \frac{8C_1}{\mathcal{M}_1} .
		\label{MMC-positive-1}
	\end{equation}
\end{small}
For the $Q_1$ terms, the following inequality is available:
	\begin{align}
		&Q_1(\phi_1^{\star},\phi_2^{\star})_{\vec{\alpha}_0} - Q_1(\phi_1^{\star},\phi_2^{\star})_{\vec{\alpha}_1} \nonumber\\
		=&  \frac{\partial}{\partial \phi_1} S(\phi_1^{\star},\phi_2^{\star})|_{\vec{\alpha}_0} - \frac{\partial}{\partial \phi_1} S(\phi_1^{\star},\phi_2^{\star})|_{\vec{\alpha}_1}
		\nonumber
		\\
		=&  \left(\frac{1}{M_0}\ln \frac{\alpha \phi_1^{\star}}{M_0}-  \ln(1-\phi_1^{\star}-\phi_2^{\star})\right)|_{\vec{\alpha}_0}	
		-  \left(\frac{1}{M_0}\ln \frac{\alpha \phi_1^{\star}}{M_0} -
		\ln(1-\phi_1^{\star}-\phi_2^{\star})\right)|_{\vec{\alpha}_1}
		\nonumber
		\\
		=&   \left( \ln
		\frac{(\phi_1^{\star})^{\nicefrac{1}{M_0}}}{1-\phi_1^{\star}-\phi_2^{\star}}\right)|_{\vec{\alpha}_0} - \left( \ln
		\frac{(\phi_1^{\star})^{\nicefrac{1}{M_0}}}{1-\phi_1^{\star}-\phi_2^{\star}}\right)|_{\vec{\alpha}_1}
		\nonumber
		\\
		\leq&    \ln \frac{(g(\delta))^{\nicefrac{1}{M_0}}}{\delta} -
		\ln \frac{(\overline{\phi_1^0})^{\nicefrac{1}{M_0}}}{1-\delta}
		\nonumber
		\\
		\leq&  \ln \frac{(g(\delta))^{\nicefrac{1}{M_0}}}{\delta} -
		\frac{1}{M_0}\ln \overline{\phi_1^0}.\label{MMC-positive-2}
\end{align}
Using the logarithm property $\ln(ab) = \ln a + \ln b$, we have eliminated the constant $\frac{1}{M_0}\ln \frac{\alpha}{M_0}$. The next-to-last step comes from the facts that $(\phi_1^{\star})_{\vec{\alpha}_0}=g(\delta)$, $(\phi_1^{\star})_{\vec{\alpha}_1}\geq \overline{\phi_1^0}$ and $\delta \le \phi_1+\phi_2 \le 1-\delta$. The last step comes from the inequality that $\ln (1-\delta) < 0$.
For the $Q_2$ terms, we have
\begin{align}
	&Q_2(\phi_1^{\star},\phi_2^{\star})_{\vec{\alpha}_0} - Q_2(\phi_1^{\star},\phi_2^{\star})_{\vec{\alpha}_1} \nonumber\\
	=&  \varepsilon_1^2 a_x(\kappa'(A_x\phi_1^{\star})(D_x\phi_1^{\star})^2)_{\vec{\alpha}_0}
	- \varepsilon_1^2 a_x(\kappa'(A_x\phi_1^{\star})(D_x\phi_1^{\star})^2)_{\vec{\alpha}_1}
	\nonumber
	\\
	\leq&  - \varepsilon_1^2 a_x(\kappa'(A_x\phi_1^{\star})(D_x\phi_1^{\star})^2)_{\vec{\alpha}_1}
	\nonumber
	\\
	\leq&  \frac{\varepsilon_1^2}{9h^2}.\label{MMC-positive-4}
\end{align}
The second step above comes from the fact that
\[
\varepsilon_1^2 a_x(\kappa'(A_x\phi_1^{\star})(D_x\phi_1^{\star})^2)_{\vec{\alpha}_0}\leq 0,
\]
since $\kappa'(\phi)= -\frac{1}{36\phi^2}<0$. The last step is based on the definitions of $\kappa'(\phi)$, $a_x$, $A_x$, and $D_x$,  as well as the fact that $|\frac{a-b}{a+b}|<1$, $\forall a>0, b>0$. In details, we observe the following expansion
\begin{align}
	- \varepsilon_1^2 a_x(\kappa'(A_x\phi_1^{\star})(D_x\phi_1^{\star})^2)_{\vec{\alpha}_1} & =
	\frac{\varepsilon_1^2}{18h^2}
	\left[\frac{(\phi_1^{\star})_{i_1+1,j_1}-(\phi_1^{\star})_{i_1,j_1}}{(\phi_1^{\star})_{i_1+1,j_1}+(\phi_1^{\star})_{i_1,j_1}}\right]^2
	\nonumber
	\\
	& \quad  + \frac{\varepsilon_1^2}{18h^2}
	\left[\frac{(\phi_1^{\star})_{i_1,j_1}-(\phi_1^{\star})_{i_1-1,j_1}}{(\phi_1^{\star})_{i_1-1,j_1}+(\phi_1^{\star})_{i_1,j_1}}\right]^2
	\nonumber
	\\
	& \leq \frac{\varepsilon_1^2}{9h^2} .
	\nonumber
\end{align}
The $Q_4$ terms can be similarly handled:
\begin{align}
	Q_4(\phi_1^{\star},\phi_2^{\star})_{\vec{\alpha}_0} - Q_4(\phi_1^{\star},\phi_2^{\star})_{\vec{\alpha}_1} & = \varepsilon_1^2 a_y(\kappa'(A_y\phi_1^{\star})(D_y\phi_1^{\star})^2)_{\vec{\alpha}_0}
	\label{MMC-positive-5}
	\\
	& \quad - \varepsilon_1^2 a_y(\kappa'(A_y\phi_1^{\star})(D_y\phi_1^{\star})^2)_{\vec{\alpha}_1}
	\nonumber
	\\
	& \leq \frac{\varepsilon_1^2}{9h^2} .
	\nonumber
\end{align}

For the $Q_3$ terms, we see that
\begin{align}
	Q_3(\phi_1^{\star},\phi_2^{\star})_{\vec{\alpha}_0} - Q_3(\phi_1^{\star},\phi_2^{\star})_{\vec{\alpha}_1}& = -2\varepsilon_1^2 d_x(\kappa(A_x\phi_1^{\star}) D_x\phi_1^{\star})_{\vec{\alpha}_0}
	\label{MMC-positive-6}
	\\
	& \quad + 2\varepsilon_1^2 d_x(\kappa(A_x\phi_1^{\star}) D_x\phi_1^{\star})_{\vec{\alpha}_1}
	\nonumber
	\\
	& \leq 0 ,
	\nonumber
\end{align}
in which the last step comes from the fact that $(D_x\phi_1^{\star})_{i_0-\nicefrac{1}{2},j_0} \leq 0, (D_x\phi_1^{\star})_{i_0+\nicefrac{1}{2},j_0} \geq 0, (D_x\phi_1^{\star})_{i_1-\nicefrac{1}{2},j_1} \geq 0$, and $(D_x\phi_1^{\star})_{i_1+\nicefrac{1}{2},j_1} \leq 0$. 

A bound for the $Q_5$ terms could be similarly derived:
\begin{align}
	Q_5(\phi_1^{\star},\phi_2^{\star})_{\vec{\alpha}_0} - Q_5(\phi_1^{\star},\phi_2^{\star})_{\vec{\alpha}_1} & = -2\varepsilon_1^2 d_y(\kappa(A_y\phi_1^{\star}) D_y\phi_1^{\star})_{\vec{\alpha}_0}
	\nonumber
	\\
	& \quad + 2\varepsilon_1^2 d_y(\kappa(A_y\phi_1^{\star}) D_y\phi_1^{\star})_{\vec{\alpha}_1}
	\nonumber
	\\
	& \leq 0 .\label{MMC-positive-7}
\end{align}

Use a technique similar to that used for $Q_2$, the $Q_6$ terms could be controlled as follows:
	\begin{align}
		&Q_6(\phi_1^{\star},\phi_2^{\star})_{\vec{\alpha}_0} - Q_6(\phi_1^{\star},\phi_2^{\star})_{\vec{\alpha}_1} \nonumber\\
		=&  -\varepsilon_3^2 a_x\left(\kappa'(A_x(1-\phi_1^{\star}-\phi_2^{\star}))(D_x(1-\phi_1^{\star}-\phi_2^{\star}))^2\right)_{\vec{\alpha}_0}
		\nonumber
		\\
		& + \varepsilon_3^2 a_x\left(\kappa'(A_x(1-\phi_1^{\star}-\phi_2^{\star}))(D_x(1-\phi_1^{\star}-\phi_2^{\star}))^2\right)_{\vec{\alpha}_1}
		\nonumber
		\\
		\leq&  -\varepsilon_3^2
		a_x\left(\kappa'(A_x(1-\phi_1^{\star}-\phi_2^{\star}))(D_x(1-\phi_1^{\star}-\phi_2^{\star}))^2\right)_{\vec{\alpha}_0}
		\nonumber
		\\
		\leq&   \frac{\varepsilon_3^2 }{9h^2} .\label{MMC-positive-8}
\end{align}
A similar inequality could be derived for the $Q_8$ terms:
\begin{align}
	&Q_8(\phi_1^{\star},\phi_2^{\star})_{\vec{\alpha}_0} - Q_8(\phi_1^{\star},\phi_2^{\star})_{\vec{\alpha}_1} \nonumber\\
	=& -\varepsilon_3^2
	a_y\left(\kappa'(A_y(1-\phi_1^{\star}-\phi_2^{\star}))(D_y(1-\phi_1^{\star}-\phi_2^{\star}))^2\right)_{\vec{\alpha}_0}
	\nonumber
	\\
	& + \varepsilon_3^2
	a_y\left(\kappa'(A_y(1-\phi_1^{\star}-\phi_2^{\star}))(D_y(1-\phi_1^{\star}-\phi_2^{\star}))^2\right)_{\vec{\alpha}_1}
	\nonumber
	\\
	\leq&  -\varepsilon_3^2
	a_y\left(\kappa'(A_y(1-\phi_1^{\star}-\phi_2^{\star}))(D_y(1-\phi_1^{\star}-\phi_2^{\star}))^2\right)_{\vec{\alpha}_0}
	\nonumber
	\\
	\leq&   \frac{\varepsilon_3^2}{9h^2} .\label{MMC-positive-9}
\end{align}

For the $Q_7$ terms, we have
	\begin{align}
		&Q_7(\phi_1^{\star},\phi_2^{\star})_{\vec{\alpha}_0} -
		Q_7(\phi_1^{\star},\phi_2^{\star})_{\vec{\alpha}_1} \nonumber\\
		=&  2 \varepsilon_3^2 d_x(\kappa(A_x(1-\phi_1^{\star}-\phi_2^{\star})) D_x(1-\phi_1^{\star}-\phi_2^{\star}) )_{\vec{\alpha}_0}
		\nonumber
		\\
		&  - 2 \varepsilon_3^2
		d_x(\kappa(A_x(1-\phi_1^{\star}-\phi_2^{\star})) D_x(1-\phi_1^{\star}-\phi_2^{\star}) )_{\vec{\alpha}_1} 	
		\nonumber
		\\
		=&  \frac{\varepsilon_3^2}{18h}
		\left(\frac{D_x(1-\phi_1^{\star}-\phi_2^{\star})}{A_x(1-\phi_1^{\star}-\phi_2^{\star})}\right)_{i_0+\nicefrac{1}{2},j_0}
		-  \frac{\varepsilon_3^2}{18h}
		\left(\frac{D_x(1-\phi_1^{\star}-\phi_2^{\star})}
		{A_x(1-\phi_1^{\star}-\phi_2^{\star})}\right)_{i_0-\nicefrac{1}{2},j_0}
		\nonumber
		\\
		&  - \frac{\varepsilon_3^2}{18h}
		\left(\frac{D_x(1-\phi_1^{\star}-\phi_2^{\star})}{A_x(1-\phi_1^{\star}-\phi_2^{\star})}\right)_{i_1+\nicefrac{1}{2},j_1}
		+ \frac{\varepsilon_3^2}{18h}
		\left(\frac{D_x(1-\phi_1^{\star}-\phi_2^{\star})}{A_x(1-\phi_1^{\star}-\phi_2^{\star})}\right)_{i_1-\nicefrac{1}{2},j_1}
		\nonumber
		\\
		\leq&  \frac{4 \varepsilon_3^2 }{9h^2} .\label{MMC-positive-10}
\end{align}
The last step above is based on the definitions of $A_x$ and $D_x$,  as well as the fact that $|\frac{a-b}{a+b}|<1$, $\forall a>0, b>0$.

Similarly, for the $Q_9$ terms, we have
\begin{align}
	Q_9(\phi_1^{\star},\phi_2^{\star})_{\vec{\alpha}_0} - Q_9(\phi_1^{\star},\phi_2^{\star})_{\vec{\alpha}_1} & = 2 \varepsilon_3^2
	d_y(\kappa(A_y(1-\phi_1^{\star}-\phi_2^{\star}))D_y(1-\phi_1^{\star}-\phi_2^{\star}))_{\vec{\alpha}_0}
	\label{MMC-positive-11}
	\\
	& \quad - 2 \varepsilon_3^2
	d_y(\kappa(A_y(1-\phi_1^{\star}-\phi_2^{\star}))D_y(1-\phi_1^{\star}-\phi_2^{\star}))_{\vec{\alpha}_1}
	\nonumber
	\\
	& \leq \frac{4 \varepsilon_3^2 }{9 h^2}.
	\nonumber
\end{align}
For the numerical solution $\phi_i^n,\, \phi_i^{n-1}$ at the previous time step, the a-priori assumption $0<\phi_i^n<1$ indicates that
\begin{equation}
	-1 \leq (\phi_i^n)_{\vec{\alpha}_0}-(\phi_i^n)_{\vec{\alpha}_1}\leq 1,\, 	-1 \leq (\phi_i^{n-1})_{\vec{\alpha}_0}-(\phi_i^{n-1})_{\vec{\alpha}_1}\leq 1,\quad i = 1,2,
\end{equation}
then, we have
\begin{equation}
	-3 \leq (\hat{\phi}_i^n)_{\vec{\alpha}_0}-(\hat{\phi}_i^n)_{\vec{\alpha}_1}\leq 3,\quad i = 1,2.
\end{equation}
For the fifth and sixth terms appearing in \eqref{MMC-Positive}, we see that
\begin{align}
	&\frac{\partial}{\partial \phi_1}H(\hat{\phi}_1^n,\hat{\phi}_2^n)|_{\vec{\alpha}_0} -
	\frac{\partial}{\partial \phi_1}H(\hat{\phi}_1^n,\hat{\phi}_2^n)|_{\vec{\alpha}_1} \nonumber\\
	=&  - 2 \chi_{13}
	[(\hat{\phi}_1^n)_{\vec{\alpha}_0}-(\hat{\phi}_1^n)_{\vec{\alpha}_1}]
	+ (\chi_{12}-\chi_{13}-\chi_{23})[(\hat{\phi}_2^n)_{\vec{\alpha}_0}-(\hat{\phi}_2^n)_{\vec{\alpha}_1}]
	\nonumber
	\\
	\leq&  3(\chi_{12} + 3 \chi_{13}+\chi_{23}) .\label{MMC-positive-3}
\end{align}
For the last term appearing in the right hand of \eqref{MMC-Positive}, we see that
\begin{equation}
	\Delta_h(\phi_1^\star)_{\vec{\alpha}_0}-\Delta_h(\phi_1^\star)_{\vec{\alpha}_1}\ge 0,\,
	-\frac{8}{h^2}\le\Delta_h(\phi_1^n)_{\vec{\alpha}_0}-\Delta_h(\phi_1^n)_{\vec{\alpha}_1}\le \frac{8}{h^2},
\end{equation}
this means 
\begin{equation}
	- A_1\Delta t\left(\Delta_h(\phi_1^\star-\phi_1^n)_{\vec{\alpha}_0}-\Delta_h(\phi_1^\star-\phi_1^n)_{\vec{\alpha}_1}\right)	\le \frac{8A_1\Delta t}{h^2}.	\label{MMC-positive-04}
\end{equation}
Putting every terms together, we have
\begin{align}
	\frac{1}{h^2}d_s \mathcal{J}_h^n(\phi_1^\star+s\psi,\phi_2^\star)|_{s=0} & \leq \ln \frac{(g(\delta))^{\nicefrac{1}{M_0}}}{\delta} - \frac{1}{M_0}\ln\overline{\phi_1^0}+\frac{4C_1}{\mathcal{M}_1\dt}
	\nonumber
	\\
	& \quad + \frac{2\varepsilon_1^2 }{9h^2} + \frac{10\varepsilon_3^2 }{9h^2} + 3(\chi_{12} + 3
	\chi_{13}+\chi_{23})+\frac{8A_1\Delta t}{h^2}.
	\nonumber
\end{align}
The following quantity is introduced: 
\[
D_0 := - \frac{1}{M_0}\ln \overline{\phi_1^0}+\frac{4C_1}{\mathcal{M}_1\dt} + \frac{2\varepsilon_1^2 }{9h^2}+ \frac{10\varepsilon_3^2 }{9h^2} + 3(\chi_{12}+3 \chi_{13}+\chi_{23})+\frac{8A_1\Delta t}{h^2}.
\]
Notice that $D_0$ is a constant for a fixed $\dt, h$, while it becomes singular as $\dt, h \rightarrow 0$. For any fixed $\dt,h$, we could choose $g(\delta)$ small enough so that
\begin{equation}
	\ln \frac{(g(\delta))^{\nicefrac{1}{M_0}}}{\delta} + D_0 < 0.
	\label{MMC-positive-delta}
\end{equation}
In particular, we can choose
\[
g(\delta) := (\delta \exp(-D_0-1))^{M_0}.
\]
This in turn shows that
\[
\frac{1}{h^2}d_s \mathcal{J}_h^n(\phi_1^\star+s\psi,\phi_2^\star)|_{s=0}  <0,  	\]
provided that $g(\delta)$ satisfies \eqref{MMC-positive-delta}. But, this contradicts the assumption that $\mathcal{J}_h^n$ has a minimum at $(\phi_1^{\star},\phi_2^{\star})$, since the directional derivative is negative in a direction pointing into $(A_{h,\delta})^{\rm o}$, the interior of $A_{h,\delta}$.
\subsubsection{The minimizer $(\phi_1^{\star},\phi_2^{\star})\in A_{h,\delta}$ could not occur at $\phi_2^{\star}=g(\delta)$.}\label{subsubsection3.2.2}

Using similar arguments to the subsection \ref{subsubsection3.2.1}, we are able to prove that, the global minimum of $\mathcal{J}_h^n$ over $A_{h,\delta}$ could not occur on the boundary section where $(\phi_2^{\star})_{\vec{\alpha}_0}=g(\delta)$, if $g(\delta)$ is small enough, for any grid index $\vec{\alpha}_0$.
\subsubsection{The minimizer $(\phi_1^{\star},\phi_2^{\star})\in A_{h,\delta}$ could not occur at $\phi_1^{\star}+\phi_2^{\star}=1-\delta$.}\label{subsubsection3.2.3}

Suppose the minimum point $(\phi_1^{\star},\phi_2^{\star})$  satisfies
\[
(\phi_1^{\star})_{\vec{\alpha}_0} +
(\phi_2^{\star})_{\vec{\alpha}_0}=1-\delta,
\]
with $\vec{\alpha}_0:=(i_0,j_0)$. We could choose $\delta \in (0, \nicefrac{1}{3}).$ Without loss of generality, it is assumed that $(\phi_1^{\star})_{\vec{\alpha}_0} \geq \frac{1}{3}$. In addition, we see that 
\[
\frac{1}{N^2}\sum_{i,j=1}^N (\phi_1^{\star}+\phi_2^{\star})_{i,j}=\overline{\phi_1^{\star}}+\overline{\phi_2^{\star}}.
\]
There exists one grid point $\vec{\alpha}_1:=(i_1,j_1)$, so that $\phi_1^{\star}+\phi_2^{\star}$ reaches the minimum value at $\vec{\alpha}_1$. Then it is obvious that $(\phi_1^{\star})_{\vec{\alpha}_1}+(\phi_2^{\star})_{\vec{\alpha}_1}\leq
\overline{\phi_1^{\star}}+\overline{\phi_2^{\star}}=\overline{\phi_1^0}+\overline{\phi_2^0}$.
In turn, the following directional derivative could be derived:
	\begin{align*}
		&d_s \mathcal{J}_h^n(\phi_1^\star+s\psi,\phi_2^\star)|_{s=0} \nonumber\\
		=& \frac{1}{2\mathcal{M}_1\Delta
			t}\langle (-\Delta_h)^{-1}\left(3\phi_1^\star-4\phi_1^n+\phi_1^{n-1} \right),\psi\rangle - A_1\Delta t\langle \Delta_h(\phi_1^\star-\phi_1^n),\psi\rangle
		\\
		&  + \langle \delta_{\phi_1}G_{h,c}(\phi_1^{\star},\phi_2^{\star}),\psi\rangle+
		\langle \frac{\partial}{\partial \phi_1}H(\hat{\phi}_1^n,\hat{\phi}_2^n),\psi\rangle ,
\end{align*}
for any $\psi \in \mathring{{\cal C}}_{\rm per}$. Setting  the direction as
\[
\psi_{i,j} = \delta_{i,i_0}\delta_{j,j_0} - \delta_{i,i_1}\delta_{j,j_1},
\]
then the derivative may be expanded as
\begin{align}
	\frac{1}{h^2}d_s \mathcal{J}_h^n(\phi_1^\star+s\psi,\phi_2^\star)|_{s=0} &= \frac{1}{2\mathcal{M}_1\Delta
		t}(-\Delta_h)^{-1}\left(3\phi_1^\star-4\phi_1^n+\phi_1^{n-1} \right)_{\vec{\alpha}_0}
	\nonumber
	\\
	& \quad  -
	\frac{1}{2\mathcal{M}_1\Delta
		t}(-\Delta_h)^{-1}\left(3\phi_1^\star-4\phi_1^n+\phi_1^{n-1} \right)_{\vec{\alpha}_1}
	\nonumber
	\\
	& \quad + Q_\ell(\phi_1^{\star},\phi_2^{\star})_{\vec{\alpha}_0}
	- Q_\ell(\phi_1^{\star},\phi_2^{\star})_{\vec{\alpha}_1}
	\label{MMC-Positive-right}
	\\
	& \quad +\frac{\partial}{\partial \phi_1}H(\hat{\phi}_1^n,\hat{\phi}_2^n)|_{\vec{\alpha}_0}
	-\frac{\partial}{\partial \phi_1}H(\hat{\phi}_1^n,\hat{\phi}_2^n)|_{\vec{\alpha}_1}
	\nonumber
	\\
	& \quad - A_1\Delta t\left(\Delta_h(\phi_1^\star-\phi_1^n)_{\vec{\alpha}_0}-\Delta_h(\phi_1^\star-\phi_1^n)_{\vec{\alpha}_1}\right).
	\nonumber
\end{align}
For the first and second terms appearing in \eqref{MMC-Positive-right}, we apply
Lemma \ref{MMC-positivity-Lem-0} and obtain
\begin{small}
	\begin{equation}
		- \frac{8C_1}{\mathcal{M}_1} \le
		\frac{1}{\mathcal{M}_1}(-\Delta_h)^{-1}\left(3\phi_1^\star-4\phi_1^n+\phi_1^{n-1} \right)_{\vec{\alpha}_0}
		-\frac{1}{\mathcal{M}_1}(-\Delta_h)^{-1}\left(3\phi_1^\star-4\phi_1^n+\phi_1^{n-1} \right)_{\vec{\alpha}_1}
		\le  \frac{8C_1}{\mathcal{M}_1} .
		\label{MMC-positive-right-1}
	\end{equation}
\end{small}
For the $Q_1$ terms, we have
\begin{align}
	&Q_1(\phi_1^{\star},\phi_2^{\star})_{\vec{\alpha}_0} -
	Q_1(\phi_1^{\star},\phi_2^{\star})_{\vec{\alpha}_1} \nonumber\\
	=&  \left( \ln
	\frac{(\phi_1^{\star})^{\nicefrac{1}{M_0}}}{1-\phi_1^{\star}-\phi_2^{\star}}\right)|_{\vec{\alpha}_0}
	- \left( \ln
	\frac{(\phi_1^{\star})^{\nicefrac{1}{M_0}}}{1-\phi_1^{\star}-\phi_2^{\star}}\right)|_{\vec{\alpha}_1}
	\nonumber
	\\
	\geq&  \ln \frac{(\frac{1}{3})^{\nicefrac{1}{M_0}}}{\delta} -
	\ln \frac{1}{1-\overline{\phi_1^0}-\overline{\phi_2^0}} .
	\label{MMC-positive-right-2}
\end{align}
The last step above comes from the facts that $(\phi_1^{\star})_{\vec{\alpha}_0} \geq \frac{1}{3}$,
$(\phi_1^{\star})_{\vec{\alpha}_1}+(\phi_2^{\star})_{\vec{\alpha}_1}\leq
\overline{\phi_1^0}+\overline{\phi_2^0}$, and $(\phi_1^{\star})_{\vec{\alpha}_1} < 1$ .

For the $Q_2$ terms, we have
\begin{align}
	&Q_2(\phi_1^{\star},\phi_2^{\star})_{\vec{\alpha}_0} -
	Q_2(\phi_1^{\star},\phi_2^{\star})_{\vec{\alpha}_1} \nonumber\\
	=&  \varepsilon_1^2
	a_x(\kappa'(A_x\phi_1^{\star})(D_x\phi_1^{\star})^2)_{\vec{\alpha}_0}
	- \varepsilon_1^2
	a_x(\kappa'(A_x\phi_1^{\star})(D_x\phi_1^{\star})^2)_{\vec{\alpha}_1}
	\nonumber
	\\
	\geq&  \varepsilon_1^2 a_x(\kappa'(A_x\phi_1^{\star})(D_x\phi_1^{\star})^2)_{\vec{\alpha}_0}
	\nonumber
	\\
	\geq&  -\frac{\varepsilon_1^2 }{9h^2} ,\label{MMC-positive-right-4}
\end{align}
in which the second step comes from the fact that
$-\varepsilon_1^2 a_x(\kappa'(A_x\phi_1^{\star})(D_x\phi_1^{\star})^2)_{\vec{\alpha}_1}\geq 0$,
since $\kappa'(\phi)= -\frac{1}{36\phi^2}<0$, and the last step is based on the definitions of
$\kappa'(\phi)$, $a_x$, $A_x$, and $D_x$,  as well as the fact that $|\frac{a-b}{a+b}|<1$,
$\forall a>0, b>0$.

For the $Q_4$ terms, similarly, we get 
\begin{align}
	Q_4(\phi_1^{\star},\phi_2^{\star})_{\vec{\alpha}_0} -
	Q_4(\phi_1^{\star},\phi_2^{\star})_{\vec{\alpha}_1} & = \varepsilon_1^2
	a_y(\kappa'(A_y\phi_1^{\star})(D_y\phi_1^{\star})^2)_{\vec{\alpha}_0}
	\label{MMC-positive-right-5}
	\\
	& \quad - \varepsilon_1^2 a_y(\kappa'(A_y\phi_1^{\star})(D_y\phi_1^{\star})^2)_{\vec{\alpha}_1}
	\nonumber
	\\
	& \geq -\frac{\varepsilon_1^2 }{9h^2} .
	\nonumber
\end{align}
The $Q_3$ and $Q_5$ terms could be analyzed as follows
\begin{align}
	Q_3(\phi_1^{\star},\phi_2^{\star})_{\vec{\alpha}_0} -
	Q_3(\phi_1^{\star},\phi_2^{\star})_{\vec{\alpha}_1} & = -2\varepsilon_1^2
	d_x(\kappa(A_x\phi_1^{\star}) D_x\phi_1^{\star})_{\vec{\alpha}_0}
	\label{MMC-positive-right-6}
	\\
	& \quad + 2\varepsilon_1^2
	d_x(\kappa(A_x\phi_1^{\star}) D_x\phi_1^{\star})_{\vec{\alpha}_1}
	\nonumber
	\\
	& \geq -\frac{4\varepsilon_1^2 }{9h^2} ,
	\nonumber
	\\
	Q_5(\phi_1^{\star},\phi_2^{\star})_{\vec{\alpha}_0} -
	Q_5(\phi_1^{\star},\phi_2^{\star})_{\vec{\alpha}_1} & = -2\varepsilon_1^2
	d_y(\kappa(A_y\phi_1^{\star}) D_y\phi_1^{\star})_{\vec{\alpha}_0}
	\label{MMC-positive-right-7}
	\\
	& \quad + 2\varepsilon_1^2 d_y(\kappa(A_y\phi_1^{\star}) D_y\phi_1^{\star})_{\vec{\alpha}_1}
	\nonumber
	\\
	& \geq -\frac{4\varepsilon_1^2 }{9h^2} .    \nonumber
\end{align}
The estimates for the $Q_6$ and $Q_8$ terms are similar:
	\begin{align}
		&Q_6(\phi_1^{\star},\phi_2^{\star})_{\vec{\alpha}_0} -
		Q_6(\phi_1^{\star},\phi_2^{\star})_{\vec{\alpha}_1} \nonumber\\
		=&  -\varepsilon_3^2
		a_x\left(\kappa'(A_x(1-\phi_1^{\star}-\phi_2^{\star}))(D_x(1-\phi_1^{\star}-\phi_2^{\star}))^2
		\right)_{\vec{\alpha}_0}
		\nonumber
		\\
		&+ \varepsilon_3^2
		a_x\left(\kappa'(A_x(1-\phi_1^{\star}-\phi_2^{\star}))(D_x(1-\phi_1^{\star}-\phi_2^{\star}))^2\right)_{\vec{\alpha}_1}
		\nonumber
		\\
		\geq&  \varepsilon_3^2
		a_x\left(\kappa'(A_x(1-\phi_1^{\star}-\phi_2^{\star}))(D_x(1-\phi_1^{\star}-\phi_2^{\star}))^2\right)_{\vec{\alpha}_1}
		\nonumber
		\\
		\geq&  - \frac{\varepsilon_3^2 }{9h^2} ,
		\label{MMC-positive-right-8}
		\\
		&Q_8(\phi_1^{\star},\phi_2^{\star})_{\vec{\alpha}_0} -
		Q_8(\phi_1^{\star},\phi_2^{\star})_{\vec{\alpha}_1} \nonumber\\
		=&  -\varepsilon_3^2
		a_y\left(\kappa'(A_y(1-\phi_1^{\star}-\phi_2^{\star}))(D_y(1-\phi_1^{\star}-\phi_2^{\star}))^2\right)_{\vec{\alpha}_0}
		\nonumber
		\\
		&+ \varepsilon_3^2
		a_y\left(\kappa'(A_y(1-\phi_1^{\star}-\phi_2^{\star}))(D_y(1-\phi_1^{\star}-\phi_2^{\star}))^2\right)_{\vec{\alpha}_1}
		\nonumber
		\\
		\geq&  \varepsilon_3^2
		a_y\left(\kappa'(A_y(1-\phi_1^{\star}-\phi_2^{\star}))(D_y(1-\phi_1^{\star}-\phi_2^{\star}))^2\right)_{\vec{\alpha}_1}
		\nonumber
		\\
		\geq&  -\frac{\varepsilon_3^2 }{9h^2} .\label{MMC-positive-right-9}
\end{align}
For the $Q_7$ terms, we see that
\begin{align}
	Q_7(\phi_1^{\star},\phi_2^{\star})_{\vec{\alpha}_0} -
	Q_7(\phi_1^{\star},\phi_2^{\star})_{\vec{\alpha}_1} & = 2\varepsilon_3^2
	d_x(\kappa(A_x(1-\phi_1^{\star}-\phi_2^{\star}))D_x(1-\phi_1^{\star}-\phi_2^{\star}))_{\vec{\alpha}_0}
	\label{MMC-positive-right-10}
	\\
	& \quad - 2\varepsilon_3^2 d_x(\kappa(A_x(1-\phi_1^{\star}-\phi_2^{\star})) D_x(1-\phi_1^{\star}-\phi_2^{\star}) )_{\vec{\alpha}_1}
	\nonumber
	\\
	& \geq 0 .
	\nonumber
\end{align}
The last step above comes from the fact that
\begin{align*}
	(D_x (1-\phi_1^{\star}-\phi_2^{\star}))_{i_0-\nicefrac{1}{2},j_0} & \leq 0,
	\\
	(D_x(1-\phi_1^{\star}-\phi_2^{\star}))_{i_0+\nicefrac{1}{2},j_0} & \geq 0,
	\\
	(D_x(1-\phi_1^{\star}-\phi_2^{\star}))_{i_1-\nicefrac{1}{2},j_1} & \geq 0,
	\\
	(D_x(1-\phi_1^{\star}-\phi_2^{\star}))_{i_1+\nicefrac{1}{2},j_1} &\leq 0.
\end{align*}
Similarly, for the $Q_9$ terms, we see that
\begin{align}
	Q_9(\phi_1^{\star},\phi_2^{\star})_{\vec{\alpha}_0} -
	Q_9(\phi_1^{\star},\phi_2^{\star})_{\vec{\alpha}_1} & = 2\varepsilon_3^2
	d_y(\kappa(A_y(1-\phi_1^{\star}-\phi_2^{\star}))D_y(1-\phi_1^{\star}-\phi_2^{\star}))_{\vec{\alpha}_0}
	\label{MMC-positive-right-11}
	\\
	& \quad - 2\varepsilon_3^2
	d_y(\kappa(A_y(1-\phi_1^{\star}-\phi_2^{\star}))D_y(1-\phi_1^{\star}-\phi_2^{\star}))_{\vec{\alpha}_1}
	\nonumber
	\\
	& \geq 0 .
	\nonumber
\end{align}
For the numerical solution $\phi_i^{n-1}, \phi_i^n$ at the previous time step, similar bounds could be derived for the fifth and sixth terms appearing in \eqref{MMC-Positive-right}, we have
	\begin{align}
		&\frac{\partial}{\partial \phi_1}H(\hat{\phi}_1^n,\hat{\phi}_2^n)|_{\vec{\alpha}_0} -
		\frac{\partial}{\partial \phi_1}H(\hat{\phi}_1^n,\hat{\phi}_2^n)|_{\vec{\alpha}_1}\nonumber\\
		= &  - 2 \chi_{13}
		[(\hat{\phi}_1^n)_{\vec{\alpha}_0}-(\hat{\phi}_1^n)_{\vec{\alpha}_1}]
		+ (\chi_{12}-\chi_{13}-\chi_{23})[(\hat{\phi}_2^n)_{\vec{\alpha}_0}-(\hat{\phi}_2^n)_{\vec{\alpha}_1}]
		\nonumber
		\\
		\ge&  -3(\chi_{12} + 3 \chi_{13}+\chi_{23}) .\label{MMC-positive-right-3}
\end{align}
For the last term appearing in \eqref{MMC-Positive-right}, similar bounds could be derived 
\begin{equation}
	- A_1\Delta t\left(\Delta_h(\phi_1^\star-\phi_1^n)_{\vec{\alpha}_0}-\Delta_h(\phi_1^\star-\phi_1^n)_{\vec{\alpha}_1}\right)	\ge -\frac{16A_1\Delta t}{h^2}.	\label{MMC-positive-right-14}
\end{equation}
Putting estimates together, we arrive at 
\begin{align}
	\frac{1}{h^2}d_s \mathcal{J}_h^n(\phi_1^\star+s\psi,\phi_2^\star)|_{s=0} & \geq \ln \frac{(\frac{1}{3})^{\nicefrac{1}{M_0}}}{\delta} -
	\ln \frac{1}{1-\overline{\phi_1^0}-\overline{\phi_2^0}}-\frac{4C_1}{\mathcal{M}_1\dt}
	\nonumber
	\\
	& \quad - \frac{10\varepsilon_1^2}{9h^2} - \frac{2\varepsilon_3^2}{9h^2} - 3(\chi_{12}+3 \chi_{13}+\chi_{23})-\frac{16A_1\Delta t}{h^2}.
	\nonumber
\end{align}
The following quantity is introduced:
	\begin{small}
		\begin{equation}
			D_1 := \frac{1}{M_0} \ln3 +\ln
			\frac{1}{1-\overline{\phi_1^0}-\overline{\phi_2^0}}+\frac{4C_1}{\mathcal{M}_1\dt} +  \frac{10\varepsilon_1^2}{9h^2}+\frac{2\varepsilon_3^2}{9h^2} +3(\chi_{12}+3 \chi_{13}+\chi_{23})+\frac{16A_1\Delta t}{h^2}.
		\end{equation}
\end{small}
For any fixed $\dt,h$, we could choose $\delta$ small enough so that
\begin{equation}
	-\ln \delta- D_1 > 0,
	\label{MMC-positive-right-delta}
\end{equation}
in particular, $\delta = \min \{\exp(-D_1-1), \nicefrac{1}{3} \}$. This in turn shows that
\[
\frac{1}{h^2}d_s \mathcal{J}_h^n(\phi_1^\star+s\psi,\phi_2^\star)|_{s=0} > 0,
\]
provided that $\delta$ satisfies  \eqref{MMC-positive-right-delta}. This contradicts the assumption that $\mathcal{J}_h^n$ has a minimum at
$(\phi_1^{\star},\phi_2^{\star})$.
\subsubsection{The minimizer $(\phi_1^{\star},\phi_2^{\star})\in A_{h,\delta}$ could not occur at $\phi_1^{\star}+\phi_2^{\star}=\delta$.}

Using similar arguments to the subsection \ref{subsubsection3.2.3}, we can also prove that, the global minimum of $\mathcal{J}_h^n$ over $A_{h,\delta}$ could not occur on the boundary section where $(\phi_1^{\star})_{\vec{\alpha}_0} + (\phi_2^{\star})_{\vec{\alpha}_0} = \delta$, if $\delta$ is small enough, for any grid index $\vec{\alpha}_0$. The details are left to the interested readers.

Finally, a combination of these four cases reveals that, the global minimizer of $\mathcal{J}_h^n(\phi_1,\phi_2)$ could only possibly occur at interior point of $( A_{h,\delta} )^0 \subset (A_h )^0$. We conclude that there must be a solution $(\phi_1, \phi_2) \in (A_h )^0$ that minimizes $\mathcal{J}_h^n(\phi_1,\phi_2)$ over $A_h$, which is equivalent to the numerical solution of~\eqref{Full-discrete-1}-\eqref{Full-discrete-mu2}. The existence of the numerical solution is established.

In addition, since $\mathcal{J}_h^n(\phi_1,\phi_2)$ is a strictly convex function over $A_h$, the uniqueness analysis for this numerical solution is straightforward. Using similar argument, the positivity-preserving property is established for the initialization step, the details are left to the interested readers. The proof of Theorem~\ref{MMC-positivity} is complete.
\section{Energy stability}\label{sec:unconditional energy
	stability}

Due to the fully discrete scheme~\eqref{Full-discrete-1}-\eqref{Full-discrete-mu2} is three-level scheme, here, we define the modified discrete energy as
\begin{align*}
	E_h^{n+1,n}:=& 	G_h(\phi_1^{n+1},\phi_2^{n+1})+\frac{\Delta t}{4\mathcal{M}_1}\nrm{\frac{\phi_1^{n+1}-\phi_1^n}{\Delta t}}_{-1,h}^2+\frac{\Delta t}{4\mathcal{M}_2}\nrm{\frac{\phi_2^{n+1}-\phi_2^n}{\Delta t}}_{-1,h}^2\\
	&+\frac{2\chi_{12}+3\chi_{13}+\chi_{23}}{2}\nrm{\phi_1^{n+1}-\phi_1^n}_2^2
	+\frac{2\chi_{12}+\chi_{13}+3\chi_{23}}{2}\nrm{\phi_2^{n+1}-\phi_2^n}_2^2,\nonumber
\end{align*}
in which $n \geq 0$.
\begin{theorem}(Energy stability) \label{MMC-energy stab} 
	When $n \geq 1$, $A_1 \ge \frac{\mathcal{M}_1(3\chi_{13}+3\chi_{23}+2\chi_{12})^2}{4}$, $A_2 \ge \frac{\mathcal{M}_2(3\chi_{13}+3\chi_{23}+2\chi_{12})^2}{4}$, the fully discrete scheme~\eqref{Full-discrete-1}-\eqref{Full-discrete-mu2} has the energy-decay property
	\begin{equation*}
		E_h^{n+1,n} \leq E_h^{n,n-1}. 
	\end{equation*}
\end{theorem}
\begin{proof}
	Due to the mass conservation,  $ \mathcal{L}^{-1} (\phi_i^{n+1}-\phi_i^n)$ is well-defined. Taking a discrete inner product with \eqref{Full-discrete-1} by $ \frac{1}{\mathcal{M}_1}\mathcal{L}^{-1} (\phi_1^{n+1}-\phi_1^n)$ , with \eqref{Full-discrete-mu1} by $\phi_1^{n+1}-\phi_1^n$,with \eqref{Full-discrete-2} by $ \frac{1}{\mathcal{M}_2}\mathcal{L}^{-1} (\phi_2^{n+1}-\phi_2^n)$ , with \eqref{Full-discrete-mu2} by $\phi_2^{n+1}-\phi_2^n$ yields
	     \begin{align*}
		0 =& \frac{1}{2\mathcal{M}_1\dt} \langle  3\phi_1^{n+1} - 4\phi_1^n + \phi_1^{n-1} , \mathcal{L}^{-1} (\phi_1^{n+1}-\phi_1^n) \rangle + \langle  \mu_1^{n+1}, \phi_1^{n+1}-\phi_1^n \rangle \\ 
		&+ \frac{1}{2\mathcal{M}_2\dt} \langle  3\phi_2^{n+1} - 4\phi_2^n + \phi_2^{n-1} , \mathcal{L}^{-1} (\phi_2^{n+1}-\phi_2^n) \rangle
		+ \langle \mu_2^{n+1} , \phi_2^{n+1}-\phi_2^n \rangle.
	\end{align*}
The equivalent form is the following identity
     \begin{eqnarray}
	0 &=& \frac{1}{2\mathcal{M}_1\dt} \langle  3\phi_1^{n+1} - 4\phi_1^n + \phi_1^{n-1} , \mathcal{L}^{-1} (\phi_1^{n+1}-\phi_1^n) \rangle +A_1 \dt \nrm{\nabla_h(\phi_1^{n+1}-\phi_1^n)}_2^2\nonumber\\
	&&+ \langle  \delta_{\phi_1} G_{h,c}(\phi_1^{n+1},\phi_2^{n+1})  - \delta_{\phi_1} G_{h,e}(\phi_1^n,\phi_2^n) , \phi_1^{n+1}-\phi_1^n \rangle\nonumber
	\\
	&&+ \langle  \delta_{\phi_1} G_{h,e}(\phi_1^n,\phi_2^n)  - \delta_{\phi_1} G_{h,e}(\hat{\phi}_1^n,\hat{\phi}_2^n) , \phi_1^{n+1}-\phi_1^n \rangle\nonumber
	\\
	&&+\frac{1}{2\mathcal{M}_2\dt} \langle  3\phi_2^{n+1} - 4\phi_2^n + \phi_2^{n-1}, \mathcal{L}^{-1} (\phi_2^{n+1}-\phi_2^n) \rangle +A_2 \dt \nrm{\nabla_h(\phi_2^{n+1}-\phi_2^n)}_2^2\nonumber\\
	&&+ \langle  \delta_{\phi_2} G_{h,c}(\phi_1^{n+1},\phi_2^{n+1})  - \delta_{\phi_2} G_{h,e}(\phi_1^n,\phi_2^n), \phi_2^{n+1}-\phi_2^n \rangle\nonumber
	\\
	&&+ \langle  \delta_{\phi_2} G_{h,e}(\phi_1^n,\phi_2^n)  - \delta_{\phi_2} G_{h,e}(\hat{\phi}_1^n,\hat{\phi}_2^n) , \phi_2^{n+1}-\phi_2^n \rangle
	\label{MMC3term_BDF2} .
\end{eqnarray}
For the first and fifth term of the right hand side of \eqref{MMC3term_BDF2}, we have
     \begin{align}
	&\frac{1}{2\mathcal{M}_i\dt} \langle 3\phi_i^{n+1} - 4\phi_i^n + \phi_i^{n-1} , \mathcal{L}^{-1} (\phi_i^{n+1}-\phi_i^n) \rangle \nonumber\\
	=& \frac{\dt}{\mathcal{M}_i} \left(\frac{5}{4}\nrm{\frac{\phi_i^{n+1}-\phi_i^n}{\dt}}_{-1,h}^2 - \frac{1}{4}\nrm{\frac{\phi_i^n-\phi_i^{n-1}}{\dt}}_{-1,h}^2\right)
	\nonumber\\
	&+ \frac{\dt^3}{4\mathcal{M}_i}\nrm{\frac{\phi_i^{n+1}-2\phi_i^n+\phi_i^{n-1}}{\dt^2}}_{-1,h}^2\nonumber\\
	\ge& \frac{\dt}{\mathcal{M}_i} \left(\frac{5}{4}\nrm{\frac{\phi_i^{n+1}-\phi_i^n}{\dt}}_{-1,h}^2 - \frac{1}{4}\nrm{\frac{\phi_i^n-\phi_i^{n-1}}{\dt}}_{-1,h}^2\right) ,\quad i = 1,2.\label{8}
\end{align}
For the third and seventh term of the right hand side of \eqref{MMC3term_BDF2}, using the Lemma \ref{lemma2.2}, we have
	\begin{align*}
		&\langle  \delta_{\phi_1} G_{h,c}(\phi_1^{n+1},\phi_2^{n+1})  - \delta_{\phi_1} G_{h,e}(\phi_1^n,\phi_2^n), \phi_1^{n+1}-\phi_1^n \rangle\nonumber\\
		&+ \langle  \delta_{\phi_2} G_{h,c}(\phi_1^{n+1},\phi_2^{n+1})  - \delta_{\phi_2} G_{h,e}(\phi_1^n,\phi_2^n) , \phi_2^{n+1}-\phi_2^n \rangle\nonumber\\
		\ge& G_h(\phi_1^{n+1},\phi_2^{n+1}) - G_h(\phi_1^n,\phi_2^n).
	\end{align*}
For the fourth term of the right hand side of \eqref{MMC3term_BDF2}, we have
\begin{align*}
	&\langle  \delta_{\phi_1} G_{h,e}(\phi_1^n,\phi_2^n)  - \delta_{\phi_1} G_{h,e}(\hat{\phi}_1^n,\hat{\phi}_2^n) , \phi_1^{n+1}-\phi_1^n \rangle\nonumber\\
	=&- \langle \frac{\partial}{\partial \phi_1} H(\phi_1^n,\phi_2^n) - \frac{\partial}{\partial \phi_1} H(\hat{\phi}_1^n,\hat{\phi}_2^n) , \phi_1^{n+1}-\phi_1^n \rangle\nonumber\\
	=&\langle 2\chi_{13}(\phi_1^n-\hat{\phi}_1^n)-(\chi_{12}-\chi_{13}-\chi_{23})(\phi_2^n-\hat{\phi}_2^n),\phi_1^{n+1}-\phi_1^n \rangle\nonumber\\
	=&-2\chi_{13}\langle \phi_1^n-\phi_1^{n-1}, \phi_1^{n+1}-\phi_1^n \rangle + (\chi_{12}-\chi_{13}-\chi_{23})\langle \phi_2^n-\phi_2^{n-1}, \phi_1^{n+1}-\phi_1^n \rangle\nonumber\\
	\ge& -\chi_{13} (\|\phi_1^n-\phi_1^{n-1}\|_2^2+\|\phi_1^{n+1}-\phi_1^n\|_2^2) \nonumber\\
	&+ \frac{\chi_{12}-\chi_{13}-\chi_{23}}{2}
	\left(\|\phi_2^n-\phi_2^{n-1}\|_2^2+\|\phi_1^{n+1}-\phi_1^n\|_2^2-\|\phi_2^n-\phi_2^{n-1}-\phi_1^{n+1}+\phi_1^n\|_2^2\right)\nonumber\\
	\ge& -\chi_{13} (\|\phi_1^n-\phi_1^{n-1}\|_2^2+\|\phi_1^{n+1}-\phi_1^n\|_2^2) \nonumber\\
	&- \frac{\chi_{13}+\chi_{23}}{2}
	\left(\|\phi_2^n-\phi_2^{n-1}\|_2^2+\|\phi_1^{n+1}-\phi_1^n\|_2^2\right) -\frac{\chi_{12}}{2}\|\phi_2^n-\phi_2^{n-1}-\phi_1^{n+1}+\phi_1^n\|_2^2\nonumber\\
	\ge& -\chi_{13} (\|\phi_1^n-\phi_1^{n-1}\|_2^2+\|\phi_1^{n+1}-\phi_1^n\|_2^2) \nonumber\\
	&- \frac{\chi_{13}+\chi_{23}}{2}
	\left(\|\phi_2^n-\phi_2^{n-1}\|_2^2+\|\phi_1^{n+1}-\phi_1^n\|_2^2\right)\nonumber\\
	& -\chi_{12}(\|\phi_2^n-\phi_2^{n-1}\|_2^2+\|\phi_1^{n+1}-\phi_1^n\|_2^2)\nonumber\\
	=& -\chi_{13} \|\phi_1^n-\phi_1^{n-1}\|_2^2 - \frac{3\chi_{13}+\chi_{23}+2\chi_{12}}{2} \|\phi_1^{n+1}-\phi_1^n\|_2^2 \nonumber\\
	&- \frac{2\chi_{12}+\chi_{13}+\chi_{23}}{2}\|\phi_2^n-\phi_2^{n-1}\|_2^2,
\end{align*}
in which the fourth step is based on the formula $2ab\le a^2+b^2, \,2ab=a^2+b^2-(a-b)^2$, the next-to-last step comes from inequality $(a-b)^2\le 2(a^2+b^2)$.

For the last term of the right hand side of \eqref{MMC3term_BDF2}, similarly, we have
\begin{small}
\begin{align*}
	&\langle  \delta_{\phi_2} G_{h,e}(\phi_1^n,\phi_2^n)  - \delta_{\phi_2} G_{h,e}(\hat{\phi}_1^n,\hat{\phi}_2^n) , \phi_2^{n+1}-\phi_2^n \rangle\nonumber\\
	=&- \langle \frac{\partial}{\partial \phi_2} H(\phi_1^n,\phi_2^n) - \frac{\partial}{\partial \phi_2} H(\hat{\phi}_1^n,\hat{\phi}_2^n) , \phi_2^{n+1}-\phi_2^n \rangle\nonumber\\
	=&\langle 2\chi_{23}(\phi_2^n-\hat{\phi}_2^n)-(\chi_{12}-\chi_{13}-\chi_{23})(\phi_1^n-\hat{\phi}_1^n),\phi_2^{n+1}-\phi_2^n \rangle\nonumber\\
	=&-2\chi_{23}\langle \phi_2^n-\phi_2^{n-1}, \phi_2^{n+1}-\phi_2^n \rangle + (\chi_{12}-\chi_{13}-\chi_{23})\langle \phi_1^n-\phi_1^{n-1} , \phi_2^{n+1}-\phi_2^n \rangle\nonumber\\
	\ge& -\chi_{23} (\|\phi_2^n-\phi_2^{n-1}\|_2^2+\|\phi_2^{n+1}-\phi_2^n\|_2^2) \nonumber\\
	&+ \frac{\chi_{12}-\chi_{13}-\chi_{23}}{2}
	\left(\|\phi_1^n-\phi_1^{n-1}\|_2^2+\|\phi_2^{n+1}-\phi_2^n\|_2^2-\|\phi_1^n-\phi_1^{n-1}-\phi_2^{n+1}+\phi_2^n\|_2^2\right)\nonumber\\
	\ge& -\chi_{23} (\|\phi_2^n-\phi_2^{n-1}\|_2^2+\|\phi_2^{n+1}-\phi_2^n\|_2^2) \nonumber\\
	&- \frac{\chi_{13}+\chi_{23}}{2}
	\left(\|\phi_1^n-\phi_1^{n-1}\|_2^2+\|\phi_2^{n+1}-\phi_2^n\|_2^2\right) -\frac{\chi_{12}}{2}\|\phi_1^n-\phi_1^{n-1}-\phi_2^{n+1}+\phi_2^n\|_2^2\nonumber\\
	\ge& -\chi_{23} (\|\phi_2^n-\phi_2^{n-1}\|_2^2+\|\phi_2^{n+1}-\phi_2^n\|_2^2) \nonumber\\
	&- \frac{\chi_{13}+\chi_{23}}{2}
	\left(\|\phi_1^n-\phi_1^{n-1}\|_2^2+\|\phi_2^{n+1}-\phi_2^n\|_2^2\right)\nonumber\\
	& -\chi_{12}(\|\phi_1^n-\phi_1^{n-1}\|_2^2+\|\phi_2^{n+1}-\phi_2^n\|_2^2)\nonumber\\
	=& -\chi_{23} \|\phi_2^n-\phi_2^{n-1}\|_2^2 - \frac{\chi_{13}+3\chi_{23}+2\chi_{12}}{2} \|\phi_2^{n+1}-\phi_2^n\|_2^2 \nonumber\\
	&- \frac{2\chi_{12}+\chi_{13}+\chi_{23}}{2}\|\phi_1^n-\phi_1^{n-1}\|_2^2.
\end{align*}
\end{small}
Going back \eqref{MMC3term_BDF2} and by simple calculation,  we arrive at
   \begin{eqnarray}
	&& \frac{\dt}{\mathcal{M}_1} \left(\frac{5}{4}\nrm{\frac{\phi_1^{n+1}-\phi_1^n}{\dt}}_{-1,h}^2 - \frac{1}{4}\nrm{\frac{\phi_1^n-\phi_1^{n-1}}{\dt}}_{-1,h}^2\right) 
	+ A_1 \dt \nrm{\nabla_h(\phi_1^{n+1}-\phi_1^n)}_2^2  \nonumber\\
	&& +  \frac{\dt}{\mathcal{M}_2} \left(\frac{5}{4}\nrm{\frac{\phi_2^{n+1}-\phi_2^n}{\dt}}_{-1,h}^2 - \frac{1}{4}\nrm{\frac{\phi_2^n-\phi_2^{n-1}}{\dt}}_{-1,h}^2\right)  + A_2 \dt \nrm{\nabla_h(\phi_2^{n+1}-\phi_2^n)}_2^2 \nonumber\\
	&&+G_h(\phi_1^{n+1},\phi_2^{n+1})-G_h(\phi_1^n,\phi_2^n)\nonumber\\
	&&- \frac{2\chi_{12}+3\chi_{13}+\chi_{23}}{2}\left(\|\phi_1^n-\phi_1^{n-1}\|_2^2 - \|\phi_1^{n+1}-\phi_1^n\|_2^2\right)
	\nonumber\\
	&& - \frac{\chi_{13}+3\chi_{23}+2\chi_{12}}{2} \left(\|\phi_2^n-\phi_2^{n-1}\|_2^2-\|\phi_2^{n+1}-\phi_2^n\|_2^2\right) \nonumber\\
	\leq&&  (3\chi_{13}+3\chi_{23}+2\chi_{12})\left(\nrm{ \phi_1^{n+1}-\phi_1^n }_2^2+\nrm{ \phi_2^{n+1}-\phi_2^n }_2^2\right).\label{energy_process}
\end{eqnarray}
For the right hand side of \eqref{energy_process}, we have
     \begin{eqnarray*}
	\nrm{\phi_1^{n+1}-\phi_1^n}_2^2 
	&\leq& \frac{\dt}{2\alpha}\nrm{\nabla_h(\phi_1^{n+1}-\phi_1^n)}_2^2 + \frac{\alpha\dt}{2}\nrm{\frac{\phi_1^{n+1}-\phi_1^n}{\dt}}_{-1,h}^2,
\end{eqnarray*}
and
     \begin{eqnarray*}
	\nrm{\phi_2^{n+1}-\phi_2^n}_2^2 
	&\leq& \frac{\dt}{2\beta}\nrm{\nabla_h(\phi_2^{n+1}-\phi_2^n)}_2^2 + \frac{\beta\dt}{2}\nrm{\frac{\phi_2^{n+1}-\phi_2^n}{\dt}}_{-1,h}^2.
\end{eqnarray*}
Let $\alpha = \frac{3\chi_{13}+3\chi_{23}+2\chi_{12}}{2A_1}$, $ \beta =\frac{3\chi_{13}+3\chi_{23}+2\chi_{12}}{2A_2}$, when $A_1 \ge \frac{\mathcal{M}_1(3\chi_{13}+3\chi_{23}+2\chi_{12})^2}{4}$, $A_2 \ge \frac{\mathcal{M}_2(3\chi_{13}+3\chi_{23}+2\chi_{12})^2}{4}$, we have
	\begin{equation*}
	E_h^{n+1,n} \leq E_h^{n,n-1}.
\end{equation*}
This completes the proof.
\end{proof}

In fact, for the fourth term of the right hand side of \eqref{MMC3term_BDF2}, it can be analyzed in another way:
\begin{align}
	&\langle  \delta_{\phi_1} G_{h,e}(\phi_1^n,\phi_2^n)  - \delta_{\phi_1} G_{h,e}(\hat{\phi}_1^n,\hat{\phi}_2^n) , \phi_1^{n+1}-\phi_1^n \rangle\nonumber\\
	=&-2\chi_{13}\langle \phi_1^n-\phi_1^{n-1}, \phi_1^{n+1}-\phi_1^n \rangle + (\chi_{12}-\chi_{13}-\chi_{23})\langle \phi_2^n-\phi_2^{n-1}, \phi_1^{n+1}-\phi_1^n \rangle\nonumber\\
	=&
	(\chi_{12}-\chi_{13}-\chi_{23})\langle \phi_2^n-\phi_2^{n-1}, \phi_1^{n+1}-\phi_1^n \rangle\nonumber\\	
	&- \chi_{13}(\|\phi_1^n-\phi_1^{n-1}\|_2^2-\|\phi_1^{n+1}-2\phi_1^n+\phi_1^{n-1}\|_2^2+\|\phi_1^{n+1}-\phi_1^n\|_2^2).\label{7}
\end{align}
Similarly, the following estimate is valid for $\phi_2$: 
\begin{align}
	&\langle  \delta_{\phi_2} G_{h,e}(\phi_1^n,\phi_2^n)  - \delta_{\phi_2} G_{h,e}(\hat{\phi}_1^n,\hat{\phi}_2^n) , \phi_2^{n+1}-\phi_2^n \rangle\nonumber\\
	=&
	(\chi_{12}-\chi_{13}-\chi_{23})\langle \phi_1^n-\phi_1^{n-1}, \phi_2^{n+1}-\phi_2^n \rangle\nonumber\\	
	&- \chi_{23}(\|\phi_2^n-\phi_2^{n-1}\|_2^2-\|\phi_2^{n+1}-2\phi_2^n+\phi_2^{n-1}\|_2^2+\|\phi_2^{n+1}-\phi_2^n\|_2^2).\label{9}
\end{align}
Meanwhile, the following vector norms are introduced: 
\begin{align*}
	&\|\phi^{n+1}-\phi^n\|_{-1,h}^2:=\|\phi_1^{n+1}-\phi_1^n\|_{-1,h}^2 + \|\phi_2^{n+1}-\phi_2^n\|_{-1,h}^2,\\
	&\|\phi^{n+1}-\phi^n\|_2^2:=\|\phi_1^{n+1}-\phi_1^n\|_2^2 + \|\phi_2^{n+1}-\phi_2^n\|_2^2.
\end{align*}
For simplify, we set the mobility parameters  ${\mathcal{M}_1}={\mathcal{M}_2}=1$ in the model.
Substituting \eqref{8}, \eqref{7} and \eqref{9} into \eqref{MMC3term_BDF2}, and applying Lemma \ref{Full-discrete-energy-splitting}, we obtain
\begin{small}
\begin{align}
	&\frac{1}{\dt} \left(\frac{5}{4}\nrm{\phi^{n+1}-\phi^n}_{-1,h}^2 - \frac{1}{4}\nrm{\phi^n-\phi^{n-1}}_{-1,h}^2\right) +G_{h}(\phi_1^{n+1},\phi_2^{n+1}) - 
	G_{h}(\phi_1^n,\phi_2^n)\nonumber\\
	&+ A_1 \dt \nrm{\nabla_h(\phi_1^{n+1}-\phi_1^n)}_2^2+ A_2 \dt \nrm{\nabla_h(\phi_2^{n+1}-\phi_2^n)}_2^2\nonumber\\
	&-\chi_{13}(\|\phi_1^n-\phi_1^{n-1}\|_2^2-\|\phi_1^{n+1}-\phi_1^n\|_2^2)
	-\chi_{23}(\|\phi_2^n-\phi_2^{n-1}\|_2^2-\|\phi_2^{n+1}-\phi_2^n\|_2^2)\nonumber\\
	\leq &-(\chi_{12}-\chi_{13}-\chi_{23})\langle \phi_2^n-\phi_2^{n-1}, \phi_1^{n+1}-\phi_1^n \rangle-(\chi_{12}-\chi_{13}-\chi_{23})\langle \phi_1^n-\phi_1^{n-1}, \phi_2^{n+1}-\phi_2^n \rangle
	\nonumber\\
	&+ 2\chi_{13}\|\phi_1^{n+1}-\phi_1^n\|_2^2 + 2\chi_{23}\|\phi_2^{n+1}-\phi_2^n\|_2^2.\label{11}
\end{align}
\end{small}
For the first term of the right hand of \eqref{11}, it is observed that 
\begin{align}
	&-(\chi_{12}-\chi_{13}-\chi_{23})\langle \phi_2^n-\phi_2^{n-1}, \phi_1^{n+1}-\phi_1^n \rangle\nonumber\\
	\leq&|\chi_{12}-\chi_{13}-\chi_{23}|\left(\frac{\dt}{2\beta_1}\|\nabla_h(\phi_1^{n+1}-\phi_1^n)\|_2^2 + \frac{\beta_1}{2\dt}\|\phi_2^n-\phi_2^{n-1}\|_{-1,h}^2\right).\label{12}
\end{align}
Similarly, for the second term of the right hand of \eqref{11}, the following inequality is available: 
\begin{align}
	&-(\chi_{12}-\chi_{13}-\chi_{23})\langle \phi_1^n-\phi_1^{n-1}, \phi_2^{n+1}-\phi_2^n \rangle\nonumber\\
	\leq&|\chi_{12}-\chi_{13}-\chi_{23}|\left(\frac{\dt}{2\beta_2}\|\nabla_h(\phi_2^{n+1}-\phi_2^n)\|_2^2 + \frac{\beta_2}{2\dt}\|\phi_1^n-\phi_1^{n-1}\|_{-1,h}^2\right).\label{13}
\end{align}
For the third term of the right hand of \eqref{11}, we see that 
\begin{equation}
	2\chi_{13}\|\phi_1^{n+1}-\phi_1^n\|_2^2
	\leq \chi_{13}\left(\frac{\dt}{\alpha_1}\|\nabla_h(\phi_1^{n+1}-\phi_1^n)\|_2^2 + \frac{\alpha_1}{\dt}\|\phi_1^{n+1}-\phi_1^n\|_{-1,h}^2.
	\right).\label{14}
\end{equation}
A similar estimate could also be derived for the fourth term of the right hand of \eqref{11}: 
\begin{equation}
	2\chi_{23}\|\phi_2^{n+1}-\phi_2^n\|_2^2
	\leq \chi_{23}\left(\frac{\dt}{\alpha_2}\|\nabla_h(\phi_2^{n+1}-\phi_2^n)\|_2^2 + \frac{\alpha_2}{\dt}\|\phi_2^{n+1}-\phi_2^n\|_{-1,h}^2.
	\right).\label{15}
\end{equation}
Subsequently, the following constant quantities are denoted: 
\begin{align*}
	&\alpha_1=\frac{1}{2\chi_{13}}, \quad \beta_2=\frac{1}{|\chi_{12}-\chi_{13}-\chi_{23}|},\\
	&\alpha_2=\frac{1}{2\chi_{23}}, \quad \beta_1=\frac{1}{|\chi_{12}-\chi_{13}-\chi_{23}|} , 	
\end{align*}
and 
\begin{equation*}
	A_1=\frac{\chi_{13}}{\alpha_1}+\frac{|\chi_{12}-\chi_{13}-\chi_{23}|}{2\beta_1},\quad A_2=\frac{\chi_{23}}{\alpha_2}+\frac{|\chi_{12}-\chi_{13}-\chi_{23}|}{2\beta_2}.	
\end{equation*}
A careful calculation reveals that 
\begin{equation*}
	A_1=2\chi_{13}^2+0.5(\chi_{12}-\chi_{13}-\chi_{23})^2,\quad A_2=2\chi_{23}^2+0.5(\chi_{12}-\chi_{13}-\chi_{23})^2.	
\end{equation*}
In turn, with an introduction of the following quantity 
\begin{align*}
	F_h^{n+1}:=&G_h(\phi_1^{n+1},\phi_2^{n+1}) + \frac{3}{4\dt}\left(\|\phi_1^{n+1}-\phi_1^n\|_{-1,h}^2+\|\phi_2^{n+1}-\phi_2^n\|_{-1,h}^2\right)\\
	& +\chi_{13}\|\phi_1^{n+1}-\phi_1^n\|_2^2 + \chi_{23}\|\phi_2^{n+1}-\phi_2^n\|_2^2 , 
\end{align*}
a substitution of \eqref{12}-\eqref{15} into \eqref{11} results in 
\begin{equation*}
	F_h^{n+1} \leq F_h^n.
\end{equation*}
\section{Optimal rate convergence analysis in $\ell^\infty (0,T; H_h^{-1}) \cap \ell^2 (0,T; H_h^1)$} \label{sec:convergence}

Now we proceed into the convergence analysis.
Let $\Phi_1$, $\Phi_2$ be the exact solution for the ternary Cahn-Hilliard flow \eqref{MMC3term-equation}-\eqref{mu_2}. With sufficiently regular initial data, we could assume that the exact solution has regularity of class $\mathcal{R}$:
\begin{equation*}
	\Phi_1, \Phi_2 \in \mathcal{R} := H^5 \left(0,T; C_{\rm per}(\Omega)\right) \cap H^4 \left(0,T; C^2_{\rm per}(\Omega)\right) \cap L^\infty \left(0,T; C^6_{\rm per}(\Omega)\right).
\end{equation*}
In addition, we assume that the following separation property is valid for the exact solution: for some $\delta_0$, 
\begin{equation}
	(\Phi_1,\Phi_2) \in \mathcal{G}_{\delta_0},
	\label{assumption:separation}
\end{equation}
which is satisfied at a point-wise level, for all $t\in[0,T]$. Define $\Phi_{1,N} (\, \cdot \, ,t) := {\cal P}_N \Phi_1 (\, \cdot \, ,t)$, $\Phi_{2,N} (\, \cdot \, ,t) := {\cal P}_N \Phi_2 (\, \cdot \, ,t)$, the (spatial) Fourier projection of the exact solution into ${\cal B}^K$, the space of trigonometric polynomials of degree up to and including $K$ (with $N=2K +1$).  The following projection approximation is standard: if $\Phi_j \in L^\infty(0,T;H^\ell_{\rm per}(\Omega))$, for some $\ell\in\mathbb{N}$,
\begin{equation}
	\nrm{\Phi_{j,N} - \Phi_j}_{L^\infty(0,T;H^k)}
	\le C h^{\ell-k} \nrm{\Phi_j }_{L^\infty(0,T;H^\ell)},  \quad \forall \ 0 \le k \le \ell ,
	\, \, j= 1, 2. \label{projection-est-0}
\end{equation}
By $\Phi_{j,N}^m$, $\Phi_j^m$ we denote $\Phi_{j,N} (\, \cdot \, , t_m)$ and $\Phi_j (\, \cdot \, , t_m)$, respectively, with $t_m = m\cdot \dt,~ j = 1,2$. 

Since the exact solution has regularity of class $\mathcal{R}$, the separation property \eqref{assumption:separation} for the exact solution and the projection approximation \eqref{projection-est-0}, we are able to obtain a discrete $W^{1,\infty}$ bound and the separation property for the projection of the exact solution:
\begin{equation*}
	\| \nabla_h \Phi_{j,N} \|_\infty \le C^* ,\quad j=1, 2,\quad (\Phi_{1,N},\Phi_{2,N}) \in \mathcal{G}_{\delta_0}  .
\end{equation*}
Since $\Phi_{j,N} \in {\cal B}^K$, the mass conservative property is available at the discrete level:
\begin{equation}
	\overline{\Phi_{j,N}^m} = \frac{1}{|\Omega|}\int_\Omega \, \Phi_{j,N} ( \cdot, t_m) \, d {\bf x} = \frac{1}{|\Omega|}\int_\Omega \, \Phi_{j,N} ( \cdot, t_{m-1}) \, d {\bf x} = \overline{\Phi_{j,N}^{m-1}} ,  \quad \forall \ m \in\mathbb{N}.
	\label{mass conserv-1}
\end{equation}
On the other hand, the solution of~\eqref{Full-discrete-1}-\eqref{scheme-CH_initial} is also mass conservative at the discrete level:
\begin{equation}
	\overline{\phi_j^m} = \overline{\phi_j^{m-1}} ,  \quad \forall \ m \in \mathbb{N} , \, \,
	j = 1, 2 .
	\label{mass conserv-2}
\end{equation}
Defined the grid projection operator, ${\mathcal P}_h : C_{\rm per}^0(\Omega) \rightarrow{\mathcal C}_{\rm per}$, by ${\mathcal P}_h f_{i,j} = f(p_i,p_j)$, for all $f\in C_{\rm per}^0(\Omega)$. For the initial data, we have  $\phi_j^0 = {\mathcal P}_h \Phi_{j,N} (\, \cdot \, , t=0)$, that is
\begin{equation*}
	(\phi_1^0)_{i,j} := \Phi_{1,N} (p_i, p_j, t=0) , \, \, \,
	(\phi_2^0)_{i,j} := \Phi_{2,N} (p_i, p_j, t=0) .
\end{equation*}
The error grid function is defined as
\begin{equation}
	e_1^m := \mathcal{P}_h \Phi_{1,N}^m - \phi_1^m ,  \, \, \,
	e_2^m := \mathcal{P}_h \Phi_{2,N}^m - \phi_2^m , \quad \forall \ m \in \left\{ 0 ,1 , 2, 3, \cdots \right\} .
	\label{error function-1}
\end{equation}
Therefore, it follows that  $\overline{e_j^m} =0$, for any $m \in \left\{ 0 ,1 , 2, 3, \cdots \right\}$,  $j=1, 2$, so the discrete norm $\nrm{ \, \cdot \, }_{-1,h}$ is well defined for the error grid functions $e_1^m$ and $e_2^m$. Before proceeding into the convergence analysis, we introduce a new norm $\nrm{\, \cdot \, }_{-1,G}$~\cite{Dong2018Convergence}. Let $\Omega$ be an arbitrary bounded domain and $\mathbf{p} = [u,v]^T \in [L^2(\Omega)]^2$. We define $\nrm{\, \cdot \, }_{-1,G}$ to be a weighted inner product
\begin{equation}
	\nrm{ \mathbf{p} }_{-1,G}^2 = (\mathbf{p},G(-\Delta_h)^{-1}\mathbf{p}),\quad G = \left(
	\begin{array}{cc}
		\frac{1}{2} & -1 \\
		-1 & \frac{5}{2} \\
	\end{array}
	\right).	
	\nonumber
\end{equation}
Since G is symmetric positive definite, the norm is well-defined. Moreover,
\[
G = \left(
\begin{array}{cc}
	\frac{1}{2} & -1 \\
	-1 & \frac{5}{2} \\
\end{array}
\right)
=  \left(
\begin{array}{cc}
	\frac{1}{2} & -1 \\
	-1 & 2 \\
\end{array}
\right)
+ \left(
\begin{array}{cc}
	0 & 0 \\
	0 & \frac{1}{2} \\
\end{array}
\right)
=:G_1 + G_2.
\]
By the positive semi-definiteness of $G_1$, we immediately have
\[
\nrm{ \mathbf{p} }_{-1,G}^2 = (\mathbf{p},(G_1+G_2)(-\Delta_h)^{-1}\mathbf{p}) \geq (\mathbf{p},G_2(-\Delta_h)^{-1}\mathbf{p}) = \frac{1}{2} \nrm{v}_{-1,h}^2.
\]
In addition, for any $v_i \in \mathring{\mathcal C}_{\rm per},i = 1,2$, the following equality is valid:
\begin{align}
	&\langle \frac{3}{2}v_i^{n+1}-2v_i^n+\frac{1}{2}v_i^{n-1}, (-\Delta_h)^{-1}v_i^{n+1}\rangle \nonumber\\
	=& \frac{1}{2}(\nrm{ \mathbf{p}_i^{n+1} }_{-1,G}^2 - \nrm{ \mathbf{p}_i^n }_{-1,G}^2) + \frac{\nrm{v_i^{n+1}-2v_i^n+v_i^{n-1} }_{-1,h}^2}{4},\label{BDF2_term_estimate}
\end{align}
with $\mathbf{p}_i^{n+1} = [v_i^n,v_i^{n+1}]^T, \mathbf{p}_i^n = [v_i^{n-1},v_i^n]^T$.

The optimal rate convergence result is stated in the following theorem.
\begin{thm}
	\label{thm:convergence}
	Given initial data $\Phi_j(\, \cdot \, ,t=0) \in C^6_{\rm per}(\Omega)$ and $\Phi_j(\, \cdot \, ,t=0)\in\mathcal{G}$, point-wise, suppose the exact solution for ternary MMC flow~\eqref{MMC3term-equation}-\eqref{mu_2} is of regularity class $\mathcal{R}$. Then, provided $\dt$ and $h$ are sufficiently small, and under the linear refinement requirement $C_1 h \le \dt \le C_2 h$, we have
	\begin{equation}
		\| e_j^n \|_{-1,h} +  \Bigl( \frac{\varepsilon_0^2}{36} \dt   \sum_{m=1}^{n} \| \nabla_h e_j^m \|_2^2 \Bigr)^{1/2}  \le C ( \dt^2 + h^2 ),  \quad \varepsilon_0 = \min (\varepsilon_1^2, \varepsilon_2^2, \varepsilon_3^2 ) ,  \, \, \, j=1, 2 ,
		\label{convergence-0}
	\end{equation}
	for all positive integers $n$, such that $t_n=n\dt \le T$, where $n \geq 1$. $C>0$ is independent of $\dt$ and $h$.
\end{thm}
\subsection{Higher order consistency analysis} \label{subsec: consistency} 

A direct substitution of the projection solution $\Phi_{j,N}$ into the numerical scheme \eqref{Full-discrete-1}-\eqref{Full-discrete-mu2} gives the second order accuracy in both time and space. However, due to  the explicit part of the extra regularization term, this leading local truncation error will not be sufficient to recover an $L_{\dt}^\infty$ bound of the discrete temporal derivative of the numerical solution, which is needed in the nonlinear convergence analysis. This technique has been reported for a wide class of nonlinear PDEs, such as incompressible fluid flow, various gradient models, the porous medium equation based on the energetic variational approach, nonlinear wave equation, et cetera. Such a higher order consistency result and the detailed proof is stated below.
\begin{proposition}
	Given the exact solution $\Phi_j$ for the ternary MMC system \eqref{MMC3term-equation}-\eqref{mu_2} and its Fourier projection $\Phi_{j,N}$. There exist auxiliary fields, $\Phi_{j,\dt},\Phi_{j,h}$, so that the following
\begin{equation}
	 \check{\Phi}_j = \Phi_{j,N} + {\cal P}_N (\dt^2 \Phi_{j,\dt} + h^2 \Phi_{j,h}),
	\quad j =1, 2 ,
	\label{consistency-1}
\end{equation}	
satisfies the numerical scheme up to a higher $O(\Delta t^3 + h^4)$ consistency:
\begin{align}
	&\frac{3\check{\Phi}_1^{n+1} - 4\check{\Phi}_1^n+\check{\Phi}_1^{n-1}}{2\dt} \nonumber\\
	=&  	
	\mathcal{M}_1 \Delta_h \biggl(  \frac{1}{M_0}\ln \frac{\alpha \check{\Phi}_1^{n+1}}{M_0} - \ln(1- \check{\Phi}_1^{n+1} - \check{\Phi}_2^{n+1}) - 2 \chi_{13} (2\check{\Phi}_1^n-\check{\Phi}_1^{n-1})  \nonumber
	\\
	& 
	+(\chi_{12}-\chi_{13}-\chi_{23}) (2\check{\Phi}_2^n-\check{\Phi}_2^{n-1}) - \frac{\varepsilon_1^2}{36} {\cal A}_h \Bigl( \frac{| \nabla_h \check{\Phi}_1^{n+1} |^2}{ ( {\cal A}_h \check{\Phi}_1^{n+1} )^2}  \Bigr)
	\nonumber
	\\
	& 
	- \frac{\varepsilon_1^2}{18} \nabla_h \cdot \Bigl( \frac{\nabla_h \check{\Phi}_1^{n+1}}{{\cal A}_h \check{\Phi}_1^{n+1} } \Bigr) + \frac{\varepsilon_3^2}{36} {\cal A}_h \Bigl( \frac{| \nabla_h ( 1 - \check{\Phi}_1^{n+1} - \check{\Phi}_2^{n+1}) |^2}{ ( {\cal A}_h (1 - \check{\Phi}_1^{n+1} - \check{\Phi}_2^{n+1} ) )^2}  \Bigr) 
	\nonumber\\
	&  
	+ \frac{\varepsilon_3^2}{18} \nabla_h \cdot \Bigl( \frac{\nabla_h (1 - \check{\Phi}_1^{n+1} - \check{\Phi}_2^{n+1}) }{{\cal A}_h (1 - \check{\Phi}_1^{n+1} - \check{\Phi}_2^{n+1}) } \Bigr) + A_1\Delta t \Delta_h(\check{\Phi}_1^{n+1} - \check{\Phi}_1^n ) \biggr) \nonumber\\
	& 
	+ \tau_1^{n+1} ,
	\label{consistency-0-1}
	\\
	&\frac{3\check{\Phi}_2^{n+1} - 4\check{\Phi}_2^n+\check{\Phi}_2^{n-1}}{2\dt}\nonumber\\
	= & 	
	\mathcal{M}_2 \Delta_h \biggl(  \frac{1}{N_0} \ln \frac{\beta \check{\Phi}_2^{n+1}}{N_0} - \ln(1- \check{\Phi}_1^{n+1} - \check{\Phi}_2^{n+1}) - 2 \chi_{23} (2\check{\Phi}_2^n-\check{\Phi}_2^{n-1})   \nonumber
	\\
	& 
	+(\chi_{12}-\chi_{13}-\chi_{23}) (2\check{\Phi}_1^n-\check{\Phi}_1^{n-1}) - \frac{\varepsilon_2^2}{36} {\cal A}_h \Bigl( \frac{| \nabla_h \check{\Phi}_2^{n+1} |^2}{ ( {\cal A}_h \check{\Phi}_2^{n+1} )^2}  \Bigr)
	\nonumber
	\\
	& 
	- \frac{\varepsilon_2^2}{18} \nabla_h \cdot \Bigl( \frac{\nabla_h \check{\Phi}_2^{n+1}}{{\cal A}_h \check{\Phi}_2^{n+1} } \Bigr) + \frac{\varepsilon_3^2}{36} {\cal A}_h \Bigl( \frac{| \nabla_h ( 1 - \check{\Phi}_1^{n+1} - \check{\Phi}_2^{n+1}) |^2}{ ( {\cal A}_h (1 - \check{\Phi}_1^{n+1} - \check{\Phi}_2^{n+1} ) )^2}  \Bigr)
	\nonumber\\
	&   
	+ \frac{\varepsilon_3^2}{18} \nabla_h \cdot \Bigl( \frac{\nabla_h (1 - \check{\Phi}_1^{n+1} - \check{\Phi}_2^{n+1}) }{{\cal A}_h (1 - \check{\Phi}_1^{n+1} - \check{\Phi}_2^{n+1}) } \Bigr) + A_2\Delta t \Delta_h(\check{\Phi}_2^{n+1} - \check{\Phi}_2^n ) \biggr) \nonumber\\
	& 
	+ \tau_2^{n+1}.
	\label{consistency-0-2}
\end{align}
with $\| \tau_1^{n+1} \|_{-1,h}$, $\| \tau_2^{n+1} \|_{-1,h} \le C (\dt^3 + h^4)$, $n \geq 1$. The constructed functions, $\Phi_{j,\dt}, \Phi_{j,h}$, depend solely on the exact solution $\Phi_j$, and their derivatives are bounded. For the initialization step, we have 
\begin{equation*}
	\check{\Phi}_j^0=\Phi_{j,N}^0,\quad \check{\Phi}_j^1=\Phi_{j,N}^1.
\end{equation*}	
\end{proposition}
\begin{proof}
In terms of the temporal discretization, the following local truncation error can be derived by a Taylor expansion in time, combined with the projection estimate \eqref{projection-est-0}:
\begin{align}
	&\frac{3\Phi_{1,N}^{n+1} - 4\Phi_{1,N}^n +\Phi_{1,N}^{n-1}}{2\dt} \nonumber\\
	= & 	
	\mathcal{M}_1 \Delta \biggl(  \frac{1}{M_0}\ln \frac{\alpha \Phi_{1,N}^{n+1}}{M_0} - \ln(1-\Phi_{1,N}^{n+1} - \Phi_{2,N}^{n+1}) - 2 \chi_{13} \hat{\Phi}_{1,N}^n   \nonumber
	\\
	&
	+(\chi_{12}-\chi_{13}-\chi_{23}) \hat{\Phi}_{2,N}^n - \frac{\varepsilon_1^2}{36} \frac{| \nabla \Phi_{1,N}^{n+1} |^2}{ ( \Phi_{1,N}^{n+1} )^2}
	- \frac{\varepsilon_1^2}{18} \nabla \cdot \Bigl( \frac{\nabla \Phi_{1,N}^{n+1}}{ \Phi_{1,N}^{n+1} } \Bigr)
   \nonumber
	\\
	&
	+ \frac{\varepsilon_3^2}{36} \Bigl( \frac{ | \nabla ( 1 - \Phi_{1,N}^{n+1} - \Phi_{2,N}^{n+1} ) |^2 }{ (1 - \Phi_{1,N}^{n+1} - \Phi_{2,N}^{n+1} )^2}  \Bigr) + \frac{\varepsilon_3^2}{18} \nabla \cdot \Bigl( \frac{\nabla (1 - \Phi_{1,N}^{n+1} - \Phi_{2,N}^{n+1} ) }{ (1 - \Phi_{1,N}^{n+1} - \Phi_{2,N}^{n+1} ) } \Bigr) \nonumber\\
	&
	+ A_1 \dt\Delta(\Phi_{1,N}^{n+1} - \Phi_{1,N}^n)  \biggr) 
	+ \dt^2 \g_1^{(0)} + O(\dt^3) + O (h^m) ,
	\label{consistency-2-1} \\
	&\frac{3\Phi_{2,N}^{n+1} - 4\Phi_{2,N}^n + \Phi_{2,N}^{n-1}}{2\dt} \nonumber\\
	 =&  \mathcal{M}_2 \Delta \biggl( \frac{1}{N_0}\ln \frac{\beta \Phi_{2,N}^{n+1} }{N_0} - \ln(1-\Phi_{1,N}^{n+1} - \Phi_{2,N}^{n+1} ) - 2 \chi_{23} \hat{\Phi}_{2,N}^n   \nonumber
	\\
	& 
	+(\chi_{12}-\chi_{13}-\chi_{23}) \hat{\Phi}_{1,N}^n - \frac{\varepsilon_2^2}{36} \frac{| \nabla \Phi_{2,N}^{n+1} |^2}{ ( \Phi_{2,N}^{n+1} )^2}
	- \frac{\varepsilon_2^2}{18} \nabla \cdot \Bigl( \frac{\nabla \Phi_{2,N}^{n+1}}{ \Phi_{2,N}^{n+1} } \Bigr)
   \nonumber
	\\
	& 
	+ \frac{\varepsilon_3^2}{36} \Bigl( \frac{ | \nabla ( 1 - \Phi_{1,N}^{n+1} - \Phi_{2,N}^{n+1} ) |^2 }{ (1 - \Phi_{1,N}^{n+1} - \Phi_{2,N}^{n+1} )^2}  \Bigr) + \frac{\varepsilon_3^2}{18} \nabla \cdot \Bigl( \frac{\nabla (1 - \Phi_{1,N}^{n+1} - \Phi_{2,N}^{n+1} ) }{ 1 - \Phi_{1,N}^{n+1} - \Phi_{2,N}^{n+1}  } \Bigr)  \nonumber\\
	&	+ A_2 \dt\Delta(\Phi_{2,N}^{n+1} - \Phi_{2,N}^n) \biggr)
	+ \dt^2 \g_2^{(0)} + O(\dt^3)  + O (h^m) .
	\label{consistency-2-2}
\end{align}
with the projection accuracy order $m\ge 4$. In fact, the spatial functions $\g_1^{(0)}, \g_2^{(0)}$ are smooth enough in the sense that their derivatives are bounded.

Subsequently, the leading order temporal correction function $\Phi_{j,\dt}$ turns out to be the solution of the following equations:
\begin{small}
\begin{align}
	&\partial_t \Phi_{1,\dt}\nonumber\\
	 =&  	
	\mathcal{M}_1 \Delta \biggl(  \frac{1}{M_0} \frac{\Phi_{1,\dt}}{\Phi_{1,N}} + \frac{\Phi_{1,\dt} + \Phi_{2,\dt} }{1-\Phi_{1,N} - \Phi_{2,N}} - 2 \chi_{13} \Phi_{1,\dt} +(\chi_{12}-\chi_{13}-\chi_{23}) \Phi_{2,\dt}  \nonumber
	\\
	& 
	+ \frac{\varepsilon_1^2}{36} \frac{2 | \nabla \Phi_{1,N} |^2 \Phi_{1,\dt} }{ \Phi_{1,N}^3}
	- \frac{\varepsilon_1^2}{36} \frac{2 \nabla \Phi_{1,N} \cdot \nabla \Phi_{1,\dt}}{ \Phi_{1,N}^2}
	- \frac{\varepsilon_1^2}{18} \nabla \cdot \Bigl( \frac{\nabla \Phi_{1,\dt}}{ \Phi_{1,N} } - \frac{\Phi_{1,\dt} \nabla \Phi_{1,N}}{ \Phi_{1,N}^2 } \Bigr)   \nonumber
	\\
	& 
	+ \frac{\varepsilon_3^2}{36} \Bigl( \frac{ 2 | \nabla ( 1 - \Phi_{1,N} - \Phi_{2,N} ) |^2 ( \Phi_{1, \dt} + \Phi_{2,\dt} ) }{ (1 - \Phi_{1,N} - \Phi_{2,N} )^3}  \nonumber
	\\
	& \quad \quad
	- \frac{ 2 \nabla ( 1 - \Phi_{1,N} - \Phi_{2,N} ) \cdot \nabla ( \Phi_{1,\dt} + \Phi_{2,\dt} ) }{ (1 - \Phi_{1,N} - \Phi_{2,N} )^2}  \Bigr)   \nonumber
	\\
	& 
	+ \frac{\varepsilon_3^2}{18} \nabla \cdot \left( \frac{ - \nabla ( \Phi_{1,\dt} + \Phi_{2,\dt} ) }{ 1 - \Phi_{1,N} - \Phi_{2,N} }
	+  \frac{ ( \Phi_{1,\dt} + \Phi_{2,\dt} ) \nabla (1 - \Phi_{1,N} - \Phi_{2,N} ) }{ (1 - \Phi_{1,N} - \Phi_{2,N} )^2 }  \right)  \biggr)
	- \g_1^{(0)} .
	\label{consistency-3-1}
\end{align}
\end{small}
\begin{small}
\begin{align}
	&\partial_t \Phi_{2,\dt} \nonumber\\
	=&  	
	\mathcal{M}_2 \Delta \biggl(  \frac{1}{N_0} \frac{\Phi_{2,\dt}}{\Phi_{2,N}} + \frac{\Phi_{1,\dt} + \Phi_{2,\dt}}{1-\Phi_{1,N} - \Phi_{2,N}} - 2 \chi_{23} \Phi_{2,\dt} +(\chi_{12}-\chi_{13}-\chi_{23}) \Phi_{1,\dt}  \nonumber
	\\
	& 
	+ \frac{\varepsilon_1^2}{36} \frac{2 | \nabla \Phi_{2,N} |^2 \Phi_{2,\dt} }{ \Phi_{2,N}^3}
	- \frac{\varepsilon_1^2}{36} \frac{2 \nabla \Phi_{2,N} \cdot \nabla \Phi_{2,\dt}}{ \Phi_{2,N}^2}
	- \frac{\varepsilon_1^2}{18} \nabla \cdot \Bigl( \frac{\nabla \Phi_{2,\dt}}{ \Phi_{2,N} } - \frac{\Phi_{2,\dt} \nabla \Phi_{2,N}}{ \Phi_{2,N}^2 } \Bigr)   \nonumber
	\\
	& 
	+ \frac{\varepsilon_3^2}{36} \Bigl( \frac{ 2 | \nabla ( 1 - \Phi_{1,N} - \Phi_{2,N} ) |^2 ( \Phi_{1, \dt} + \Phi_{2,\dt} ) }{ (1 - \Phi_{1,N} - \Phi_{2,N} )^3}  \nonumber
	\\
	& \quad \quad
	- \frac{ 2 \nabla ( 1 - \Phi_{1,N} - \Phi_{2,N} ) \cdot \nabla ( \Phi_{1,\dt} + \Phi_{2,\dt} ) }{ (1 - \Phi_{1,N} - \Phi_{2,N} )^2}  \Bigr)   \nonumber
	\\
	&
	+ \frac{\varepsilon_3^2}{18} \nabla \cdot \left( \frac{ - \nabla ( \Phi_{1,\dt} + \Phi_{2,\dt} ) }{ 1 - \Phi_{1,N} - \Phi_{2,N} }
	+  \frac{ ( \Phi_{1,\dt} + \Phi_{2,\dt} ) \nabla (1 - \Phi_{1,N} - \Phi_{2,N} ) }{ (1 - \Phi_{1,N} - \Phi_{2,N} )^2 }  \right)  \biggr)
	- \g_2^{(0)}  .
	\label{consistency-3-2}
\end{align}
\end{small}
Initial data $\Phi_{j, \dt} ( \cdot, t=0) \equiv 0$. In fact, existence of a solution of the above linear PDE system is straightforward . It depends only on the projection solution $\Phi_{j,N}$. And also, the derivatives of $\Phi_{j,\dt}$ in various orders are bounded. In turn,  when $n \geq 1$, an application of the semi-implicit discretization to \eqref{consistency-3-1}-\eqref{consistency-3-2} gives
\begin{small}
\begin{align}
	&\frac{3\Phi_{1,\dt}^{n+1} - 4\Phi_{1,\dt}^n + \Phi_{1,\dt}^{n-1}}{2\dt} \nonumber\\
	= 	
	&\mathcal{M}_1 \Delta \biggl(  \frac{1}{M_0} \frac{\Phi_{1,\dt}^{n+1}}{\Phi_{1,N}^{n+1}} + \frac{\Phi_{1,\dt}^{n+1} + \Phi_{2,\dt}^{n+1} }{1-\Phi_{1,N}^{n+1} - \Phi_{2,N}^{n+1} } - 2 \chi_{13} (2\Phi_{1,\dt}^n - \Phi_{1,\dt}^{n-1}) \nonumber
	\\
	& \quad
	+(\chi_{12}-\chi_{13}-\chi_{23}) (2\Phi_{2,\dt}^n - \Phi_{2,\dt}^{n-1} ) + \frac{\varepsilon_1^2}{36} \frac{2 | \nabla \Phi_{1,N}^{n+1} |^2 \Phi_{1,\dt}^{n+1} }{ ( \Phi_{1,N}^{n+1} )^3}
	- \frac{\varepsilon_1^2}{36} \frac{2 \nabla \Phi_{1,N}^{n+1} \cdot \nabla \Phi_{1,\dt}^{n+1} }{ ( \Phi_{1,N}^{n+1} )^2}
 \nonumber
	\\
	& \quad
	- \frac{\varepsilon_1^2}{18} \nabla \cdot \Bigl( \frac{\nabla \Phi_{1,\dt}^{n+1} }{ \Phi_{1,N}^{n+1} } - \frac{\Phi_{1,\dt}^{n+1} \nabla \Phi_{1,N}^{n+1} }{ ( \Phi_{1,N}^{n+1} )^2 } \Bigr)
	+ A_1\dt \Delta(\Phi_{1,\dt}^{n+1}-\Phi_{1,\dt}^n)\nonumber\\
	&\quad	
	+ \frac{\varepsilon_3^2}{36} \Bigl( \frac{ 2 | \nabla ( 1 - \Phi_{1,N}^{n+1} - \Phi_{2,N}^{n+1} ) |^2 ( \Phi_{1, \dt}^{n+1} + \Phi_{2,\dt}^{n+1} ) }{ (1 - \Phi_{1,N}^{n+1} - \Phi_{2,N}^{n+1} )^3}  
	- \frac{ 2 \nabla ( 1 - \Phi_{1,N}^{n+1} - \Phi_{2,N}^{n+1} ) \cdot \nabla ( \Phi_{1,\dt}^{n+1} + \Phi_{2,\dt}^{n+1} ) }{ (1 - \Phi_{1,N}^{n+1} - \Phi_{2,N}^{n+1} )^2}  \Bigr)   \nonumber
	\\
	& \quad
	+ \frac{\varepsilon_3^2}{18} \nabla \cdot \Bigl( \frac{ - \nabla  ( \Phi_{1,\dt}^{n+1} + \Phi_{2,\dt}^{n+1} ) }{ 1 - \Phi_{1,N}^{n+1} - \Phi_{2,N}^{n+1} } 
	+  \frac{ ( \Phi_{1,\dt}^{n+1} + \Phi_{2,\dt}^{n+1} ) \nabla (1 - \Phi_{1,N}^{n+1} - \Phi_{2,N}^{n+1} ) }{ (1 - \Phi_{1,N}^{n+1} - \Phi_{2,N}^{n+1} )^2 }  \Bigr) \biggr)\nonumber\\
	&- \g_1^{(0)}  + \dt^2 \hh_1^n + O(\dt^3) ,\label{consistency-4-1}
	\\
	&\frac{3\Phi_{2,\dt}^{n+1} - 4\Phi_{2,\dt}^n + \Phi_{2,\dt}^{n-1} }{2\dt} \nonumber\\	
	=
	&\mathcal{M}_2 \Delta \biggl(  \frac{1}{N_0} \frac{\Phi_{2,\dt}^{n+1} }{\Phi_{2,N}^{n+1}} + \frac{ \Phi_{1,\dt}^{n+1} + \Phi_{2,\dt}^{n+1} }{1-\Phi_{1,N}^{n+1} - \Phi_{2,N}^{n+1} } - 2 \chi_{23} (2\Phi_{2,\dt}^n - \Phi_{2,\dt}^{n-1})  \nonumber
	\\
	& \quad
	+(\chi_{12}-\chi_{13}-\chi_{23}) (2\Phi_{1,\dt}^n - \Phi_{1,\dt}^{n-1})  + \frac{\varepsilon_1^2}{36} \frac{2 | \nabla \Phi_{2,N}^{n+1} |^2 \Phi_{2,\dt}^{n+1} }{ ( \Phi_{2,N}^{n+1} )^3}
	- \frac{\varepsilon_1^2}{36} \frac{2 \nabla \Phi_{2,N}^{n+1} \cdot \nabla \Phi_{2,\dt}^{n+1} }{ ( \Phi_{2,N}^{n+1} )^2}
   \nonumber
	\\
	& \quad
	- \frac{\varepsilon_1^2}{18} \nabla \cdot \Bigl( \frac{\nabla \Phi_{2,\dt}^{n+1} }{ \Phi_{2,N}^{n+1} } - \frac{\Phi_{2,\dt}^{n+1} \nabla \Phi_{2,N}^{n+1} }{ ( \Phi_{2,N}^{n+1} )^2 } \Bigr) + A_2 \dt \Delta(\Phi_{2,\dt}^{n+1} - \Phi_{2,\dt}^n)
	\nonumber\\
	&\quad		
	+ \frac{\varepsilon_3^2}{36} \Bigl( \frac{ 2 | \nabla ( 1 - \Phi_{1,N}^{n+1} - \Phi_{2,N}^{n+1} ) |^2 ( \Phi_{1, \dt}^{n+1} + \Phi_{2,\dt}^{n+1} ) }{ (1 - \Phi_{1,N}^{n+1} - \Phi_{2,N}^{n+1} )^3}  
	- \frac{ 2 \nabla ( 1 - \Phi_{1,N}^{n+1} - \Phi_{2,N}^{n+1} ) \cdot \nabla ( \Phi_{1,\dt}^{n+1} + \Phi_{2,\dt}^{n+1} ) }{ (1 - \Phi_{1,N}^{n+1} - \Phi_{2,N}^{n+1} )^2}  \Bigr)   \nonumber
	\\
	& \quad
	+ \frac{\varepsilon_3^2}{18} \nabla \cdot \Bigl( \frac{ - \nabla  ( \Phi_{1,\dt}^{n+1} + \Phi_{2,\dt}^{n+1} ) }{ 1 - \Phi_{1,N}^{n+1} - \Phi_{2,N}^{n+1} }  
	+  \frac{ ( \Phi_{1,\dt}^{n+1} + \Phi_{2,\dt}^{n+1} ) \nabla (1 - \Phi_{1,N}^{n+1} - \Phi_{2,N}^{n+1} ) }{ (1 - \Phi_{1,N}^{n+1} - \Phi_{2,N}^{n+1} )^2 }  \Bigr)  \biggr)\nonumber\\
	&
	- \g_2^{(0)}  + \dt^2 \hh_2^n + O(\dt^3) .  \label{consistency-4-2}
\end{align}
\end{small}
The initial data $\Phi_{j,\dt}^0 = \Phi_{j,\dt}^1 \equiv 0$.
A combination of \eqref{consistency-2-1}-\eqref{consistency-2-2} and \eqref{consistency-4-1}-\eqref{consistency-4-2} results in the third order temporal truncation error for $	 \check{\Phi}_{j,1} = \Phi_{j,N} +\dt^2 {\cal P}_N \Phi_{j,\dt},\quad j =1, 2 $, when $n\geq 1$:
\begin{align}
	&\frac{3\check{\Phi}_{1,1}^{n+1} - 4\check{\Phi}_{1,1}^n+\check{\Phi}_{1,1}^{n-1}}{2\dt} \nonumber\\
	=&  	
	\mathcal{M}_1 \Delta \biggl(  \frac{1}{M_0}\ln \frac{\alpha \check{\Phi}_{1,1}^{n+1}}{M_0} - \ln(1- \check{\Phi}_{1,1}^{n+1} - \check{\Phi}_{2,1}^{n+1}) - 2 \chi_{13} (2\check{\Phi}_{1,1}^n-\check{\Phi}_{1,1}^{n-1})  \nonumber
	\\
	&\quad
	+(\chi_{12}-\chi_{13}-\chi_{23}) (2\check{\Phi}_{2,1}^n-\check{\Phi}_{2,1}^{n-1}) - \frac{\varepsilon_1^2}{36} 
	 \frac{| \nabla \check{\Phi}_{1,1}^{n+1} |^2}{ ( \check{\Phi}_{1,1}^{n+1} )^2} 
	\nonumber
	\\
	&\quad
	- \frac{\varepsilon_1^2}{18} \nabla \cdot \Bigl( \frac{\nabla \check{\Phi}_{1,1}^{n+1}}{ \check{\Phi}_{1,1}^{n+1} } \Bigr) + \frac{\varepsilon_3^2}{36}  \frac{| \nabla ( 1 - \check{\Phi}_{1,1}^{n+1} - \check{\Phi}_{2,1}^{n+1}) |^2}{ (1 - \check{\Phi}_{1,1}^{n+1} - \check{\Phi}_{2,1}^{n+1} )^2}  
	\nonumber\\
	&\quad  
	+ \frac{\varepsilon_3^2}{18} \nabla \cdot \Bigl( \frac{\nabla (1 - \check{\Phi}_{1,1}^{n+1} - \check{\Phi}_{2,1}^{n+1}) }{1 - \check{\Phi}_{1,1}^{n+1} - \check{\Phi}_{2,1}^{n+1} } \Bigr) + A_1\Delta t \Delta(\check{\Phi}_{1,1}^{n+1} - \check{\Phi}_{1,1}^n ) \biggr) \nonumber\\
	&+ O(\Delta t^3) + O(h^m) ,
	\label{consistency-5-1}
	\\
	&\frac{3\check{\Phi}_{2,1}^{n+1} - 4\check{\Phi}_{2,1}^n+\check{\Phi}_{2,1}^{n-1}}{2\dt} \nonumber\\
	=&  	
	\mathcal{M}_2 \Delta \biggl(  \frac{1}{N_0} \ln \frac{\beta \check{\Phi}_{2,1}^{n+1}}{N_0} - \ln(1- \check{\Phi}_{1,1}^{n+1} - \check{\Phi}_{2,1}^{n+1}) - 2 \chi_{23} (2\check{\Phi}_{2,1}^n-\check{\Phi}_{2,1}^{n-1})   \nonumber
	\\
	&\quad
	+(\chi_{12}-\chi_{13}-\chi_{23}) (2\check{\Phi}_{1,1}^n-\check{\Phi}_{1,1}^{n-1}) - \frac{\varepsilon_2^2}{36}  \frac{| \nabla \check{\Phi}_{2,1}^{n+1} |^2}{ (  \check{\Phi}_{2,1}^{n+1} )^2}  
	\nonumber
	\\
	&\quad
	- \frac{\varepsilon_2^2}{18} \nabla \cdot \Bigl( \frac{\nabla \check{\Phi}_{2,1}^{n+1}}{ \check{\Phi}_{2,1}^{n+1} } \Bigr) + \frac{\varepsilon_3^2}{36}  \frac{| \nabla ( 1 - \check{\Phi}_{1,1}^{n+1} - \check{\Phi}_{2,1}^{n+1}) |^2}{  (1 - \check{\Phi}_{1,1}^{n+1} - \check{\Phi}_{2,1}^{n+1} )^2} 
	\nonumber\\
	&  \quad 
	+ \frac{\varepsilon_3^2}{18} \nabla \cdot \Bigl( \frac{\nabla (1 - \check{\Phi}_{1,1}^{n+1} - \check{\Phi}_{2,1}^{n+1}) }{ 1 - \check{\Phi}_{1,1}^{n+1} - \check{\Phi}_{2,1}^{n+1} } \Bigr) + A_2\Delta t \Delta(\check{\Phi}_{2,1}^{n+1} - \check{\Phi}_{2,1}^n ) \biggr)\nonumber\\
	& + O(\Delta t^3) + O(h^m) ,
	\label{consistency-5-2}
\end{align}
in which the initial data is $\check{\Phi}_{j,1}^0 = \Phi_{j,N}^0$, $\check{\Phi}_{j,1}^1 = \Phi_{j,N}^1$. In the derivation of \eqref{consistency-5-1}-\eqref{consistency-5-2}, the following linearized expansions have been utilized
	\begin{eqnarray*}
	&&
	\ln \check{\Phi}_{j,1} = \ln ( \Phi_{j,N} + \dt^2 \Phi_{j,\dt} )
	= \ln \Phi_{j,N} + \frac{\dt^2 \Phi_{j,\dt}}{\Phi_{j,N}} + O (\dt^4) ,   \quad j=1, 2 ,
	\label{consistency-6-1}	
	\\
	&&
	\ln ( 1- \check{\Phi}_{1,1} - \check{\Phi}_{2,1} ) = \ln ( 1 - \Phi_{1,N} - \Phi_{2,N} - \dt^2 \Phi_{1,\dt} - \dt^2 \Phi_{2,\dt} )  \nonumber
	\\
	&& \quad
	= \ln ( 1 - \Phi_{1,N} - \Phi_{2,N} ) - \dt^2 \frac{\Phi_{1,\dt} + \Phi_{2,\dt} }{ 1 - \Phi_{1,N} - \Phi_{2,N} }  + O (\dt^4) ,
	\label{consistency-6-2}
	\\
	&&
	\frac{| \nabla \check{\Phi}_{j,1} |^2}{ \check{\Phi}_{j,1}^2}
	= \frac{| \nabla ( \Phi_{j,N} + \dt^2 \Phi_{j,\dt} ) |^2}{ ( \Phi_{j,N} + \dt^2 \Phi_{j,\dt} )^2}
	\nonumber
	\\
	&&  \quad
	= \frac{| \nabla \Phi_{j,N} |^2}{ \Phi_{j,N}^2}
	- 2 \dt^2 \frac{| \nabla \Phi_{j,N} |^2 \Phi_{j,\dt}}{ \Phi_{j,N}^3}
	+ 2 \dt^2 \frac{ \nabla \Phi_{j,N} \cdot \nabla \Phi_{j,\dt} }{ \Phi_{j,N}^2}
	+ O (\dt^4) ,   \quad j=1, 2,
	\label{consistency-6-3}
	\\
	&&
	\frac{\nabla \check{\Phi}_{j,1} }{ \check{\Phi}_{j,1} }
	= \frac{\nabla ( \Phi_{j,N} + \dt^2 \Phi_{j,\dt} ) }{ \Phi_{j,N} + \dt^2 \Phi_{j,\dt} }
	= \frac{\nabla \Phi_{j,N} }{ \Phi_{j,N} }
	+ \dt^2 \frac{\nabla \Phi_{j,\dt} }{ \Phi_{j,N} }
	- \dt^2 \frac{ \Phi_{j,\dt} \nabla \Phi_{j,N} }{ \Phi_{j,N}^2 }  + O (\dt^4) ,
	\label{consistency-6-4}
	\\
	&&
	\frac{ | \nabla ( 1 - \check{\Phi}_{1,1} - \check{\Phi}_{2,1} ) |^2 }{ (1 - \check{\Phi}_{1,1} - \check{\Phi}_{2,1} )^2}
	= \frac{ | \nabla ( 1 - \Phi_{1,N} - \Phi_{2,N} - \dt^2 \Phi_{1,\dt} - \dt^2 \Phi_{2,\dt} ) |^2 }{ (1 - \Phi_{1,N} - \Phi_{2,N} - \dt^2 \Phi_{1,\dt} - \dt^2 \Phi_{2,\dt} )^2}  \nonumber
	\\
	&& \quad
	= \frac{ | \nabla ( 1 - \Phi_{1,N} - \Phi_{2,N} ) |^2 }{ (1 - \Phi_{1,N} - \Phi_{2,N}  )^2}
	- \dt^2 \frac{ 2 \nabla ( 1 - \Phi_{1,N} - \Phi_{2,N}  ) \cdot \nabla ( \Phi_{1,\dt} + \Phi_{2,\dt} )  }{ (1 - \Phi_{1,N} - \Phi_{2,N} )^2}   \nonumber
	\\
	&&  \quad \quad
	+ \frac{ 2 \dt^2 | \nabla ( 1 - \Phi_{1,N} - \Phi_{2,N} )|^2 ( \Phi_{1,\dt} +  \Phi_{2,\dt} ) }{ (1 - \Phi_{1,N}- \Phi_{2,N} )^3}  + O (\dt^4) ,
	\label{consistency-6-5}
	\\
	&&
	\frac{\nabla (1 - \check{\Phi}_{1,1} - \check{\Phi}_{2,1} ) }{ (1 - \check{\Phi}_{1,1} - \check{\Phi}_{2,1} ) }
	= \frac{\nabla ( 1 - \Phi_{1,N} - \Phi_{2,N} - \dt^2 \Phi_{1,\dt} - \dt^2 \Phi_{2,\dt} ) }{ ( 1 - \Phi_{1,N} - \Phi_{2,N} - \dt^2 \Phi_{1,\dt} - \dt^2 \Phi_{2,\dt}  ) }    \nonumber
	\\
	&&  \quad
	= \frac{\nabla ( 1 - \Phi_{1,N} - \Phi_{2,N}  ) }{ 1 - \Phi_{1,N} - \Phi_{2,N}  }
	- \dt^2 \frac{\nabla ( \Phi_{1,\dt} + \Phi_{2,\dt} ) }{1 - \Phi_{1,N} - \Phi_{2,N} }  \nonumber
	\\
	&& \quad \quad
	+  \dt^2 \frac{ ( \Phi_{1,\dt} + \Phi_{2,\dt} ) \nabla ( 1 - \Phi_{1,N} - \Phi_{2,N} ) }{ ( 1 - \Phi_{1,N} - \Phi_{2,N} )^2 }  + O (\dt^4) .
	\label{consistency-6-6}                   	
\end{eqnarray*}
For the sake of representation, the operator ${\cal P}_N$ is omitted from the above formulas. 

In terms of spatial discretization, we construct the spatial correction term $\Phi_{j,h}$ to improve the spatial accuracy order when $n\geq 1$. The following truncation error estimate for the spatial discretization is available, by using a straightforward Taylor expansion for the constructed profile $\check{\Phi}_{j,1}^{n+1}$:
\begin{align}
	&\frac{3\check{\Phi}_{1,1}^{n+1} - 4\check{\Phi}_{1,1}^n+\check{\Phi}_{1,1}^{n-1}}{2\dt} \nonumber\\
	 =& 	
	\mathcal{M}_1 \Delta_h \biggl(  \frac{1}{M_0}\ln \frac{\alpha \check{\Phi}_{1,1}^{n+1}}{M_0} - \ln(1- \check{\Phi}_{1,1}^{n+1} - \check{\Phi}_{2,1}^{n+1}) - 2 \chi_{13} (2\check{\Phi}_{1,1}^n-\check{\Phi}_{1,1}^{n-1})  \nonumber
	\\
	& \quad
	+(\chi_{12}-\chi_{13}-\chi_{23}) (2\check{\Phi}_{2,1}^n-\check{\Phi}_{2,1}^{n-1}) - \frac{\varepsilon_1^2}{36} {\cal A}_h \Bigl( \frac{| \nabla_h \check{\Phi}_{1,1}^{n+1} |^2}{ ( {\cal A}_h \check{\Phi}_{1,1}^{n+1} )^2}  \Bigr)
	\nonumber
	\\
	& \quad
	- \frac{\varepsilon_1^2}{18} \nabla_h \cdot \Bigl( \frac{\nabla_h \check{\Phi}_{1,1}^{n+1}}{{\cal A}_h \check{\Phi}_{1,1}^{n+1} } \Bigr) + \frac{\varepsilon_3^2}{36} {\cal A}_h \Bigl( \frac{| \nabla_h ( 1 - \check{\Phi}_{1,1}^{n+1} - \check{\Phi}_{2,1}^{n+1}) |^2}{ ( {\cal A}_h (1 - \check{\Phi}_{1,1}^{n+1} - \check{\Phi}_{2,1}^{n+1} ) )^2}  \Bigr) 
	\nonumber\\
	& \quad 
	+ \frac{\varepsilon_3^2}{18} \nabla_h \cdot \Bigl( \frac{\nabla_h (1 - \check{\Phi}_{1,1}^{n+1} - \check{\Phi}_{2,1}^{n+1}) }{{\cal A}_h (1 - \check{\Phi}_{1,1}^{n+1} - \check{\Phi}_{2,1}^{n+1}) } \Bigr) + A_1\Delta t \Delta_h(\check{\Phi}_{1,1}^{n+1} - \check{\Phi}_{1,1}^n ) \biggr) \nonumber\\
	& 
	+ h^2\HH_1^{(0)} + O(\Delta t^3 + h^4) ,
	\label{consistency-7-1}
	\\
	&\frac{3\check{\Phi}_{2,1}^{n+1} - 4\check{\Phi}_{2,1}^n+\check{\Phi}_{2,1}^{n-1}}{2\dt} \nonumber\\
	=&  	
	\mathcal{M}_2 \Delta_h \biggl(  \frac{1}{N_0} \ln \frac{\beta \check{\Phi}_{2,1}^{n+1}}{N_0} - \ln(1- \check{\Phi}_{1,1}^{n+1} - \check{\Phi}_{2,1}^{n+1}) - 2 \chi_{23} (2\check{\Phi}_{2,1}^n-\check{\Phi}_{2,1}^{n-1})   \nonumber
	\\
	& \quad
	+(\chi_{12}-\chi_{13}-\chi_{23}) (2\check{\Phi}_{1,1}^n-\check{\Phi}_{1,1}^{n-1}) - \frac{\varepsilon_2^2}{36} {\cal A}_h \Bigl( \frac{| \nabla_h \check{\Phi}_{2,1}^{n+1} |^2}{ ( {\cal A}_h \check{\Phi}_{2,1}^{n+1} )^2}  \Bigr)
	\nonumber
	\\
	& \quad
	- \frac{\varepsilon_2^2}{18} \nabla_h \cdot \Bigl( \frac{\nabla_h \check{\Phi}_{2,1}^{n+1}}{{\cal A}_h \check{\Phi}_{2,1}^{n+1} } \Bigr) + \frac{\varepsilon_3^2}{36} {\cal A}_h \Bigl( \frac{| \nabla_h ( 1 - \check{\Phi}_{1,1}^{n+1} - \check{\Phi}_{2,1}^{n+1}) |^2}{ ( {\cal A}_h (1 - \check{\Phi}_{1,1}^{n+1} - \check{\Phi}_{2,1}^{n+1} ) )^2}  \Bigr)
	\nonumber\\
	&  \quad 
	+ \frac{\varepsilon_3^2}{18} \nabla_h \cdot \Bigl( \frac{\nabla_h (1 - \check{\Phi}_{1,1}^{n+1} - \check{\Phi}_{2,1}^{n+1}) }{{\cal A}_h (1 - \check{\Phi}_{1,1}^{n+1} - \check{\Phi}_{2,1}^{n+1}) } \Bigr) + A_2\Delta t \Delta_h(\check{\Phi}_{2,1}^{n+1} - \check{\Phi}_{2,1}^n ) \biggr) \nonumber\\
	& 
	+ h^2\HH_2^{(0)} + O(\Delta t^3 + h^4).
	\label{consistency-7-2}
\end{align}
Similarly, the spatially discrete functions $\HH_j^{(0)}$ are smooth enough in the sense that their derivatives are bounded. Because of the symmetry in the centered finite difference approximation, there is no $O(h^3)$ truncation error term. In turn, the spatial correction function $\Phi_{j,h}$ is determined by solving the following linear PDE system:
\begin{small}
\begin{align}
	&\partial_t \Phi_{1,h} \nonumber\\
	=&  	
	\mathcal{M}_1 \Delta \biggl(  \frac{1}{M_0} \frac{\Phi_{1,h}}{\check{\Phi}_{1,1}} + \frac{\Phi_{1,h} + \Phi_{2,h} }{1-\check{\Phi}_{1,1} - \check{\Phi}_{2,1}} - 2 \chi_{13} \Phi_{1,h} +(\chi_{12}-\chi_{13}-\chi_{23}) \Phi_{2,h}  \nonumber
	\\
	& \quad
	+ \frac{\varepsilon_1^2}{36} \frac{2 | \nabla \check{\Phi}_{1,1} |^2 \Phi_{1,h} }{ \check{\Phi}_{1,1}^3}
	- \frac{\varepsilon_1^2}{36} \frac{2 \nabla \check{\Phi}_{1,1} \cdot \nabla \Phi_{1,h}}{ \check{\Phi}_{1,1}^2}
	- \frac{\varepsilon_1^2}{18} \nabla \cdot \Bigl( \frac{\nabla \Phi_{1,h}}{ \check{\Phi}_{1,1} } - \frac{\Phi_{1,h} \nabla \check{\Phi}_{1,1}}{ \check{\Phi}_{1,1}^2 } \Bigr)   \nonumber
	\\
	& \quad
	+ \frac{\varepsilon_3^2}{36} \Bigl( \frac{ 2 | \nabla ( 1 - \check{\Phi}_{1,1} - \check{\Phi}_{2,1} ) |^2 ( \Phi_{1, h} + \Phi_{2,h} ) }{ (1 - \check{\Phi}_{1,1} - \check{\Phi}_{2,1} )^3}  \nonumber
	\\
	& \quad \quad \quad
	- \frac{ 2 \nabla ( 1 - \check{\Phi}_{1,1} - \check{\Phi}_{2,1} ) \cdot \nabla ( \Phi_{1,h} + \Phi_{2,h} ) }{ (1 - \check{\Phi}_{1,1} - \check{\Phi}_{2,1} )^2}  \Bigr)   \nonumber
	\\
	& \quad
	+ \frac{\varepsilon_3^2}{18} \nabla \cdot \left( \frac{ - \nabla ( \Phi_{1,h} + \Phi_{2,h} ) }{ 1 - \check{\Phi}_{1,1} - \check{\Phi}_{2,1} }
	+  \frac{ ( \Phi_{1,h} + \Phi_{2,h} ) \nabla (1 - \check{\Phi}_{1,1} - \check{\Phi}_{2,1} ) }{ (1 - \check{\Phi}_{1,1} - \check{\Phi}_{2,1} )^2 }  \right)  \biggr)
	- \HH_1^{(0)} .
	\label{consistency-8-1}
\end{align}
\end{small}
\begin{small}
\begin{align}
	&\partial_t \Phi_{2,h} \nonumber\\
	=&  	
	\mathcal{M}_2 \Delta \biggl(  \frac{1}{N_0} \frac{\Phi_{2,h}}{\check{\Phi}_{2,1}} + \frac{\Phi_{1,h} + \Phi_{2,h}}{1-\check{\Phi}_{1,1} - \check{\Phi}_{2,1}} - 2 \chi_{23} \Phi_{2,h} +(\chi_{12}-\chi_{13}-\chi_{23}) \Phi_{1,h}  \nonumber
	\\
	& \quad
	+ \frac{\varepsilon_1^2}{36} \frac{2 | \nabla \check{\Phi}_{2,1} |^2 \Phi_{2,h} }{ \check{\Phi}_{2,1}^3}
	- \frac{\varepsilon_1^2}{36} \frac{2 \nabla \check{\Phi}_{2,1} \cdot \nabla \Phi_{2,h}}{ \check{\Phi}_{2,1}^2}
	- \frac{\varepsilon_1^2}{18} \nabla \cdot \Bigl( \frac{\nabla \Phi_{2,h}}{ \check{\Phi}_{2,1} } - \frac{\Phi_{2,h} \nabla \check{\Phi}_{2,1}}{ \check{\Phi}_{2,1}^2 } \Bigr)   \nonumber
	\\
	& \quad
	+ \frac{\varepsilon_3^2}{36} \Bigl( \frac{ 2 | \nabla ( 1 - \check{\Phi}_{1,1} - \check{\Phi}_{2,1} ) |^2 ( \Phi_{1, h} + \Phi_{2,h} ) }{ (1 - \check{\Phi}_{1,1} - \check{\Phi}_{2,1} )^3}  \nonumber
	\\
	& \quad \quad \quad
	- \frac{ 2 \nabla ( 1 - \check{\Phi}_{1,1} - \check{\Phi}_{2,1} ) \cdot \nabla ( \Phi_{1,h} + \Phi_{2,h} ) }{ (1 - \check{\Phi}_{1,1} - \check{\Phi}_{2,1} )^2}  \Bigr)   \nonumber
	\\
	& \quad
	+ \frac{\varepsilon_3^2}{18} \nabla \cdot \left( \frac{ - \nabla ( \Phi_{1,h} + \Phi_{2,h} ) }{ 1 - \check{\Phi}_{1,1} - \check{\Phi}_{2,1} }
	+  \frac{ ( \Phi_{1,h} + \Phi_{2,h} ) \nabla (1 - \check{\Phi}_{1,1} - \check{\Phi}_{2,1} ) }{ (1 - \check{\Phi}_{1,1} - \check{\Phi}_{2,1} )^2 }  \right)  \biggr)
	- \HH_2^{(0)} ,
	\label{consistency-8-2}
\end{align}
\end{small}
in which the initial data $\Phi_{j,h}(\cdot,t=0)=0$. Again, the solution depends only on the exact solution $\Phi_j$, with the divided differences of various orders stay bounded. In turn an application of a full discretization to \eqref{consistency-8-1}-\eqref{consistency-8-2} leads to 
\begin{small}
\begin{align}
	&\frac{3\Phi_{1,h}^{n+1} - 4\Phi_{1,h}^n + \Phi_{1,h}^{n-1}}{2\dt} \nonumber\\
	= 	
	&\mathcal{M}_1 \Delta_h \biggl(  \frac{1}{M_0} \frac{\Phi_{1,h}^{n+1}}{\check{\Phi}_{1,1}^{n+1}} + \frac{\Phi_{1,h}^{n+1} + \Phi_{2,h}^{n+1} }{1-\check{\Phi}_{1,1}^{n+1} - \check{\Phi}_{2,1}^{n+1} } - 2 \chi_{13} (2\Phi_{1,h}^n - \Phi_{1,h}^{n-1}) \nonumber
	\\
	& \quad
	+(\chi_{12}-\chi_{13}-\chi_{23}) (2\Phi_{2,h}^n - \Phi_{2,h}^{n-1} ) + \frac{\varepsilon_1^2}{36} \frac{2 | \nabla_h \check{\Phi}_{1,1}^{n+1} |^2 \Phi_{1,h}^{n+1} }{ ( \check{\Phi}_{1,1}^{n+1} )^3}	\nonumber
	\\
	& \quad - \frac{\varepsilon_1^2}{36} \frac{2 \nabla_h \check{\Phi}_{1,1}^{n+1} \cdot \nabla_h \Phi_{1,h}^{n+1} }{ ( \check{\Phi}_{1,1}^{n+1} )^2}
	- \frac{\varepsilon_1^2}{18} \nabla_h \cdot \Bigl( \frac{\nabla_h \Phi_{1,h}^{n+1} }{ \check{\Phi}_{1,1}^{n+1} } - \frac{\Phi_{1,h}^{n+1} \nabla_h \check{\Phi}_{1,1}^{n+1} }{ ( \check{\Phi}_{1,1}^{n+1} )^2 } \Bigr)
	\nonumber\\
	&\quad	
	+ \frac{\varepsilon_3^2}{36}  \frac{ 2 | \nabla_h ( 1 - \check{\Phi}_{1,1}^{n+1} - \check{\Phi}_{2,1}^{n+1} ) |^2 ( \Phi_{1,h}^{n+1} + \Phi_{2,h}^{n+1} ) }{ (1 - \check{\Phi}_{1,1}^{n+1} - \check{\Phi}_{2,1}^{n+1} )^3} + A_1\dt \Delta_h(\Phi_{1,h}^{n+1}-\Phi_{1,h}^n)\nonumber\\
	 &\quad
	- \frac{\varepsilon_3^2}{36}\frac{ 2 \nabla_h ( 1 - \check{\Phi}_{1,1}^{n+1} - \check{\Phi}_{2,1}^{n+1} ) \cdot \nabla_h ( \Phi_{1,h}^{n+1} + \Phi_{2,h}^{n+1} ) }{ (1 - \check{\Phi}_{1,1}^{n+1} - \check{\Phi}_{2,1}^{n+1} )^2}   \nonumber
	\\
	& \quad
	+ \frac{\varepsilon_3^2}{18} \nabla_h \cdot \Bigl( \frac{ - \nabla_h  ( \Phi_{1,h}^{n+1} + \Phi_{2,h}^{n+1} ) }{ 1 - \check{\Phi}_{1,1}^{n+1} - \check{\Phi}_{2,1}^{n+1} } 
	+  \frac{ ( \Phi_{1,h}^{n+1} + \Phi_{2,h}^{n+1} ) \nabla_h (1 - \check{\Phi}_{1,1}^{n+1} - \check{\Phi}_{2,1}^{n+1} ) }{ (1 - \check{\Phi}_{1,1}^{n+1} - \check{\Phi}_{2,1}^{n+1} )^2 }  \Bigr) \biggr)\nonumber\\
	&
	- \HH_1^{(0)}  +  O(\dt^2 +h^2) ,\label{consistency-9-1}
	\\
	&\frac{3\Phi_{2,h}^{n+1} - 4\Phi_{2,h}^n + \Phi_{2,h}^{n-1} }{2\dt}   \nonumber\\	
	=
	&\mathcal{M}_2 \Delta_h \biggl(  \frac{1}{N_0} \frac{\Phi_{2,h}^{n+1} }{\check{\Phi}_{2,1}^{n+1}} + \frac{ \Phi_{1,h}^{n+1} + \Phi_{2,h}^{n+1} }{1-\check{\Phi}_{1,1}^{n+1} - \check{\Phi}_{2,1}^{n+1} } - 2 \chi_{23} (2\Phi_{2,h}^n - \Phi_{2,h}^{n-1})  \nonumber
	\\
	& \quad
	+(\chi_{12}-\chi_{13}-\chi_{23}) (2\Phi_{1,h}^n - \Phi_{1,h}^{n-1})  + \frac{\varepsilon_1^2}{36} \frac{2 | \nabla_h \check{\Phi}_{2,1}^{n+1} |^2 \Phi_{2,h}^{n+1} }{ ( \check{\Phi}_{2,1}^{n+1} )^3}
	\nonumber
	\\
	& \quad
		- \frac{\varepsilon_1^2}{36} \frac{2 \nabla_h \check{\Phi}_{2,1}^{n+1} \cdot \nabla_h \Phi_{2,h}^{n+1} }{ ( \check{\Phi}_{2,1}^{n+1} )^2}- \frac{\varepsilon_1^2}{18} \nabla_h \cdot \Bigl( \frac{\nabla_h \Phi_{2,h}^{n+1} }{ \check{\Phi}_{2,1}^{n+1} } - \frac{\Phi_{2,h}^{n+1} \nabla_h \check{\Phi}_{2,1}^{n+1} }{ ( \check{\Phi}_{2,1}^{n+1} )^2 } \Bigr)
	\nonumber\\
	&\quad		
	+ \frac{\varepsilon_3^2}{36} \frac{ 2 | \nabla_h ( 1 - \check{\Phi}_{1,1}^{n+1} - \check{\Phi}_{2,1}^{n+1} ) |^2 ( \Phi_{1, h}^{n+1} + \Phi_{2,h}^{n+1} ) }{ (1 - \check{\Phi}_{1,1}^{n+1} - \check{\Phi}_{2,1}^{n+1} )^3}  + A_2 \dt \Delta_h(\Phi_{2,h}^{n+1} - \Phi_{2,h}^n)\nonumber\\
	&\quad 
	- \frac{\varepsilon_3^2}{36}\frac{ 2 \nabla_h ( 1 - \check{\Phi}_{1,1}^{n+1} - \check{\Phi}_{2,1}^{n+1} ) \cdot \nabla_h ( \Phi_{1,h}^{n+1} + \Phi_{2,h}^{n+1} ) }{ (1 - \check{\Phi}_{1,1}^{n+1} - \check{\Phi}_{2,1}^{n+1} )^2}   \nonumber
	\\
	& \quad
	+ \frac{\varepsilon_3^2}{18} \nabla_h \cdot \Bigl( \frac{ - \nabla_h  ( \Phi_{1,h}^{n+1} + \Phi_{2,h}^{n+1} ) }{ 1 - \check{\Phi}_{1,1}^{n+1} - \check{\Phi}_{2,1}^{n+1} }  
	+  \frac{ ( \Phi_{1,h}^{n+1} + \Phi_{2,h}^{n+1} ) \nabla_h (1 - \check{\Phi}_{1,1}^{n+1} - \check{\Phi}_{2,1}^{n+1} ) }{ (1 - \check{\Phi}_{1,1}^{n+1} - \check{\Phi}_{2,1}^{n+1} )^2 }  \Bigr)  \biggr)\nonumber\\
	&
	- \HH_2^{(0)}  + O(\dt^2+h^2) ,\label{consistency-9-2}
\end{align}
\end{small}
in which $\Phi_{j,h}^0 = \Phi_{j,h}^1 = 0$. Finally, a combination of \eqref{consistency-7-1}-\eqref{consistency-7-2} and \eqref{consistency-9-1}-\eqref{consistency-9-2} yields the higher order spatial truncation error for $\check{\Phi}_j$, as given by \eqref{consistency-0-1}-\eqref{consistency-0-2}. Of course, the linear expansions have been extensively utilized.
\end{proof}
\begin{rem}
Trivial initial data $\Phi_{j, \dt} ( \cdot, t=0) = \Phi_{j, h} ( \cdot, t=0) \equiv 0$ are given to $\Phi_{j,\dt}, \Phi_{j,h}$ as \eqref{consistency-3-1}-\eqref{consistency-3-2} and \eqref{consistency-8-1}-\eqref{consistency-8-2}. Therefore, using similar arguments as in~\eqref{mass conserv-1}-\eqref{mass conserv-2}, we conclude that
\begin{eqnarray*}
	\hspace{-0.35in} &&
	\phi_j^0 \equiv  \check{\Phi}_{j,N}^0 ,  \quad
	\overline{\phi_j^k} = \overline{\phi_j^0} ,  \quad \forall\, k \ge 0,
	\\
	\hspace{-0.35in} &&
	\overline{\check{\Phi}_{j,N}^k}
	= \overline{\phi_j^0} ,  \quad \forall \, k \ge 0,  
	\\
	\hspace{-0.35in} &&
	\overline{\tau_1^{n+1}} = \overline{\tau_2^{n+1}} = 0 ,  \quad \forall n \ge 0, \, \, \,
	j=1, 2. 
\end{eqnarray*}	
\end{rem}
\begin{rem}
	Since the correction function $\Phi_{j,\dt}, \Phi_{j,h}$ is bounded, we recall the separation property~\eqref{assumption:separation} for the exact solution, and obtain a similar property for the constructed profile $\check{\Phi}_{j,N}$:
	\begin{equation}
		\check{\Phi}_{1,N} \ge \delta_0 , \, \, \, \check{\Phi}_{2,N} \ge \delta_0 , \, \, \,
		1 - \check{\Phi}_{1,N} - \check{\Phi}_{2,N} \ge \delta_0 ,  \quad \exists \, \delta_0 > 0 ,
		\label{assumption:separation-2}
	\end{equation}
	in which the projection estimate~\eqref{projection-est-0} has been repeatedly used. Such a uniform bound will be used in the convergence analysis.
	
	In addition, since the correction function $\Phi_{j,\dt}, \Phi_{j,h}$ only depends on $\Phi_{j,N}$ and the exact solution, its $W^{1,\infty}$ norm will stay bounded. In turn, we are able to obtain a discrete $W^{1,\infty}$ bound for the  constructed profile $\check{\Phi}_{j,N}$:
	\begin{equation}
		\| \nabla_h \check{\Phi}_{j,N} \|_\infty \le C^* ,  \quad j=1, 2 .
		\label{assumption:W1-infty bound}
	\end{equation}
\end{rem}
\subsection{A rough error estimate}

Instead of a direct analysis for the error function defined in~\eqref{error function-1}, we introduce an alternate numerical error function:
\begin{equation}
	\tilde{\phi}_1^m := \mathcal{P}_h \check{\Phi}_{1,N}^m - \phi_1^m ,  \, \, \,
	\tilde{\phi}_2^m := \mathcal{P}_h \check{\Phi}_{2,N}^m - \phi_2^m , \quad \forall \ m \in \left\{ 0 ,1 , 2, 3, \cdots \right\} .
	\label{error function-2}
\end{equation}
The advantage of such a numerical error function is associated with its higher order accuracy, which comes from the higher order consistency estimate~\eqref{consistency-0-1}-\eqref{consistency-0-2}. Again, since $\overline{\tilde{\phi}_1^m} = \overline{\tilde{\phi}_2^m} = 0$. Obviously, we have
\begin{equation}
	e_j^0 = \tilde{\phi}_j^0,\quad e_j^1 = \tilde{\phi}_j^1,\quad j = 1,2.\label{error-for-initial}
\end{equation}

When $n\geq 1$, a careful consistent analysis indicates the following truncation error estimate:
\begin{align}
	\frac{3\tilde{\phi}_1^{n+1} - 4\tilde{\phi}_1^n+\tilde{\phi}_1^{n-1}}{2\dt} & = 	
	\mathcal{M}_1 \Delta_h \tilde{\mu}_1^{n+1} + \tau_1^{n+1} , \label{error equation-1}
	\\
	\frac{3\tilde{\phi}_2^{n+1} - 4\tilde{\phi}_2^n+\tilde{\phi}_2^{n-1}}{2\dt} & =	
	\mathcal{M}_2 \Delta_h \tilde{\mu}_2^{n+1} + \tau_2^{n+1} , \label{error equation-2}
\end{align}
with $\|\tau_i^{n+1}\|_{-1,h}\le C(\dt^3+h^4)$, and
\begin{small}
\begin{align}
	\tilde{\mu}_1^{n+1} & = 	
	\frac{1}{M_0} ( \ln \check{\Phi}_{1,N}^{n+1} - \ln \phi_1^{n+1} )
	- \left( \ln(1- \check{\Phi}_{1,N}^{n+1} - \check{\Phi}_{2,N}^{n+1}) - \ln(1-\phi_1^{n+1} -\phi_2^{n+1}) \right) \nonumber
	\\
	& \quad  \quad
	- 2 \chi_{13} \tilde{\hat{\phi}}_1^n +(\chi_{12}-\chi_{13}-\chi_{23}) \tilde{\hat{\phi}}_2^n
	+ \tilde{\mu}_{1,s}^{n+1} + \tilde{\mu}_{3,s}^{n+1} - A_1\dt \Delta_h(\tilde{\phi}_1^{n+1} - \tilde{\phi}_1^n) , \label{error equation-3}
	\\
	\tilde{\mu}_2^{n+1} & = 	
	\frac{1}{N_0} ( \ln \check{\Phi}_{2,N}^{n+1} - \ln \phi_2^{n+1} )
	- \left( \ln(1- \check{\Phi}_{1,N}^{n+1} - \check{\Phi}_{2,N}^{n+1}) - \ln(1-\phi_1^{n+1} -\phi_2^{n+1}) \right) \nonumber
	\\
	& \quad  \quad
	- 2 \chi_{23} \tilde{\hat{\phi}}_2^n +(\chi_{12}-\chi_{13}-\chi_{23}) \tilde{\hat{\phi}}_1^n
	+ \tilde{\mu}_{2,s}^{n+1} + \tilde{\mu}_{3,s}^{n+1} - A_2\dt \Delta_h(\tilde{\phi}_2^{n+1} - \tilde{\phi}_2^n) , \label{error equation-4}
	\\
	\tilde{\mu}_{1,s}^{n+1} & =
	\frac{\varepsilon_1^2}{36} {\cal A}_h \Bigl(  \gamma^{(1)} {\cal A}_h \tilde{\phi}_1^{n+1}
	- \frac{ \nabla_h ( \check{\Phi}_{1,N}^{n+1} + \phi_1^{n+1}) \cdot \nabla_h \tilde{\phi}_1^{n+1} }{ ( {\cal A}_h \check{\Phi}_{1,N}^{n+1} )^2}  \Bigr) \nonumber\\   
	&\quad\quad- \frac{\varepsilon_1^2}{18} \nabla_h \cdot \Bigl( \frac{\nabla_h \tilde{\phi}_1^{n+1}}{{\cal A}_h \phi_1^{n+1} }  - \frac{ \tilde{\phi}_1^{n+1} \nabla_h \check{\Phi}_{1,N}^{n+1}}{{\cal A}_h \check{\Phi}_1^{n+1} {\cal A}_h \check{\Phi}_{1,N}^{n+1} }  \Bigr) , \label{error equation-5}
	\\
	\tilde{\mu}_{2,s}^{n+1} & =
	\frac{\varepsilon_2^2}{36} {\cal A}_h \Bigl(  \gamma^{(2)} {\cal A}_h \tilde{\phi}_2^{n+1}
	- \frac{ \nabla_h ( \check{\Phi}_{2,N}^{n+1} + \phi_2^{n+1}) \cdot \nabla_h \tilde{\phi}_2^{n+1} }{ ( {\cal A}_h \Phi_{2,N}^{n+1} )^2}  \Bigr) \nonumber\\   
	&\quad\quad- \frac{\varepsilon_2^2}{18} \nabla_h \cdot \Bigl( \frac{\nabla_h \tilde{\phi}_2^{n+1}}{{\cal A}_h \phi_2^{n+1} }  - \frac{ \tilde{\phi}_2^{n+1} \nabla_h \check{\Phi}_{2,N}^{n+1}}{{\cal A}_h \check{\Phi}_2^{n+1} {\cal A}_h \check{\Phi}_{2,N}^{n+1} }  \Bigr) , \nonumber
	\\
	\tilde{\mu}_{3,s}^{n+1}  & =
	\frac{\varepsilon_3^2}{36} {\cal A}_h \Bigl( \gamma^{(3)} {\cal A}_h ( \tilde{\phi}_1^{n+1}
	+ \tilde{\phi}_2^{n+1} )  \Bigr)
	- \frac{\varepsilon_3^2}{18} \nabla_h \cdot \Bigl( \frac{\nabla_h ( \tilde{\phi}_1^{n+1} + \tilde{\phi}_2^{n+1}) }{{\cal A}_h (1 - \phi_1^{n+1} - \phi_2^{n+1}) }  \Bigr)
	\nonumber
	\\
	& \quad
	- \frac{\varepsilon_3^2}{36} {\cal A}_h \Bigl( \frac{ \nabla_h (2 - \phi_1^{n+1} - \phi_2^{n+1} - \check{\Phi}_{1,N}^{n+1} - \check{\Phi}_{2,N}^{n+1} ) \cdot \nabla_h ( \tilde{\phi}_1^{n+1} + \tilde{\phi}_2^{n+1} ) }{ ( {\cal A}_h (1 - \check{\Phi}_{1,N}^{n+1} - \check{\Phi}_{2,N}^{n+1} ) )^2 }  \Bigr)   \nonumber
	\\
	& \quad
	- \frac{\varepsilon_3^2}{18} \nabla_h \cdot \Bigl(
	\frac{( \tilde{\phi}_1^{n+1} + \tilde{\phi}_2^{n+1} ) \nabla_h ( 1 - \check{\Phi}_{1,N}^{n+1} - \check{\Phi}_{2,N}^{n+1}) }{ {\cal A}_h (1 - \phi_1^{n+1} - \phi_2^{n+1}) {\cal A}_h (1 - \check{\Phi}_{1,N}^{n+1} - \check{\Phi}_{2,N}^{n+1}) }
	\Bigr)  ,   \nonumber
\end{align}
\end{small}
\begin{align}
	\gamma^{(1)} & = \frac{ {\cal A}_h ( \phi_1^{n+1} + \check{\Phi}_{1,N}^{n+1} )  | \nabla_h \phi_1^{n+1} |^2  }{ ( {\cal A}_h \phi_1^{n+1} )^2 ( {\cal A}_h \check{\Phi}_{1,N}^{n+1} )^2} ,   \label{error equation-8}
	\\
	\gamma^{(2)} & = \frac{ {\cal A}_h ( \phi_2^{n+1} + \check{\Phi}_{2,N}^{n+1} )  | \nabla_h \phi_2^{n+1} |^2  }{ ( {\cal A}_h \phi_2^{n+1} )^2 ( {\cal A}_h \check{\Phi}_{2,N}^{n+1} )^2} ,   \nonumber
	\\
	\gamma^{(3)} & =
	\frac{ {\cal A}_h (2 - \phi_1^{n+1} - \phi_2^{n+1} - \check{\Phi}_{1,N}^{n+1} - \check{\Phi}_{2,N}^{n+1} ) | \nabla_h ( 1 - \phi_1^{n+1} - \phi_2^{n+1}) |^2 }{ ( {\cal A}_h (1 - \phi_1^{n+1} - \phi_2^{n+1} ) )^2 ( {\cal A}_h (1 - \check{\Phi}_{1,N}^{n+1} - \check{\Phi}_{2,N}^{n+1} ) )^2 }  .  \nonumber
\end{align}
To proceed with the nonlinear analysis, we make the following a-priori assumption at the previous time step:
\begin{eqnarray}
	\| \tilde{\phi}_j^k \|_2 \le \dt^\frac74 + h^\frac74 ,  \quad k=n-1,n,\quad j=1, 2 .
	\label{a priori-1}
\end{eqnarray}
Then, based on the fact that $\|f\|_{-1,h}\le C\|f\|_2$, we have
\begin{small}
\begin{eqnarray}
	\| \tilde{\phi}_j^k \|_{-1,h} \le C(\dt^\frac74 + h^\frac74) ,\quad
	\| \nabla_h\tilde{\phi}_j^k \|_2 \le
	C\frac{\| \tilde{\phi}_j^k \|_2}{h}\le  C(\dt^\frac34 + h^\frac34) ,	
	   \ k=n-1,n .
	\label{a priori-2}
\end{eqnarray}
\end{small}
Such an a-priori assumption will be recovered by the optimal rate convergence analysis at the next time step, as will be demonstrated later.
This means that
\begin{eqnarray*}
	\| \tilde{\hat{\phi}}_j^n \|_2^2 = 	\| 2\tilde{\phi}_j^n - \tilde{\phi}_j^{n-1} \|_2^2  \le 6\|\tilde{\phi}_j^n\|_2^2+3\|\tilde{\phi}_j^{n-1}\|_2^2 \le 9C(\dt^\frac72 + h^\frac72) ,  \quad j=1, 2 .
\end{eqnarray*}
Taking a discrete inner product with~\eqref{error equation-1}, \eqref{error equation-2} by $\tilde{\mu}_1^{n+1}$, $\tilde{\mu}_2^{n+1}$, respectively, leads to
\begin{small}
\begin{eqnarray}
	&&
	3\langle \tilde{\phi}_1^{n+1} , \tilde{\mu}_1^{n+1} \rangle
	+ 3\langle \tilde{\phi}_2^{n+1} , \tilde{\mu}_2^{n+1} \rangle
	+ 2\dt ( \mathcal{M}_1 \| \nabla_h \tilde{\mu}_1^{n+1} \|_2^2
	+ \mathcal{M}_2 \| \nabla_h \tilde{\mu}_2^{n+1} \|_2^2  )\nonumber
	\\
	&=&
	\langle 4\tilde{\phi}_1^n-\tilde{\phi}_1^{n-1} , \tilde{\mu}_1^{n+1} \rangle
	+ \langle 4\tilde{\phi}_2^n-\tilde{\phi}_2^{n-1} , \tilde{\mu}_2^{n+1} \rangle
	+ 2\dt  ( \langle \tau_1^{n+1} , \tilde{\mu}_1^{n+1} \rangle
	+ \langle \tau_2^{n+1} , \tilde{\mu}_2^{n+1} \rangle ) .	\label{convergence-rough-1}
\end{eqnarray}
\end{small}
For the two terms $\langle 4\tilde{\phi}_1^n-\tilde{\phi}_1^{n-1} , \tilde{\mu}_1^{n+1} \rangle$ and
$\langle 4\tilde{\phi}_2^n-\tilde{\phi}_2^{n-1} , \tilde{\mu}_2^{n+1} \rangle$ of the right hand side of \eqref{convergence-rough-1}, an application of the Cauchy inequality reveals that
\begin{eqnarray*}
	\langle \tilde{\phi}_j^n , \tilde{\mu}_j^{n+1} \rangle
	\le  \| \tilde{\phi}_j^n \|_{-1,h} \cdot \| \nabla_h \tilde{\mu}_j^{n+1} \|_2
	\le  \frac{2}{ \mathcal{M}_j \dt} \| \tilde{\phi}_j^n \|_{-1,h}^2
	+ \frac{\mathcal{M}_j}{8} \dt \| \nabla_h \tilde{\mu}_j^{n+1} \|_2^2   ,  \quad j=1, 2 ,
\end{eqnarray*}
\begin{eqnarray*}
	|\langle \tilde{\phi}_j^{n-1} , \tilde{\mu}_j^{n+1} \rangle|
	\le  \| \tilde{\phi}_j^{n-1} \|_{-1,h} \cdot \| \nabla_h \tilde{\mu}_j^{n+1} \|_2
	\le  \frac{1}{2 \mathcal{M}_j \dt} \| \tilde{\phi}_j^{n-1} \|_{-1,h}^2
	+ \frac{\mathcal{M}_j}{2} \dt \| \nabla_h \tilde{\mu}_j^{n+1} \|_2^2   ,
\end{eqnarray*}
that means
\begin{small}
\begin{eqnarray}
	\langle 4\tilde{\phi}_j^n-\tilde{\phi}_j^{n-1} , \tilde{\mu}_j^{n+1} \rangle
	\le  \left(\frac{8}{ \mathcal{M}_j \dt} \| \tilde{\phi}_j^n \|_{-1,h}^2
	+ \frac{1}{ 2\mathcal{M}_j \dt} \| \tilde{\phi}_j^{n-1} \|_{-1,h}^2\right)+\mathcal{M}_j \dt \| \nabla_h \tilde{\mu}_j^{n+1} \|_2^2 .
	\nonumber
\end{eqnarray}
\end{small}
For the local truncation error terms, the following estimate is available:
\begin{small}
\begin{eqnarray}
	\langle \tau_j^{n+1} , \tilde{\mu}_j^{n+1} \rangle
	\le \| \tau_j^{n+1} \|_{-1,h} \cdot \| \nabla_h \tilde{\mu}_j^{n+1} \|_2
	\le  \frac{1}{2 \mathcal{M}_j} \| \tau_j^{n+1} \|_{-1,h}^2 + \frac{\mathcal{M}_j}{2} \| \nabla_h \tilde{\mu}_j^{n+1} \|_2^2 .\nonumber
\end{eqnarray}
\end{small}
Going back~\eqref{convergence-rough-1}, we get
\begin{eqnarray}
	3\langle \tilde{\phi}_1^{n+1} , \tilde{\mu}_1^{n+1} \rangle
	+ 3\langle \tilde{\phi}_2^{n+1} , \tilde{\mu}_2^{n+1} \rangle
	\le&&\frac{8}{ \mathcal{M}_* \dt} ( \| \tilde{\phi}_1^n \|_{-1,h}^2
	+ \| \tilde{\phi}_2^n \|_{-1,h}^2  )\nonumber\\
	&&+\frac{1}{ 2\mathcal{M}_* \dt} ( \| \tilde{\phi}_1^{n-1} \|_{-1,h}^2
	+ \| \tilde{\phi}_2^{n-1} \|_{-1,h}^2  )\nonumber\\
	&&+\frac{\dt}{ \mathcal{M}_*} ( \| \tau_1^{n+1} \|_{-1,h}^2 + \| \tau_2^{n+1} \|_{-1,h}^2 )   ,
	\label{convergence-rough-4}
\end{eqnarray}
which $\mathcal{M}_* = \min(\mathcal{M}_1, \mathcal{M}_2)$.
On the other hand, the detailed expansions in~\eqref{error equation-3}-\eqref{error equation-4} reveal the following identity:
\begin{eqnarray}
	&&
	\langle \tilde{\phi}_1^{n+1} , \tilde{\mu}_1^{n+1} \rangle
	+ \langle \tilde{\phi}_2^{n+1} , \tilde{\mu}_2^{n+1} \rangle   \nonumber
	\\
	&=& \frac{1}{M_0} \langle ( \ln \check{\Phi}_{1,N}^{n+1} - \ln \phi_1^{n+1} )  ,
	\tilde{\phi}_1^{n+1}  \rangle
	+ \frac{1}{N_0} \langle ( \ln \check{\Phi}_{2,N}^{n+1} - \ln \phi_2^{n+1} )  ,
	\tilde{\phi}_2^{n+1}  \rangle  \nonumber
	\\
	&&
	- \langle ( \ln(1- \check{\Phi}_{1,N}^{n+1} - \check{\Phi}_{2,N}^{n+1}) - \ln(1-\phi_1^{n+1} -\phi_2^{n+1}) ) ,
	\tilde{\phi}_1^{n+1} + \tilde{\phi}_2^{n+1} \rangle \nonumber
	\\
	&&
	- 2 \chi_{13} \langle \tilde{\hat{\phi}}_1^n , \tilde{\phi}_1^{n+1} \rangle
	- 2 \chi_{23} \langle \tilde{\hat{\phi}}_2^n , \tilde{\phi}_2^{n+1} \rangle
	+ (\chi_{12}-\chi_{13}-\chi_{23})  (
	\langle \tilde{\hat{\phi}}_2^n , \tilde{\phi}_1^{n+1} \rangle
	+ \langle \tilde{\hat{\phi}}_1^n , \tilde{\phi}_2^{n+1} \rangle  )   \nonumber
	\\
	&&
	+ A_1\dt \langle\nabla_h(\tilde{\phi}_1^{n+1}-\tilde{\phi}_1^n), \nabla_h \tilde{\phi}_1^{n+1}\rangle + A_2\dt \langle\nabla_h(\tilde{\phi}_2^{n+1}-\tilde{\phi}_2^n), \nabla_h \tilde{\phi}_2^{n+1}\rangle
	\nonumber
	\\
	&& + \langle \tilde{\mu}_{1,s}^{n+1} , \tilde{\phi}_1^{n+1} \rangle
	+ \langle \tilde{\mu}_{2,s}^{n+1} , \tilde{\phi}_2^{n+1} \rangle
	+ \langle \tilde{\mu}_{3,s}^{n+1} ,
	\tilde{\phi}_1^{n+1} + \tilde{\phi}_2^{n+1} \rangle .   \label{convergence-rough-5}
\end{eqnarray}
For the first nonlinear inner product on the right hand side, we begin with the following observation:
\begin{eqnarray*}
	\ln \check{\Phi}_{1,N}^{n+1} - \ln \phi_1^{n+1}
	=  \frac{1}{\xi} \tilde{\phi}_1^{n+1} ,  \quad
	\mbox{with $0 < \xi < 1$ between $\phi_1^{n+1}$ and $\check{\Phi}_{1,N}^{n+1}$} ,
	\label{convergence-rough-6-1}
\end{eqnarray*}
which comes from an application of intermediate value theorem. Since the bound $0 < \xi < 1$ is available at a point-wise level, we conclude that
\begin{eqnarray}
	\langle ( \ln \check{\Phi}_{1,N}^{n+1} - \ln \phi_1^{n+1} )  ,
	\tilde{\phi}_1^{n+1}  \rangle  \ge \| \tilde{\phi}_1^{n+1} \|_2^2 .
	\label{convergence-rough-6-2}
\end{eqnarray}
Using similar arguments, we also obtain
\begin{eqnarray}
	\hspace{-0.35in} &&
	\langle ( \ln \check{\Phi}_{2,N}^{n+1} - \ln \phi_2^{n+1} )  ,
	\tilde{\phi}_2^{n+1}  \rangle  \ge \| \tilde{\phi}_2^{n+1} \|_2^2 ,
	\label{convergence-rough-6-3}
	\\
	\hspace{-0.35in} &&
	- \langle ( \ln(1- \check{\Phi}_{1,N}^{n+1} - \check{\Phi}_{2,N}^{n+1}) - \ln(1-\phi_1^{n+1} -\phi_2^{n+1}) ) ,
	\tilde{\phi}_1^{n+1} + \tilde{\phi}_2^{n+1} \rangle
	\ge \| \tilde{\phi}_1^{n+1} + \tilde{\phi}_2^{n+1} \|_2^2  .
	\label{convergence-rough-6-4}
\end{eqnarray}
Moreover, since the discrete surface energy functional presented in~\eqref{Discrete-energy-c} is convex, we conclude that
\begin{eqnarray*}
	\langle \tilde{\mu}_{1,s}^{n+1} , \tilde{\phi}_1^{n+1} \rangle
	+ \langle \tilde{\mu}_{2,s}^{n+1} , \tilde{\phi}_2^{n+1} \rangle
	+ \langle \tilde{\mu}_{3,s}^{n+1} ,
	\tilde{\phi}_1^{n+1} + \tilde{\phi}_2^{n+1} \rangle
	\ge 0 .  
\end{eqnarray*}
For the artificial term, we have
\begin{equation*}
	2\langle\nabla_h(\tilde{\phi}_i^{n+1}-\tilde{\phi}_i^n),\nabla_h\tilde{\phi}_i^{n+1}\rangle= \nrm{\nabla_h\tilde{\phi}_i^{n+1}}_2^2 -\nrm{\nabla_h\tilde{\phi}_i^n}_2^2 + \nrm{\nabla_h(\tilde{\phi}_i^{n+1}-\tilde{\phi}_i^n)}_2^2.
\end{equation*}
Going back~\eqref{convergence-rough-5}, we arrive at
\begin{eqnarray*}
	&&
	\langle \tilde{\phi}_1^{n+1} , \tilde{\mu}_1^{n+1} \rangle
	+ \langle \tilde{\phi}_2^{n+1} , \tilde{\mu}_2^{n+1} \rangle   \nonumber
	\\
	&\ge& \frac{1}{M_0} \| \tilde{\phi}_1^{n+1} \|_2^2
	+ \frac{1}{N_0} \| \tilde{\phi}_2^{n+1} \|_2^2
	+ \| \tilde{\phi}_1^{n+1} + \tilde{\phi}_2^{n+1} \|_2^2
	- 4 \chi_{13}^2 M_0  \| \tilde{\hat{\phi}}_1^n \|_2^2
	- \frac{1}{4 M_0} \| \tilde{\phi}_1^{n+1} \|_2^2   \nonumber
	\\
	&&
	- 4 \chi_{23}^2 N_0 \| \tilde{\hat{\phi}}_2^n \|_2^2 
	- \frac{1}{4 N_0} \| \tilde{\phi}_2^{n+1} \|_2^2   
	- \frac{1}{4 M_0} \| \tilde{\phi}_1^{n+1} \|_2^2
	- \frac{1}{4 N_0} \| \tilde{\phi}_2^{n+1} \|_2^2   \nonumber
	\\
	&&
	- (\chi_{12}-\chi_{13}-\chi_{23})^2 ( M_0  \| \tilde{\hat{\phi}}_2^n \|_2^2
	+ N_0 \| \tilde{\hat{\phi}}_1^n \|_2^2 ) \nonumber
	\\
	&& + \frac{A_1 \dt}{2}\left(\nrm{\nabla_h\tilde{\phi}_1^{n+1}}_2^2 -\nrm{\nabla_h\tilde{\phi}_1^n}_2^2\right)  + \frac{A_2 \dt}{2}\left(\nrm{\nabla_h\tilde{\phi}_2^{n+1}}_2^2 -\nrm{\nabla_h\tilde{\phi}_2^n}_2^2\right)\nonumber
	\\
	&\ge& \frac{1}{2M_0} \| \tilde{\phi}_1^{n+1} \|_2^2
	+ \frac{1}{2N_0} \| \tilde{\phi}_2^{n+1} \|_2^2
	- ( 4 \chi_{13}^2 M_0  + (\chi_{12}-\chi_{13}-\chi_{23})^2 N_0)
	\| \tilde{\hat{\phi}}_1^n \|_2^2   \nonumber
	\\
	&&
	- ( 4 \chi_{23}^2 N_0   + (\chi_{12}-\chi_{13}-\chi_{23})^2 M_0 )
	\| \tilde{\hat{\phi}}_2^n \|_2^2\nonumber
	\\
 	&& + \frac{A_1 \dt}{2}\left(\nrm{\nabla_h\tilde{\phi}_1^{n+1}}_2^2 -\nrm{\nabla_h\tilde{\phi}_1^n}_2^2\right) +  \frac{A_2 \dt}{2}\left(\nrm{\nabla_h\tilde{\phi}_2^{n+1}}_2^2 -\nrm{\nabla_h\tilde{\phi}_2^n}_2^2\right).
\end{eqnarray*}
In turn, its substitution into~\eqref{convergence-rough-4} yields
\begin{eqnarray*}
	&&
	\frac{1}{2M_0} \| \tilde{\phi}_1^{n+1} \|_2^2
	+ \frac{1}{2N_0} \| \tilde{\phi}_2^{n+1} \|_2^2  + \frac{A_1 \dt}{2}\nrm{\nabla_h\tilde{\phi}_1^{n+1}}_2^2  + \frac{A_2 \dt}{2}\nrm{\nabla_h\tilde{\phi}_2^{n+1}}_2^2   \nonumber
	\\
	&\le&
	( 4 \chi_{13}^2 M_0  + (\chi_{12}-\chi_{13}-\chi_{23})^2 N_0)
	\| \tilde{\hat{\phi}}_1^n \|_2^2
	+ ( 4 \chi_{23}^2 N_0   + (\chi_{12}-\chi_{13}-\chi_{23})^2 M_0 )
	\| \tilde{\hat{\phi}}_2^n \|_2^2  \nonumber
	\\
	&&
	+ \frac{8}{3 \mathcal{M}_* \dt} ( \| \tilde{\phi}_1^n \|_{-1,h}^2
	+ \| \tilde{\phi}_2^n \|_{-1,h}^2  )
	+ \frac{1}{6 \mathcal{M}_* \dt} ( \| \tilde{\phi}_1^{n-1} \|_{-1,h}^2
	+ \| \tilde{\phi}_2^{n-1} \|_{-1,h}^2  )\nonumber\\
	&&+ \frac{\dt}{3 \mathcal{M}_*} ( \| \tau_1^{n+1} \|_{-1,h}^2 + \| \tau_2^{n+1} \|_{-1,h}^2 )  + \frac{A_1 \dt}{2}\nrm{\nabla_h\tilde{\phi}_1^n}_2^2+ \frac{A_2 \dt}{2}\nrm{\nabla_h\tilde{\phi}_2^n}_2^2 .
\end{eqnarray*}
Furthermore, a substitution of the a-priori error bound~\eqref{a priori-1} and \eqref{a priori-2} at the previous time step results in a rough error estimate for $\tilde{\phi}_1^{n+1}$, $\tilde{\phi}_2^{n+1}$:
\begin{eqnarray}
	\| \tilde{\phi}_1^{n+1} \|_2 + \| \tilde{\phi}_2^{n+1} \|_2
	\le \hat {C} ( \dt^\frac54 + h^\frac54 ) ,   \label{convergence-rough-10}
\end{eqnarray}
under the linear refinement requirement $C_1 h \le \dt \le C_2 h$, with $\hat{C}$ dependent on $M_0$, $N_0$, $\chi_{12}$, $\chi_{13}$, $\chi_{23}$, $A_1$ and $A_2$. Subsequently, an application of 2-D inverse inequality implies that
\begin{small}
\begin{eqnarray}
	\| \tilde{\phi}_1^{n+1} \|_\infty + \| \tilde{\phi}_2^{n+1} \|_\infty
	\le \frac{C ( \| \tilde{\phi}_1^{n+1} \|_2 + \| \tilde{\phi}_2^{n+1} \|_2 ) }{h}
	\le \hat {C}_1 ( \dt^\frac14 + h^\frac14 )  ,   \quad
	\mbox{with $\hat{C}_1 = C \hat{C}$} , \label{convergence-rough-11}
\end{eqnarray}
\end{small}
under the same linear refinement requirement. Because of the accuracy order, we could take $\dt$ and $h$ sufficient small so that
\begin{eqnarray*}
	\hat {C}_1 ( \dt^\frac14 + h^\frac14 ) \le \frac{\delta_0}{4} ,   \quad
	\mbox{so that} \, \, \,
	\| \tilde{\phi}_1^{n+1} \|_\infty + \| \tilde{\phi}_2^{n+1} \|_\infty
	\le \frac{\delta_0}{4}  .  
\end{eqnarray*}
Its combination with the separation property~\eqref{assumption:separation-2} leads to a similar property for the numerical solution:
\begin{equation}
	\phi_1^{n+1} \ge \frac{\delta_0}{2} , \, \, \, \phi_2^{n+1} \ge \frac{\delta_0}{2} , \, \, \,
	1 - \phi_1^{n+1} - \phi_2^{n+1}  \ge \frac{\delta_0}{2} ,  \quad \mbox{for $\delta_0 > 0$} .
	\label{assumption:separation-3}
\end{equation}
Such a uniform $\| \cdot \|_\infty$ bound will play a very important role in the refined error estimate.
\begin{rem}
	In the rough error estimate~\eqref{convergence-rough-10}, we see that the accuracy order is lower than the one given by the a-priori-assumption~\eqref{a priori-1}.
	Therefore, such a rough estimate could not be used for a global induction analysis. Instead, the purpose of such an estimate is to establish a uniform $\| \cdot \|_\infty$ bound, via the technique of inverse inequality, so that a discrete separation property becomes available for the numerical solution. With such a property established for the numerical solution, the refined error analysis will yield much sharper estimates.
\end{rem}
\subsection{The refined error estimate}

Taking a discrete inner product with~\eqref{error equation-1}, \eqref{error equation-2} by $\frac{2\Delta t}{\mathcal{M}_1}(- \Delta_h)^{-1} \tilde{\phi}_1^{n+1}$, $\frac{2\Delta t}{\mathcal{M}_2}(-\Delta_h)^{-1} \tilde{\phi}_2^{n+1}$, respectively, leads to
\begin{eqnarray}
	&&
	\frac{1}{\mathcal{M}_1} \langle 3\tilde{\phi}_1^{n+1} - 4\tilde{\phi}_1^n + \tilde{\phi}_1^{n-1}, \tilde{\phi}_1^{n+1} \rangle_{-1,h}
	+\frac{1}{\mathcal{M}_2} \langle 3\tilde{\phi}_2^{n+1} - 4\tilde{\phi}_2^n + \tilde{\phi}_2^{n-1}, \tilde{\phi}_2^{n+1} \rangle_{-1,h}\nonumber\\
	&&+ 2\dt (  \langle \tilde{\phi}_1^{n+1} , \tilde{\mu}_1^{n+1} \rangle
	+ \langle \tilde{\phi}_2^{n+1} , \tilde{\mu}_2^{n+1} \rangle )  \nonumber
	\\
	&=&
	\frac{2\dt}{\mathcal{M}_1} \langle \tau_1^{n+1} , \tilde{\phi}_1^{n+1} \rangle_{-1,h}
	+ \frac{2\dt}{\mathcal{M}_2}\langle \tau_2^{n+1} , \tilde{\phi}_2^{n+1} \rangle_{-1,h}  ,
	\label{convergence-1}
\end{eqnarray}
with the summation by parts formula applied. The following identities are available for the temporal approximation terms using \eqref{BDF2_term_estimate}:
\begin{small}
\begin{eqnarray}
	&&\frac{1}{\mathcal{M}_j} \langle 3\tilde{\phi}_j^{n+1} - 4\tilde{\phi}_j^n + \tilde{\phi}_j^{n-1}, \tilde{\phi}_j^{n+1} \rangle_{-1,h}\nonumber\\
	=&&  \frac{1}{ \mathcal{M}_j} ( \| \mathbf{p}_j^{n+1} \|_{-1,h}^2 - \| \mathbf{p}_j^n \|_{-1,h}^2)
	+ \frac{\| \tilde{\phi}_j^{n+1} - 2\tilde{\phi}_j^n + \tilde{\phi}_j^{n-1} \|_{-1,h}^2 }{2\mathcal{M}_j},
	\label{convergence-2}
\end{eqnarray}
where $\mathbf{p}_j^{n+1} = [\tilde{\phi}_j^n,\tilde{\phi}_j^{n+1}]^T$.
\end{small}

For the local truncation error terms, similar estimates could be derived:
\begin{eqnarray}
	\frac{2\dt }{\mathcal{M}_j} \langle \tau_j^{n+1} , \tilde{\phi}_j^{n+1} \rangle_{-1,h}
	\le  \frac{\dt}{\mathcal{M}_j} ( \| \tau_j^{n+1} \|_{-1,h}^2 + \| \tilde{\phi}_j^{n+1} \|_{-1,h}^2 ) ,
	\quad j=1, 2 . \label{convergence-3}
\end{eqnarray}
For the term $ \langle \tilde{\phi}_1^{n+1} , \tilde{\mu}_1^{n+1} \rangle
+ \langle \tilde{\phi}_2^{n+1} , \tilde{\mu}_2^{n+1} \rangle$, the expansion~\eqref{convergence-rough-5}, as well as the inequalities~\eqref{convergence-rough-6-2}-\eqref{convergence-rough-6-4}, are still valid. For the inner product associated with the concave terms, a standard Cauchy inequality is applied:
\begin{small}
\begin{eqnarray}
	&&
	- 2 \chi_{13}  \langle \tilde{\hat{\phi}}_1^n , \tilde{\phi}_1^{n+1} \rangle
	\ge - 2 \chi_{13}  \| \tilde{\hat{\phi}}_1^n \|_{-1,h} \| \nabla_h \tilde{\phi}_1^{n+1}  \|_2
	\nonumber
	\\
	&\ge&
	- \frac{144 \chi_{13}^2 }{\varepsilon_0^2} \| \tilde{\hat{\phi}}_1^n \|_{-1,h}^2
	- \frac{\varepsilon_0^2 }{144} \| \nabla_h \tilde{\phi}_1^{n+1}  \|_2^2
	\nonumber
	\\
	&\ge&
	- \frac{144 \chi_{13}^2 }{\varepsilon_0^2} \left( 6\|\tilde{\phi}_1^n\|_{-1,h}^2 + 3\|\tilde{\phi}_1^{n-1}\|_{-1,h}^2\right)
	- \frac{\varepsilon_0^2 }{144} \| \nabla_h \tilde{\phi}_1^{n+1}  \|_2^2 , \label{convergence-4-1}
	\\
	&&
	- 2 \chi_{23}  \langle \tilde{\hat{\phi}}_2^n , \tilde{\phi}_2^{n+1} \rangle
\ge - 2 \chi_{23}  \| \tilde{\hat{\phi}}_2^n \|_{-1,h} \| \nabla_h \tilde{\phi}_2^{n+1}  \|_2
\nonumber
\\
&\ge&
- \frac{144 \chi_{23}^2 }{\varepsilon_0^2} \| \tilde{\hat{\phi}}_2^n \|_{-1,h}^2
- \frac{\varepsilon_0^2 }{144} \| \nabla_h \tilde{\phi}_2^{n+1}  \|_2^2
\nonumber
\\
&\ge&
- \frac{144 \chi_{23}^2 }{\varepsilon_0^2} \left( 6\|\tilde{\phi}_2^n\|_{-1,h}^2 + 3\|\tilde{\phi}_2^{n-1}\|_{-1,h}^2\right)
- \frac{\varepsilon_0^2 }{144} \| \nabla_h \tilde{\phi}_2^{n+1}  \|_2^2 , \label{convergence-4-2}
	\\
	&&
	(\chi_{12}-\chi_{13}-\chi_{23})  (
	\langle \tilde{\hat{\phi}}_2^n , \tilde{\phi}_1^{n+1} \rangle
	+  \langle \tilde{\hat{\phi}}_1^n , \tilde{\phi}_2^{n+1} \rangle  )   \nonumber
	\\
	&\ge&
	- \frac{36  (\chi_{12}-\chi_{13}-\chi_{23})^2}{\varepsilon_0^2} (  \| \tilde{\hat{\phi}}_1^n \|_{-1,h}^2
	+  \| \tilde{\hat{\phi}}_2^n \|_{-1,h}^2 )
	- \frac{\varepsilon_0^2}{144} (  \| \nabla_h \tilde{\phi}_1^{n+1}  \|_2^2
	+  \| \nabla_h \tilde{\phi}_2^{n+1}  \|_2^2 )\nonumber\\
	&\ge&
	- \frac{36  (\chi_{12}-\chi_{13}-\chi_{23})^2}{\varepsilon_0^2} \left(  6(\| \tilde{\phi}_1^n \|_{-1,h}^2 + \| \tilde{\phi}_2^n \|_{-1,h}^2)
	+  3(\| \tilde{\phi}_1^{n-1}\|_{-1,h}^2 + \| \tilde{\phi}_2^{n-1} \|_{-1,h}^2) \right)\nonumber\\
	&&- \frac{\varepsilon_0^2}{144} (  \| \nabla_h \tilde{\phi}_1^{n+1}  \|_2^2
	+  \| \nabla_h \tilde{\phi}_2^{n+1}  \|_2^2 )  .\label{convergence-4-3}
\end{eqnarray}
\end{small}	
The rest works are focused on the estimates for the error terms associated with the nonlinear surface diffusion, as given by $ \langle \tilde{\mu}_{1,s}^{n+1} , \tilde{\phi}_1^{n+1} \rangle$, $ \langle \tilde{\mu}_{2,s}^{n+1} , \tilde{\phi}_2^{n+1} \rangle$,  $\langle \tilde{\mu}_{3,s}^{n+1} , \tilde{\phi}_1^{n+1} +  \tilde{\phi}_2^{n+1} \rangle$, the last three terms in~\eqref{convergence-rough-5}. First, we look at the expansion for $ \langle \tilde{\mu}_{1,s}^{n+1} , \tilde{\phi}_1^{n+1} \rangle$, which comes from the expression~\eqref{error equation-5}:
\begin{small}
\begin{eqnarray*}
	&&
	\langle \tilde{\mu}_{1,s}^{n+1} , \tilde{\phi}_1^{n+1} \rangle = I_1 + I_2 + I_3 + I_4 , \quad  \mbox{with}  \nonumber
	\\
	&&
	I_1 :=
	\frac{\varepsilon_1^2 }{36} \langle {\cal A}_h (  \gamma^{(1)} {\cal A}_h \tilde{\phi}_1^{n+1} ), \tilde{\phi}_1^{n+1} \rangle ,   \quad
	I_2 : =
	- \frac{\varepsilon_1^2 }{36} \Bigl\langle {\cal A}_h \Bigl( \frac{ \nabla_h ( \check{\Phi}_{1,N}^{n+1} + \phi_1^{n+1}) \cdot \nabla_h \tilde{\phi}_1^{n+1} }{ ( {\cal A}_h \check{\Phi}_{1,N}^{n+1} )^2}  \Bigr) ,  \tilde{\phi}_1^{n+1} \Bigr\rangle,  \nonumber
	\\
	&&
	I_3 := - \frac{\varepsilon_1^2 }{18} \Bigl\langle \nabla_h \cdot \Bigl( \frac{\nabla_h \tilde{\phi}_1^{n+1}}{{\cal A}_h \phi_1^{n+1} } \Bigr) ,  \tilde{\phi}_1^{n+1} \Bigr\rangle , \quad
	I_4 :=  \frac{\varepsilon_1^2 }{18} \Bigl\langle \nabla_h \cdot \Bigl( \frac{ \tilde{\phi}_1^{n+1} \nabla_h \check{\Phi}_{1,N}^{n+1}}{{\cal A}_h \phi_1^{n+1} {\cal A}_h \check{\Phi}_{1,N}^{n+1} }  \Bigr) , \tilde{\phi}_1^{n+1} \Bigr\rangle  . 
\end{eqnarray*}
\end{small}
It is clear that $I_1$ stays non-negative:
\begin{eqnarray}
	I_1 = \frac{\varepsilon_1^2 }{36} \langle  \gamma^{(1)} {\cal A}_h \tilde{\phi}_1^{n+1} , {\cal A}_h \tilde{\phi}_1^{n+1} \rangle  \ge 0 ,  \label{convergence-5-1}
\end{eqnarray}
in which the summation by parts formula is applied in the first step, while the fact that $\gamma^{(1)} \ge 0$ (given by~\eqref{error equation-8}) is used in the second step. Similarly, for the third part $I_3$, an application of summation by parts formula reveals that
\begin{eqnarray}
	I_3 = \frac{\varepsilon_1^2 }{18} \Bigl[  \frac{\nabla_h \tilde{\phi}_1^{n+1}}{{\cal A}_h \phi_1^{n+1} } ,  \nabla_h \tilde{\phi}_1^{n+1} \Bigr]
	\ge \frac{\varepsilon_1^2 }{18}  \|  \nabla_h \tilde{\phi}_1^{n+1}  \|_2^2 ,
	\label{convergence-5-2}
\end{eqnarray}
in which the point-wise estimate $0 < \phi_1^{n+1} < 1$ has been used in the second step. For the fourth part $I_4$, an application of summation by parts formula gives
\begin{eqnarray}
	- I_4  &=&  \frac{\varepsilon_1^2 }{18} \Bigl[ \frac{ \tilde{\phi}_1^{n+1} \nabla_h \check{\Phi}_{1,N}^{n+1}}{{\cal A}_h \phi_1^{n+1} {\cal A}_h \check{\Phi}_{1,N}^{n+1} }  , \nabla_h \tilde{\phi}_1^{n+1} \Bigr]   \nonumber
	\\
	&\le&
	\frac{\varepsilon_1^2 }{18} \Bigl\|  \frac{ 1 }{{\cal A}_h \phi_1^{n+1} } \Bigr\|_\infty
	\cdot \Bigl\| \frac{ 1 }{{\cal A}_h \check{\Phi}_{1,N}^{n+1} }  \Bigr\|_\infty
	\cdot \| \nabla_h \check{\Phi}_{1,N}^{n+1} \|_\infty
	\cdot \| \tilde{\phi}_1^{n+1} \|_2 \cdot \|  \nabla_h \tilde{\phi}_1^{n+1} \|_2  \nonumber
	\\
	&\le&
	\frac{C^* \varepsilon_1^2 }{18} \cdot 2 (\delta_0)^{-2}
	\| \tilde{\phi}_1^{n+1} \|_2 \cdot \|  \nabla_h \tilde{\phi}_1^{n+1} \|_2  \nonumber
	\\
	&\le&
	\frac{C^* \varepsilon_1^2  (\delta_0)^{-2} }{9}
	\| \tilde{\phi}_1^{n+1} \|_{-1,h}^\frac12
	\cdot \|  \nabla_h \tilde{\phi}_1^{n+1} \|_2^\frac32    \nonumber
	\\
	&\le&
	C (C^*)^4 (\delta_0)^{-8} \varepsilon_1^2  \| \tilde{\phi}_1^{n+1} \|_{-1,h}^2
	+ \frac{\varepsilon_1^2 }{72} \|  \nabla_h \tilde{\phi}_1^{n+1} \|_2^2 .
	\label{convergence-5-3}
\end{eqnarray}
In more details, the preliminary estimate \eqref{assumption:W1-infty bound} has been applied in the third step, combined the separation properties~\eqref{assumption:separation-3}; the Sobolev interpolation formula, $\| \tilde{\phi}_1^{n+1} \|_2 \le \| \tilde{\phi}_1^{n+1} \|_{-1,h}^\frac12 \cdot \|  \nabla_h \tilde{\phi}_1^{n+1} \|_2^\frac12$,  has been used in the fourth step; the Young's inequality has been applied in the last step. For the second term $I_2$, we begin with the following summation by parts:
\begin{eqnarray}
	- I_2 =
	\frac{\varepsilon_1^2 }{36} \Bigl[ \frac{ \nabla_h ( \check{\Phi}_{1,N}^{n+1} + \phi_1^{n+1}) \cdot \nabla_h \tilde{\phi}_1^{n+1} }{ ( {\cal A}_h \check{\Phi}_{1,N}^{n+1} )^2} ,  {\cal A}_h \tilde{\phi}_1^{n+1} \Bigr]  .
	\label{convergence-5-4}
\end{eqnarray}
Meanwhile, because of the fact $ \phi_1^{n+1} = \Phi_{1,N}^{n+1} - \tilde{\phi}_1^{n+1}$, we are able to decompose $-I_2$ into two parts:
\begin{eqnarray*}
	- I_2 = -I_{2,1} - I_{2,2} , \quad \mbox{with} &&
	- I_{2,1} := \frac{\varepsilon_1^2 }{18} \Bigl[ \frac{ \nabla_h \check{\Phi}_{1,N}^{n+1}  \cdot \nabla_h \tilde{\phi}_1^{n+1} }{ ( {\cal A}_h \check{\Phi}_{1,N}^{n+1} )^2} ,  {\cal A}_h \tilde{\phi}_1^{n+1} \Bigr]  ,
	\\
	&&
	-I_{2,2} := - \frac{\varepsilon_1^2 }{36} \Bigl[ \frac{ | \nabla_h \tilde{\phi}_1^{n+1} |^2 }{ ( {\cal A}_h \check{\Phi}_{1,N}^{n+1} )^2} ,  {\cal A}_h \tilde{\phi}_1^{n+1} \Bigr]  .
\end{eqnarray*}
The bound for $-I_{2,1}$ could be obtained in a similar style as~\eqref{convergence-5-3}:
\begin{small}
\begin{eqnarray}
	&&- I_{2,1}  \nonumber\\ 
	&\le&\frac{\varepsilon_1^2 }{18} \Bigl\| \frac{ 1 }{{\cal A}_h \check{\Phi}_{1,N}^{n+1} }
	\Bigr\|_\infty^2
	\cdot \| \nabla_h \check{\Phi}_{1,N}^{n+1} \|_\infty
	\cdot \| {\cal A}_h \tilde{\phi}_1^{n+1} \|_2
	\cdot \|  \nabla_h \tilde{\phi}_1^{n+1} \|_2  \nonumber
	\\
	&\le&
	\frac{C^* \varepsilon_1^2 }{18} \cdot (\delta_0)^{-2}
	\| {\cal A}_h \tilde{\phi}_1^{n+1} \|_2 \cdot \|  \nabla_h \tilde{\phi}_1^{n+1} \|_2
	\le \frac{C^* \varepsilon_1^2 }{18} \cdot (\delta_0)^{-2}
	\| \tilde{\phi}_1^{n+1} \|_2 \cdot \|  \nabla_h \tilde{\phi}_1^{n+1} \|_2
	\nonumber
	\\
	&\le&
	\frac{C^* \varepsilon_1^2  (\delta_0)^{-2} }{18}
	\| \tilde{\phi}_1^{n+1} \|_{-1,h}^\frac12
	\cdot \|  \nabla_h \tilde{\phi}_1^{n+1} \|_2^\frac32    \nonumber
	\\
	&\le&
	C (C^*)^4 (\delta_0)^{-8} \varepsilon_1^2 \| \tilde{\phi}_1^{n+1} \|_{-1,h}^2
	+ \frac{\varepsilon_1^2 }{144} \|  \nabla_h \tilde{\phi}_1^{n+1} \|_2^2 .
	\label{convergence-5-5}
\end{eqnarray}
\end{small}
For the other part $-I_{2,2}$, we recall the $\| \cdot \|_\infty$ rough estimate~\eqref{convergence-rough-11} and the separation inequality~\eqref{assumption:W1-infty bound}, and arrive at
\begin{eqnarray}
	-I_{2,2} &\le& \frac{\varepsilon_1^2 }{36} \Bigl\| \frac{ 1 }{{\cal A}_h \check{\Phi}_{1,N}^{n+1} }
	\Bigr\|_\infty^2   \cdot \| {\cal A}_h \tilde{\phi}_1^{n+1} \|_\infty
	\cdot  \|  \nabla_h \tilde{\phi}_1^{n+1} \|_2^2   \nonumber
	\\
	&\le&
	\frac{\varepsilon_1^2 }{36} \cdot (\delta_0)^{-2}  \cdot \hat {C}_1 ( \dt^\frac14 + h^\frac14 )
	\|  \nabla_h \tilde{\phi}_1^{n+1} \|_2^2 .
	\label{convergence-5-6}
\end{eqnarray}
In turn, if $\dt$ and $h$ are sufficiently small so that
\begin{eqnarray}
	\frac{\hat{C}_1 (\delta_0)^{-2} }{36} ( \dt^\frac14 + h^\frac14 )
	\le \frac{1}{144}  ,  \label{condition-dt h-1}
\end{eqnarray}
we obtain a useful bound
\begin{eqnarray}
	-I_{2,2} \le \frac{\varepsilon_1^2 }{144} \|  \nabla_h \tilde{\phi}_1^{n+1} \|_2^2 .
	\label{convergence-5-7}
\end{eqnarray}
A substitution of~\eqref{convergence-5-5}-\eqref{convergence-5-7} into ~\eqref{convergence-5-4} leads to
\begin{eqnarray}
	- I_2   \le C (C^*)^4 (\delta_0)^{-8} \varepsilon_1^2  \| \tilde{\phi}_1^{n+1} \|_{-1,h}^2
	+ \frac{\varepsilon_1^2 }{72} \|  \nabla_h \tilde{\phi}_1^{n+1} \|_2^2 .
	\label{convergence-5-8}
\end{eqnarray}
Finally, a combination of~\eqref{convergence-5-1}-\eqref{convergence-5-3} and \eqref{convergence-5-8} results in
\begin{eqnarray}
	\langle \tilde{\mu}_{1,s}^{n+1} , \tilde{\phi}_1^{n+1} \rangle
	\ge \frac{\varepsilon_1^2 }{36} \|  \nabla_h \tilde{\phi}_1^{n+1} \|_2^2
	-   2 C (C^*)^4 (\delta_0)^{-8} \varepsilon_1^2  \| \tilde{\phi}_1^{n+1} \|_{-1,h}^2 .
	\label{convergence-6}
\end{eqnarray}

The two other nonlinear surface diffusion error terms could be analyzed in the same style. The results are stated below; the technical details are skipped for the sake of brevity.
\begin{eqnarray}
	&&\langle \tilde{\mu}_{2,s}^{n+1} , \tilde{\phi}_2^{n+1} \rangle\nonumber\\
	&\ge& \frac{\varepsilon_2^2 }{36} \|  \nabla_h \tilde{\phi}_2^{n+1} \|_2^2
	-   2 C (C^*)^4 (\delta_0)^{-8} \varepsilon_2^2  \| \tilde{\phi}_2^{n+1} \|_{-1,h}^2 ,
	\label{convergence-7}
	\\
	&&\langle \tilde{\mu}_{3,s}^{n+1} , \tilde{\phi}_1^{n+1} + \tilde{\phi}_2^{n+1} \rangle\nonumber\\
	&\ge& \frac{\varepsilon_3^2}{36} \|  \nabla_h ( \tilde{\phi}_1^{n+1}
	+ \tilde{\phi}_2^{n+1})  \|_2^2
	-   2 C (C^*)^4 (\delta_0)^{-8} \varepsilon_3^2 \| \tilde{\phi}_1^{n+1} + \tilde{\phi}_2^{n+1} \|_{-1,h}^2  \nonumber
	\\
	&\ge&
	\frac{\varepsilon_3^2}{36} \|  \nabla_h ( \tilde{\phi}_1^{n+1}
	+ \tilde{\phi}_2^{n+1})  \|_2^2  \nonumber
	\\
	&&
	-   4 C (C^*)^4 (\delta_0)^{-8} \varepsilon_3^2
	( \| \tilde{\phi}_1^{n+1}  \|_{-1,h}^2 +  \| \tilde{\phi}_2^{n+1} \|_{-1,h}^2 )  .
	\label{convergence-8}
\end{eqnarray}

A substitution of~\eqref{convergence-rough-6-2}-\eqref{convergence-rough-6-4}, \eqref{convergence-4-1}-\eqref{convergence-4-3}, \eqref{convergence-6}-\eqref{convergence-8} into~\eqref{convergence-rough-5} results in
\begin{small}
\begin{eqnarray}
	&&
	\langle \tilde{\phi}_1^{n+1} , \tilde{\mu}_1^{n+1} \rangle
	+ \langle \tilde{\phi}_2^{n+1} , \tilde{\mu}_2^{n+1} \rangle  \nonumber
	\\
	&\ge&
	\frac{\varepsilon_0^2}{72} ( \|  \nabla_h \tilde{\phi}_1^{n+1}  \|_2^2
	+ \| \nabla_h \tilde{\phi}_2^{n+1}  \|_2^2  )
	-   4 C (C^*)^4 (\delta_0)^{-8} (\varepsilon_1^2 + \varepsilon_2^2 + \varepsilon_3^2)
	( \| \tilde{\phi}_1^{n+1}  \|_{-1,h}^2 +  \| \tilde{\phi}_2^{n+1} \|_{-1,h}^2  )
	\nonumber
	\\
	&&
	- \frac{144}{\varepsilon_0^2}  ( \chi_{13}^2 + \chi_{23}^2 + \frac{1}{4}(\chi_{12}-\chi_{13}-\chi_{23})^2 )
	\left[ 6(\| \tilde{\phi}_1^n \|_{-1,h}^2 + \| \tilde{\phi}_2^n \|_{-1,h}^2)\right.\nonumber\\
	&&\left.
	\quad+ 3(\| \tilde{\phi}_1^{n-1} \|_{-1,h}^2 + \| \tilde{\phi}_2^{n-1} \|_{-1,h}^2) \right]\nonumber\\
	&& + \frac{1}{M_0}\| \tilde{\phi}_1^{n+1} \|_2^2 + \frac{1}{N_0} \|\tilde{\phi}_2^{n+1} \|_2^2 + \| \tilde{\phi}_1^{n+1} + \tilde{\phi}_2^{n+1} \|_2^2 \nonumber\\
	&& + \frac{A_1 \dt}{2}\left(\nrm{\nabla_h\tilde{\phi}_1^{n+1}}_2^2 -\nrm{\nabla_h\tilde{\phi}_1^n}_2^2\right)  + \frac{A_2 \dt}{2}\left(\nrm{\nabla_h\tilde{\phi}_2^{n+1}}_2^2 -\nrm{\nabla_h\tilde{\phi}_2^n}_2^2\right).\label{convergence-9}
\end{eqnarray}
\end{small}
A combination of~\eqref{convergence-1}-\eqref{convergence-3} and~\eqref{convergence-9} gives
\begin{small}
\begin{eqnarray}
	&&
	\left( \frac{1}{\mathcal{M}_1}\| \mathbf{p}_1^{n+1} \|_{-1,G}^2 + \frac{1}{\mathcal{M}_2}\| \mathbf{p}_2^{n+1} \|_{-1,G}^2 \right)
	- \left( \frac{1}{\mathcal{M}_1}\| \mathbf{p}_1^n \|_{-1,G}^2 + \frac{1}{\mathcal{M}_2}\| \mathbf{p}_2^n \|_{-1,G}^2 \right)\nonumber
	\\
	&&
	+ \frac{\varepsilon_0^2}{36} \dt ( \|  \nabla_h \tilde{\phi}_1^{n+1}  \|_2^2
	+ \| \nabla_h \tilde{\phi}_2^{n+1}  \|_2^2  )   + \left(A_1 \dt^2\nrm{\nabla_h\tilde{\phi}_1^{n+1}}_2^2  +  A_2 \dt^2\nrm{\nabla_h\tilde{\phi}_2^{n+1}}_2^2 \right) \nonumber
	\\
	&&
	- \left(A_1 \dt^2\nrm{\nabla_h\tilde{\phi}_1^n}_2^2 + A_2 \dt^2\nrm{\nabla_h\tilde{\phi}_2^n}_2^2\right)\nonumber\\
	&\le&
	\kappa^{(1)} \dt \left[ 6(\| \tilde{\phi}_1^n \|_{-1,h}^2
	+ \| \tilde{\phi}_2^n \|_{-1,h}^2) + 3(\| \tilde{\phi}_1^{n-1} \|_{-1,h}^2
	+ \| \tilde{\phi}_2^{n-1} \|_{-1,h}^2) \right]\nonumber\\
	&&+ \kappa^{(2)}  \dt ( \| \tilde{\phi}_1^{n+1}  \|_{-1,h}^2
	+   \| \tilde{\phi}_2^{n+1} \|_{-1,h}^2 )	+\dt \left( \frac{1}{\mathcal{M}_1}\| \tilde{\phi}_1^{n+1}  \|_{-1,h}^2
	+  \frac{1}{\mathcal{M}_2} \| \tilde{\phi}_2^{n+1} \|_{-1,h}^2  \right)\nonumber
	\\
	&&
	+ \dt \left( \frac{1}{\mathcal{M}_1} \| \tau_1^{n+1} \|_{-1,h}^2 + \frac{1}{\mathcal{M}_2} \| \tau_2^{n+1} \|_{-1,h}^2 \right) ,
	\label{convergence-10}
\end{eqnarray}
\end{small}
with
\begin{eqnarray*}
		&&
	\kappa^{(1)} =  \frac{288}{\varepsilon_0^2}
	( \chi_{13}^2 + \chi_{23}^2 +\frac{1}{4} (\chi_{12}-\chi_{13}-\chi_{23})^2 )  , \\
	&&
	\kappa^{(2)} = 8 C (C^*)^4 (\delta_0)^{-8} (\varepsilon_1^2 + \varepsilon_2^2 + \varepsilon_3^2) .\\
\end{eqnarray*}
In other words, we have
\begin{eqnarray}
	&&
	\left( \frac{1}{\mathcal{M}_1}\| \mathbf{p}_1^{n+1} \|_{-1,G}^2 + \frac{1}{\mathcal{M}_2}\| \mathbf{p}_2^{n+1} \|_{-1,G}^2 \right)
	- \left( \frac{1}{\mathcal{M}_1}\| \mathbf{p}_1^n \|_{-1,G}^2 + \frac{1}{\mathcal{M}_2}\| \mathbf{p}_2^n \|_{-1,G}^2 \right)\nonumber
	\\
	&&
	+ \frac{\varepsilon_0^2}{36} \dt ( \|  \nabla_h \tilde{\phi}_1^{n+1}  \|_2^2
	+ \| \nabla_h \tilde{\phi}_2^{n+1}  \|_2^2  )   + \left(A_1 \dt^2\nrm{\nabla_h\tilde{\phi}_1^{n+1}}_2^2  +  A_2 \dt^2\nrm{\nabla_h\tilde{\phi}_2^{n+1}}_2^2 \right) \nonumber
	\\
	&&
	- \left(A_1 \dt^2\nrm{\nabla_h\tilde{\phi}_1^n}_2^2 + A_2 \dt^2\nrm{\nabla_h\tilde{\phi}_2^n}_2^2\right)\nonumber\\
	&\le&
	\kappa^{(1)} \dt \left[ 6(\| \tilde{\phi}_1^n \|_{-1,h}^2
	+ \| \tilde{\phi}_2^n \|_{-1,h}^2) + 3(\| \tilde{\phi}_1^{n-1} \|_{-1,h}^2
	+ \| \tilde{\phi}_2^{n-1} \|_{-1,h}^2) \right]\nonumber\\
	&&+ \left(\kappa^{(2)}+\frac{\mathcal{M}_1 + \mathcal{M}_2}{\mathcal{M}_1 \mathcal{M}_2}\right)  \dt ( \| \tilde{\phi}_1^{n+1}  \|_{-1,h}^2
	+   \| \tilde{\phi}_2^{n+1} \|_{-1,h}^2 )\nonumber\\
	&&+ \frac{\mathcal{M}_1 + \mathcal{M}_2}{\mathcal{M}_1 \mathcal{M}_2}\dt\left( \| \tau_1^{n+1} \|_{-1,h}^2 + \| \tau_2^{n+1} \|_{-1,h}^2 \right).\label{convergence-10-b}
\end{eqnarray}
Now we observe that $\nrm{ \mathbf{p}_i^1 }_{-1,G}^2 = \frac{5}{2 }\nrm{ \tilde{\phi}_i^1 }_{-1,h}^2,~\nrm{ \tilde{\phi}_i^0 }_{-1,h}^2=0.$\\
Summing both sides of \eqref{convergence-10-b} with respect to $n$ gives
\begin{small}
     \begin{eqnarray}
	&&	\left( \frac{1}{\mathcal{M}_1}\| \mathbf{p}_1^{n+1} \|_{-1,G}^2 + \frac{1}{\mathcal{M}_2}\| \mathbf{p}_2^{n+1} \|_{-1,G}^2 \right)  	+ \frac{\varepsilon_0^2}{36} \dt \sum_{k=1}^n ( \|  \nabla_h \tilde{\phi}_1^{k+1}  \|_2^2
	+ \| \nabla_h \tilde{\phi}_2^{k+1}  \|_2^2  ) \nonumber\\
	&&+ \left(A_1 \dt^2\nrm{\nabla_h\tilde{\phi}_1^{n+1}}_2^2  +  A_2 \dt^2\nrm{\nabla_h\tilde{\phi}_2^{n+1}}_2^2 \right)\nonumber\\
	&\leq& 
	 \left(\kappa^{(2)}+\frac{\mathcal{M}_1 + \mathcal{M}_2}{\mathcal{M}_1 \mathcal{M}_2}\right)  \dt ( \| \tilde{\phi}_1^{n+1}  \|_{-1,h}^2
	+   \| \tilde{\phi}_2^{n+1} \|_{-1,h}^2 ) \nonumber\\
	&&+\frac{5}{2}\left( \frac{1}{\mathcal{M}_1}\| \tilde{\phi}_1^1 \|_{-1,h}^2 + \frac{1}{\mathcal{M}_2}\| \tilde{\phi}_2^1 \|_{-1,h}^2 \right) \nonumber\\
	&&+\left(9\kappa^{(1)} + \kappa^{(2)}+\frac{\mathcal{M}_1 + \mathcal{M}_2}{\mathcal{M}_1 \mathcal{M}_2}\right) \dt \sum_{k=1}^n (\| \tilde{\phi}_1^k \|_{-1,h}^2
	+ \| \tilde{\phi}_2^k \|_{-1,h}^2) \nonumber\\
	&&+ \frac{\mathcal{M}_1 + \mathcal{M}_2}{\mathcal{M}_1 \mathcal{M}_2}\dt \sum_{k=1}^n\left( \| \tau_1^{k+1} \|_{-1,h}^2 + \| \tau_2^{k+1} \|_{-1,h}^2 \right)\nonumber\\
	&&+\left(A_1 \dt^2\nrm{\nabla_h\tilde{\phi}_1^1}_2^2  +  A_2 \dt^2\nrm{\nabla_h\tilde{\phi}_2^1}_2^2 \right).
	\nonumber
\end{eqnarray}
\end{small}
We observe $\nrm{ \mathbf{p}_i^{n+1} }_{-1,G}^2 \geq \frac{1}{2 }\nrm{ \tilde{\phi}_i^{n+1} }_{-1,h}^2$. This means that
    \begin{eqnarray}
	&&	\left( \frac{1}{2\mathcal{M}_1} -	\left(\kappa^{(2)}+\frac{\mathcal{M}_1 + \mathcal{M}_2}{\mathcal{M}_1 \mathcal{M}_2}\right)  \dt \right) \| \tilde{\phi}_1^{n+1} \|_{-1,h}^2\nonumber\\
	&&+\left( \frac{1}{2\mathcal{M}_2} - \left(\kappa^{(2)}+\frac{\mathcal{M}_1 + \mathcal{M}_2}{\mathcal{M}_1 \mathcal{M}_2}\right)  \dt \right) \| \tilde{\phi}_2^{n+1} \|_{-1,h}^2\nonumber\\
	&& + \frac{\varepsilon_0^2}{36} \dt \sum_{k=1}^n ( \|  \nabla_h \tilde{\phi}_1^{k+1}  \|_2^2
	+ \| \nabla_h \tilde{\phi}_2^{k+1}  \|_2^2  ) \nonumber\\
	&\leq& 
	\left(9\kappa^{(1)} + \kappa^{(2)}+\frac{\mathcal{M}_1 + \mathcal{M}_2}{\mathcal{M}_1 \mathcal{M}_2}\right) \dt \sum_{k=1}^n (\| \tilde{\phi}_1^k \|_{-1,h}^2
	+ \| \tilde{\phi}_2^k \|_{-1,h}^2) \nonumber\\
	&&+\frac{5}{2}\left( \frac{1}{\mathcal{M}_1}\| \tilde{\phi}_1^1 \|_{-1,h}^2 + \frac{1}{\mathcal{M}_2}\| \tilde{\phi}_2^1 \|_{-1,h}^2 \right)+ \frac{\mathcal{M}_1 + \mathcal{M}_2}{\mathcal{M}_1 \mathcal{M}_2}\dt \sum_{k=1}^n\left( \| \tau_1^{k+1} \|_{-1,h}^2 + \| \tau_2^{k+1} \|_{-1,h}^2 \right)\nonumber\\
    &&+\left(A_1 \dt^2\nrm{\nabla_h\tilde{\phi}_1^1}_2^2  +  A_2 \dt^2\nrm{\nabla_h\tilde{\phi}_2^1}_2^2 \right).
	\label{convergence-11}
\end{eqnarray}
We need to analyze the error at the first time step separately, since the local truncation error is only second-order in time. Notice that $e_j^0 = 0$, by carefully calculation, we have
the following error equation for the initial level:
\begin{eqnarray}
	&&e_1^1 = \dt \mathcal{M}_1 \Delta_h \tilde{\mu}_1^1 + \dt \tau_1^1,\label{convergence-12-1}\\
	&&\tilde{\mu}_1^1 = \frac{1}{2}\left(\delta_{\phi_1}G_{h,c}(\Phi_{1,N}^1,\Phi_{2,N}^1)-
	\delta_{\phi_1}G_{h,c}(\phi_1^1,\phi_2^1)\right)\label{convergence-12-2}\\
	&&e_2^1 = \dt \mathcal{M}_2 \Delta_h \tilde{\mu}_2^1 + \dt \tau_2^1,\label{convergence-12-3}\\
    &&\tilde{\mu}_2^1 = \frac{1}{2}\left(\delta_{\phi_2}G_{h,c}(\Phi_{1,N}^1,\Phi_{2,N}^1)-
\delta_{\phi_2}G_{h,c}(\phi_1^1,\phi_2^1)\right)	\label{convergence-12-4},
\end{eqnarray}
where $\tau_j^1 \le C(\dt^2+h^2)$. Since $e_j^0 = 0$, we can omit the terms about the initial step $t=t_0$ in expressions above based on the intermediate value theorem. Taking the inner product with the error equation \eqref{convergence-12-1}, \eqref{convergence-12-3} by $(-\Delta_h)^{-1}e_1^1$, $(-\Delta_h)^{-1}e_2^1$, respectively, and using summation-by-parts, we have
\begin{eqnarray}
	&&\frac{1}{\mathcal{M}_1}\nrm{e_1^1}_{-1,h}^2+\frac{1}{\mathcal{M}_2}\nrm{e_2^1}_{-1,h}^2 + \dt\left(\langle e_1^1,\tilde{\mu}_1^1 \rangle + \langle e_2^1,\tilde{\mu}_2^1 \rangle  \right) \nonumber\\
	&=&  \frac{\dt}{\mathcal{M}_1}\langle \tau_1^1 , e_1^1 \rangle_{-1,h}
	+ \frac{\dt}{\mathcal{M}_2}\langle \tau_2^1, e_2^1 \rangle_{-1,h} .\label{consistence-13}
\end{eqnarray}
For the right side of \eqref{consistence-13}, using the Cauchy-Schwartz inequality, we have
\begin{eqnarray}
 &&\frac{\dt}{\mathcal{M}_1}\langle \tau_1^1 , e_1^1 \rangle_{-1,h}
	+ \frac{\dt}{\mathcal{M}_2}\langle \tau_2^1, e_2^1 \rangle_{-1,h} \nonumber\\
	&\le&
	\frac{\dt^2}{2\mathcal{M}_1}\nrm{\tau_1^1}_{-1,h}^2+\frac{\dt^2}{2\mathcal{M}_2}\nrm{\tau_2^1}_{-1,h}^2 + \frac{1}{2\mathcal{M}_1}\nrm{e_1^1}_{-1,h}^2+\frac{1}{2\mathcal{M}_2}\nrm{e_2^1}_{-1,h}^2.\label{consistence-14}
\end{eqnarray}
For the third term of the left hand side of \eqref{consistence-13}, similar to $	\langle \tilde{\phi}_1^{n+1} , \tilde{\mu}_1^{n+1} \rangle
+ \langle \tilde{\phi}_2^{n+1} , \tilde{\mu}_2^{n+1} \rangle $, we have
\begin{small}
	\begin{eqnarray}
		&&
		\langle e_1^1 , \tilde{\mu}_1^1 \rangle
		+ \langle e_2^1 , \tilde{\mu}_2^1 \rangle  \nonumber
		\\
		&\ge&
		\frac{\varepsilon_0^2}{144} ( \|  \nabla_h e_1^1  \|_2^2
		+ \| \nabla_h e_2^1  \|_2^2  )
		-   2 C (C^*)^4 (\delta_0)^{-8} (\varepsilon_1^2 + \varepsilon_2^2 + \varepsilon_3^2)
		( \| e_1^1  \|_{-1,h}^2 +  \| e_2^1 \|_{-1,h}^2  ).  \label{convergence-15}
	\end{eqnarray}
\end{small}
Combining \eqref{consistence-13}-\eqref{convergence-15} and \eqref{error-for-initial}, when the initial time step $\dt<\frac{(\delta_0)^{8}}{4\mathcal{M}_0C (C^*)^4  (\varepsilon_1^2 + \varepsilon_2^2 + \varepsilon_3^2)}$, we have the following estimate
\begin{equation}
	\nrm{ \tilde{\phi}_j^1}_{-1,h}^2+ \frac{\varepsilon_0^2\dt}{36}\nrm{\nabla_h \tilde{\phi}_j^1}_2^2 \le  \frac{C\dt^2}{2}(\nrm{\tau_1^1}_{-1,h}^2+\nrm{\tau_2^1}_{-1,h}^2) \le \hat{C}_2(\dt^3+h^3)^2,\label{consistence-16} 
\end{equation}
in which we have used the linear refinement $C_1 h \le \dt \le C_2 h$ in the second step.

Combining \eqref{consistence-16}, taking $ 	\left(\kappa^{(2)}+\frac{\mathcal{M}_1 + \mathcal{M}_2}{\mathcal{M}_1 \mathcal{M}_2}\right)  \dt < \frac{1}{2\mathcal{M}_1}$, and $ 	\left(\kappa^{(2)}+\frac{\mathcal{M}_1 + \mathcal{M}_2}{\mathcal{M}_1 \mathcal{M}_2}\right)  \dt < \frac{1}{2\mathcal{M}_2}$ in \eqref{convergence-11}, we get the following estimate by using the discrete Gronwall inequality
\begin{equation}
	\nrm{ \tilde{\phi}_j^{n+1} }_{-1,h} +\left(\frac{\varepsilon_0^2\dt}{36}\sum_{k=1}^n\nrm{ \nabla_h \tilde{\phi}_j^{k+1} }_2^2\right)^{\nicefrac{1}{2}} \le \hat{C}_3 ( \dt^3 + h^3 ) .
	\label{CH_LOG-convergence-11}
\end{equation}
based on the truncation error accuracy $\| \tau_1^{n+1} \|_{-1,h}$, $\| \tau_2^{n+1} \|_{-1,h} \le C (\dt^3 + h^4)$. This completes the refined error estimate.

\noindent
{\bf Recovery of the a-priori assumption~\eqref{a priori-1}}

With the error estimate~\eqref{CH_LOG-convergence-11} at hand, we notice that the a-priori assumption in~\eqref{a priori-1} is satisfied at the next time step $t^{n+1}$: we observe that the $L_{\dt}^2 (0, T; H_h^1)$ error estimate in~\eqref{convergence-11} implies that
\begin{eqnarray*}
	\| \nabla_h \tilde{\phi}_j^{n+1} \|_2
	\le \frac{C \hat{C}_3 ( \dt^3 + h^3 )}{\dt^\frac12}
	\le C \hat{C}_3 ( \dt^\frac52 + h^\frac52 )  , 
\end{eqnarray*}
in which we have used the linear refinement $C_1 h \le \dt \le C_2 h$ in the second step. Moreover, since $\overline{\tilde{\phi}_1^{n+1}} = \overline{\tilde{\phi}_2^{n+1}} =0$, an application of discrete Poincar\'e inequality implies that
\begin{equation}
	\| \tilde{\phi}_j^{n+1} \|_2 \le C \| \nabla_h \tilde{\phi}_j^{n+1} \|_2
	\le C^2 \hat{C}_3 ( \dt^\frac52 + h^\frac52 )
	\le \dt^\frac74 + h^\frac74  ,  \quad j=1, 2 , 
	\label{CH_LOG-convergence-12}
\end{equation}
provided that $\dt$ and $h$ are sufficiently small. This completes the proof of Theorem~\ref{thm:convergence}. 

\begin{rem} 
The positivity-preserving and energy stability analyses, as stated in Theorems~\ref{MMC-positivity} and \ref{MMC-energy stab}, are unconditional, and there is no requirement for the time step size in terms of the spatial mesh size. Meanwhile, in the statement of the convergence analysis and error estimate in Theorem~\ref{thm:convergence}, a linear refinement condition is required for the time step size, namely $C_1 h \le \dt \le C_2 h$, for certain technical reasons. In fact, this requirement is not the standard CFL condition, although it takes a similar form. In more details, such a linear refinement condition does not come from the stability requirement of the numerical scheme; instead, this condition comes from repeated applications of inverse inequality, as revealed in the rough error estimates~\eqref{a priori-2}, \eqref{convergence-rough-10}, etc. A careful calculation implies that, a combination of $\dt \le C_2 h$ and $\dt \ge C_1 h$ enables us to derive the desired $\| \cdot \|_\infty$ rough error estimate~\eqref{convergence-rough-11}, so that the phase separation property~\eqref{assumption:separation-3} becomes available for the numerical solution at the next time step, which will play an essential role in the refined error estimate.

Meanwhile, such a linear refinement condition (for the time step size) could be improved with the help of an even higher order consistency analysis via asymptotic expansion. In subsection~\ref{subsec: consistency}, we have performed an $O (\dt^3 + h^4)$ consistency estimate, and this consistency order is able to ensure the desired $\| \cdot \|_\infty$ rough error estimate~\eqref{convergence-rough-11}, under the linear refinement condition. Instead, if we perform an $O (\dt^4 + h^4)$ consistency estimate, with the help of higher order asymptotic expansion, the desired $\| \cdot \|_\infty$ rough error estimate~\eqref{convergence-rough-11} could be derived with an improved time step constraint: $C_1 h^\frac32 \le \dt \le C_2 h^\frac23$, and the a-priori assumption~\eqref{a priori-1} could be rewritten as $\| \tilde{\phi}_j^k \|_2 \le \dt^3 + h^3$, $k=n-1, n$, $j=1, 2$. Of course, this constraint is much milder than the linear refinement requirement, and the time step size could be taken in the scale from $O (h^\frac32)$ to $O (h^\frac23)$.  

In fact, under the assumption that the exact solution is sufficiently smooth, with higher and higher order consistency estimate via asymptotic expansion, such a time step constraint could be even improved to $C_1 h^{\beta_0} \le \dt \le C_2 h^{\alpha_0}$, for any $\beta_0 > 1$ and $0 < \alpha_0 <1$. With a smaller value of scaling power index $\alpha_0$ and a larger value of $\beta_0$, there is more freedom in the choice of the time step size $\dt$. The technical details are skipped for the sake of brevity. In fact, the corresponding constraint for the time step size is only a technical issue in the theoretical justification of the convergence analysis.   
\end{rem} 

\begin{rem}
The convergence estimate~\eqref{convergence-0} (stated in Theorem~\ref{thm:convergence}) gives a second order convergence rate for the phase variables, in the $\ell^\infty (0, T; H_h^{-1})$ norm. Meanwhile, based on the higher order consistency estimate via the asymptotic expansion approach, we are able to derive the second order $\ell^\infty (0, T; \ell^2)$ convergence estimate. In particular, the higher order refined error estimate~\eqref{CH_LOG-convergence-11} leads to an $\ell^2$ error estimate \eqref{CH_LOG-convergence-12}, with convergence order $O (\dt^\frac52 + h^\frac52)$. On the other hand, by the asymptotic expansion~\eqref{consistency-1} for the constructed profile $\breve{\Phi}_j$, combined with the definition~\eqref{error function-2} for the higher order error functions, we immediately conclude that 
\begin{equation} 
\begin{aligned} 
  \| e^{n+1}_j \|_2 = & \| \tilde{\phi}_j^{n+1} - {\cal P}_N (\dt^2 \Phi_{j,\dt} + h^2 \Phi_{j,h}) \|_2 
\\
  \le & \| \tilde{\phi}_j^{n+1} \|_2 + \dt^2 \| {\cal P}_N \Phi_{j,\dt} \|_2 + h^2 \| {\cal P}_N \Phi_{j,h} \|_2 
\\
  \le & 
  C^2 \hat{C}_2 (\dt^\frac52 + h^\frac52 ) + C ( \dt^2 + h^2) 
  \le C (\dt^2 + h^2 ) . 
\end{aligned} 
  \label{CH_LOG-convergence-13} 
\end{equation} 
As a result, a discrete $L^2$ error estimate has been theoretically established, with the second order accuracy in both time and space.

Of course, such a second order $L^2$ convergence estimate is under the linear refinement constraint condition, $C_1 h \le \dt \le C_2 h$. Under a milder constraint, $C_1 h^{\beta_0} \le \dt \le C_2 h^{\alpha_0}$, with $\beta_0 > 1$, $0 < \alpha_0 < 1$, a similar second order $L^2$ error estimate could be derived in a similar manner; the technical details are skipped for simplicity of presentation. 
\end{rem} 
\section{Numerical results}    \label{sec:numerical results}

In this section, we present several numerical experiments using the proposed scheme. The nonlinear Full Approximation Scheme (FAS) multigrid method is used for solving the semi-implicit numerical scheme \eqref{Full-discrete-1} -- \eqref{Full-discrete-mu2}. The details are similar to earlier works~\cite{baskaran13a, chen19b, Dong2018Convergence, Dong2020b, feng2018bsam, hu09, wise10}, etc. We take the domain $\Omega = [0,64]^2$, fixed space resolution $N = 256$ and choose the parameters in the model as $M_0 = 0.16, N_0 = 5.12,  \chi_{12} =4, \chi_{13} = 10, \chi_{23} = 1.6$ and $ \mathcal{M}_1 = \mathcal{M}_2 = 1.0$. In addition, we set the artificial parameters as $A_1=1.25\chi_{13}^2+0.25(\chi_{12}-\chi_{13}-\chi_{23})^2$ and $A_2=2\chi_{23}^2+2(\chi_{12}-\chi_{13}-\chi_{23})^2$.
\begin{example}\label{example 1}
	The initial data is set as
	\begin{eqnarray}
		&\phi_1^0(x,y) = 0.1+0.01\cos\big(3\pi x/32\big) \cos\big(3\pi y/32\big),\nonumber\\
		&\phi_2^0(x,y) = 0.5+0.01\cos\big(3\pi x/32\big) \cos\big(3\pi y/32\big).\label{eqn:init1}
	\end{eqnarray}
\end{example}
This example is designed to study the numerical accuracy in time. We use a linear refinement path, \emph{i.e.}, $\Delta t=Ch$. At the final time $T=0.4$, we expect the global error to be $\mathcal{O}(\Delta t^2)+\mathcal{O}(h^2)=\mathcal{O}(h^2)$ under either the  $\ell^2$ or $\ell^\infty$ norm, as $h, \Delta t\to 0$.  Since we do not have an exact solution, instead of calculating the error at the final time, we compute the Cauchy difference, which is defined as $\delta_\phi: =\phi_{h_f}-\mathcal{I}_c^f(\phi_{h_c})$, where $\mathcal{I}_c^f$ is a bilinear interpolation operator (We applied Nearest Neighbor Interpolation in Matlab, which is similar to the 2D case in \cite{feng2016preconditioned}). This requires having a relatively coarse solution, parametrized by $h_c$, and a relatively fine solution, parametrized by $h_f$, where $h_c = 2 h_f$, at the same final time. The $\ell^2$ and $\ell^\infty$ errors for $\phi_1$~and $\phi_2$ are displayed in Table \ref{tab:cov1}, respectively. The  results confirm our expectation for the convergence order. 
\begin{table}[!htp]
	\begin{center}
		\begin{tabular}{|c|c|c|c|}
			\hline Grid sizes&$16^2- 32^2$&$32^2- 64^2$& $64^2- 128^2$  \\
			\hline $\ell^2$-error-$\phi_1$& $2.7223 \times 10^{-2}$&$ 7.0546 \times 10^{-3}$ &$ 1.8240 \times
			10^{-3}$\\
			\hline Rate&-& 1.95  &   1.95 \\
			\hline $\ell^2$-error-$\phi_2$& $2.6907 \times 10^{-2}$&$ 6.8618 \times 10^{-3}$ &$ 1.7182 \times
			10^{-3}$\\
			\hline Rate&-& 1.97  &   2.00
			\\
			\hline $\ell^\infty$-error-$\phi_1$& $8.2277 \times 10^{-4}$&$ 2.1980 \times 10^{-4}$ &$ 5.7373 \times
			10^{-5}$\\
			\hline Rate&-& 1.90  &    1.94  \\
			\hline $\ell^\infty$-error-$\phi_2$& $8.1155 \times 10^{-4}$&$ 2.1248 \times 10^{-4}$ &$ 5.4786 \times
			10^{-5}$\\
			\hline Rate&-& 1.93  &    1.96  \\
			\hline
		\end{tabular}
	\caption{Errors and convergence rates. The $\ell^2$ error, $\ell^{\infty}$  error and convergence rate for $\phi_1$ and $\phi_2$ when $T=0.4$. The initial data are defined in \eqref{eqn:init1}. The refinement path is $\Delta t=0.002h$.} \label{tab:cov1}
	\end{center}
\end{table}
\begin{example}\label{example 2}
	A random initial perturbation is included in the initial data:
	\begin{eqnarray}
		\phi_1^0(x,y) = 0.1+r_{i,j},\nonumber\\
		\phi_2^0(x,y) = 0.4+r_{i,j},\label{eqn:init2}
	\end{eqnarray}
	where the $r_{i,j}$ are uniformly distributed random numbers in [-0.01, 0.01].
\end{example}

This example is designed to test the performance of the proposed scheme in preserving physical properties at discrete level.
\begin{figure}[!htp]
	\begin{center}	
		\includegraphics[width=2.5in]{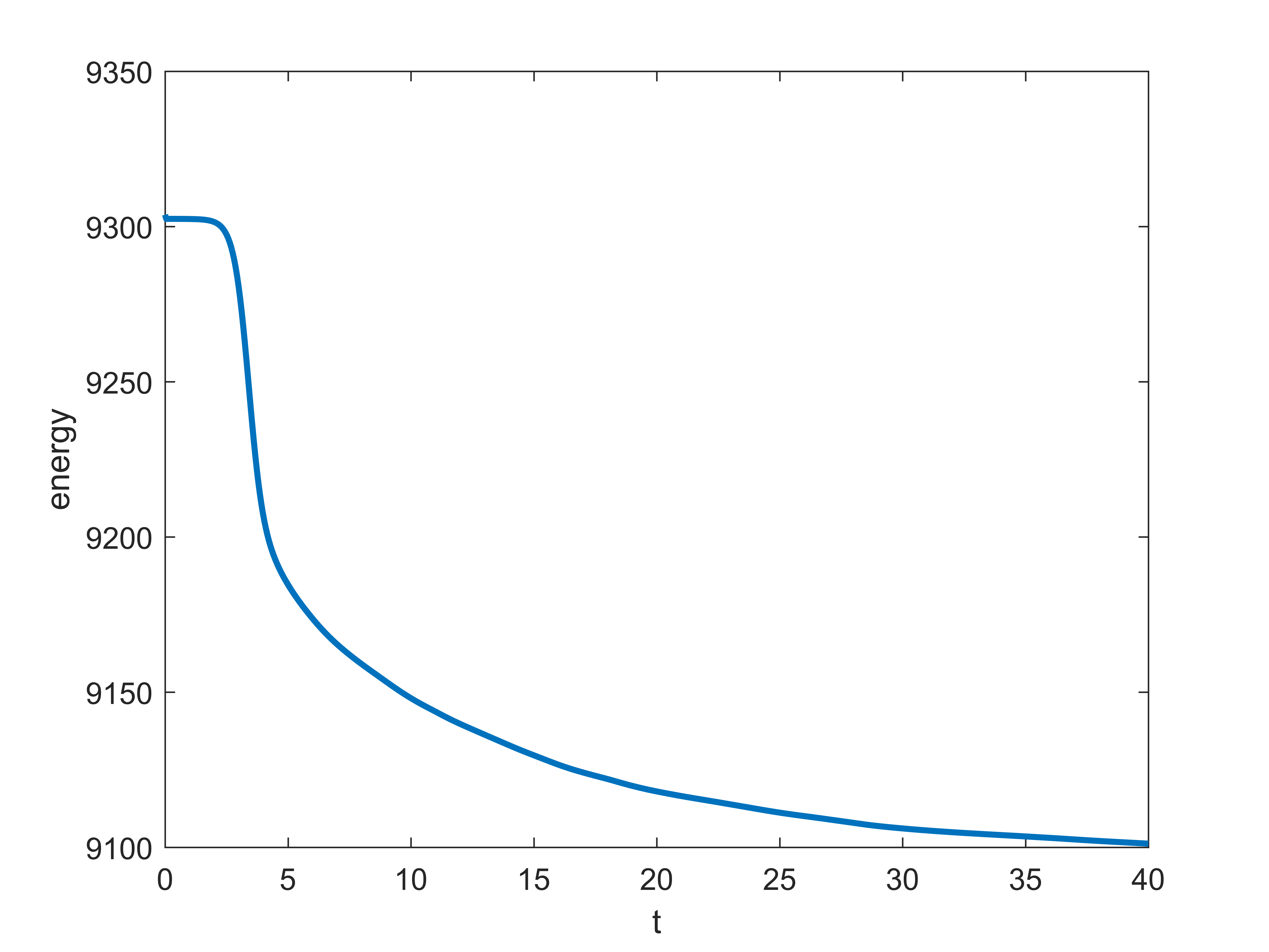}
		\caption{Example \ref{example 2}: Evolution of the energy over time. The time step size is taken as $\dt = 1.0 \times 10^{-4}$. }\label{fig:coslong_energy}
	\end{center}
\end{figure}
\begin{figure}[ht]
	\begin{center}
		\begin{subfigure}{}
			\includegraphics[width=2.2in]{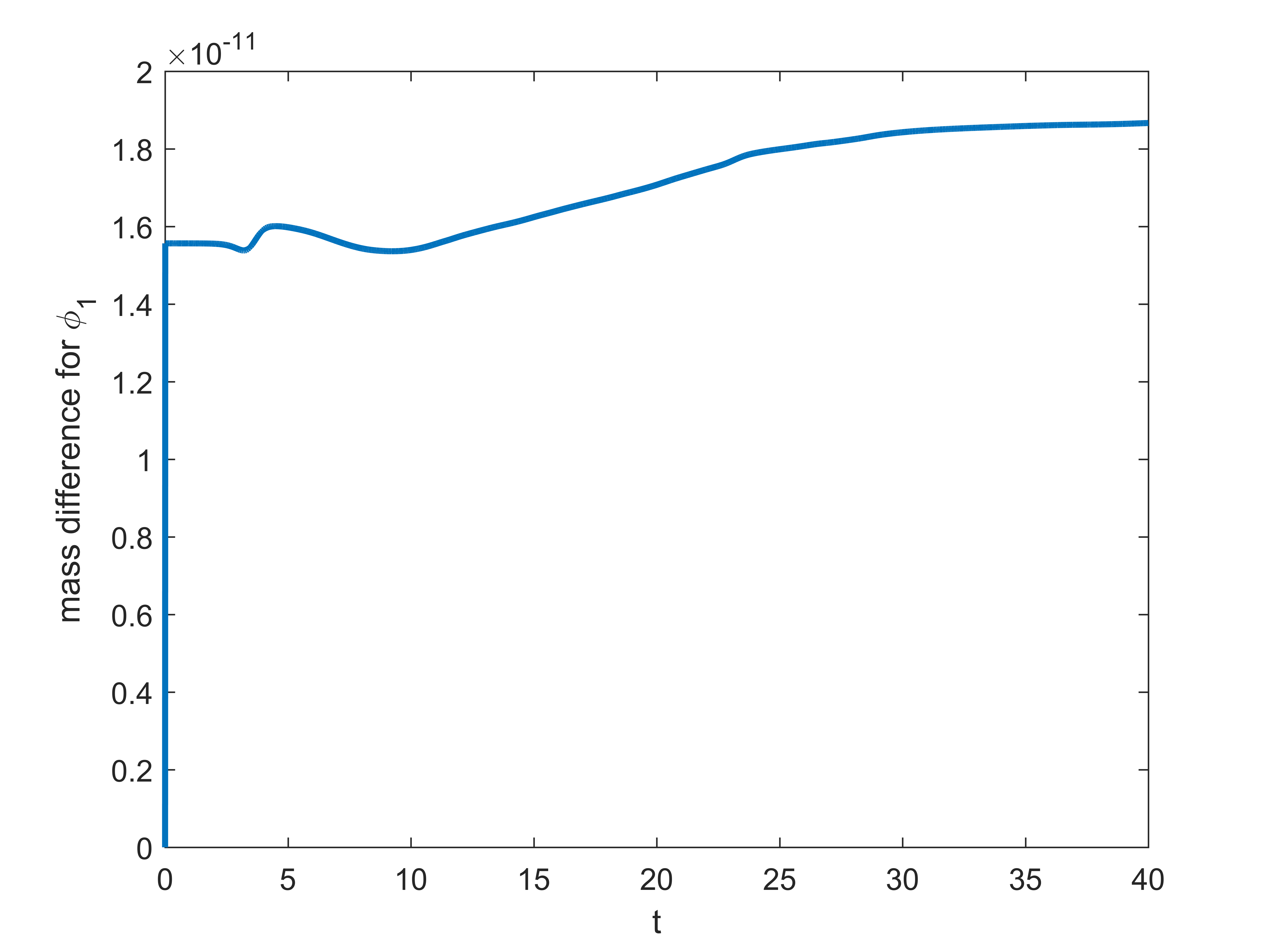}
		\end{subfigure}
		\begin{subfigure}{}
			\includegraphics[width=2.2in]{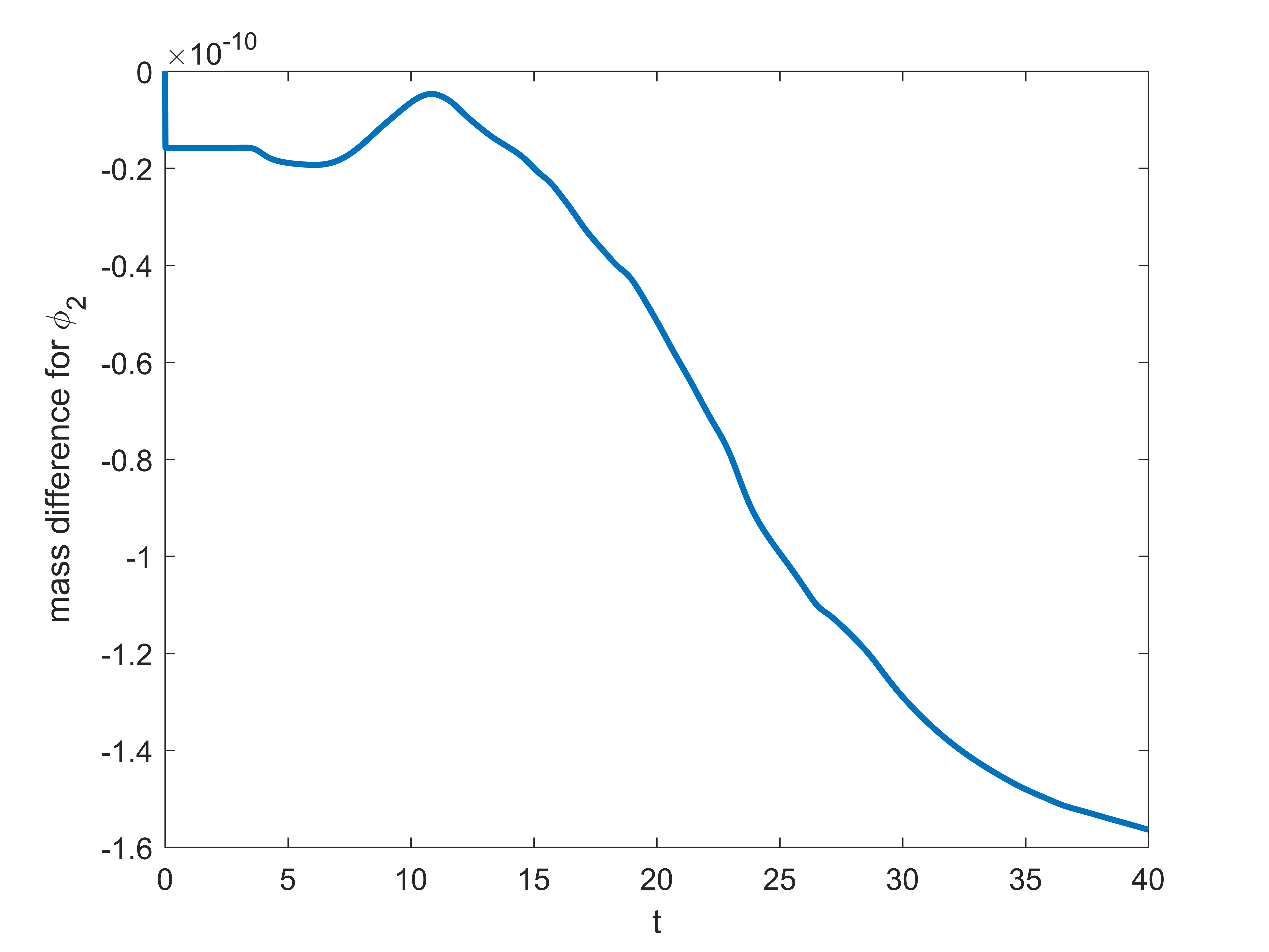}
		\end{subfigure}
	\end{center}
	\caption{Example \ref{example 2}: The error developments of the total mass for $\phi_1$ and $\phi_2$, respectively. }
	\label{fig:cosmass}
\end{figure}
\begin{figure}
	\centering
	\includegraphics[width=2.6cm,height=2.6cm]{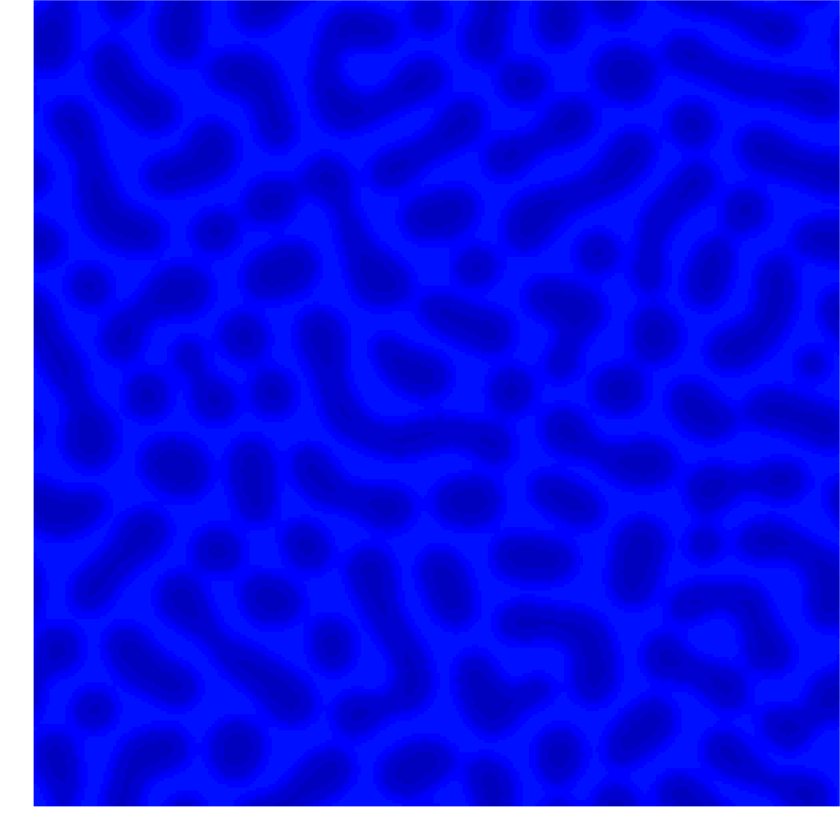}
	\includegraphics[width=2.6cm,height=2.6cm]{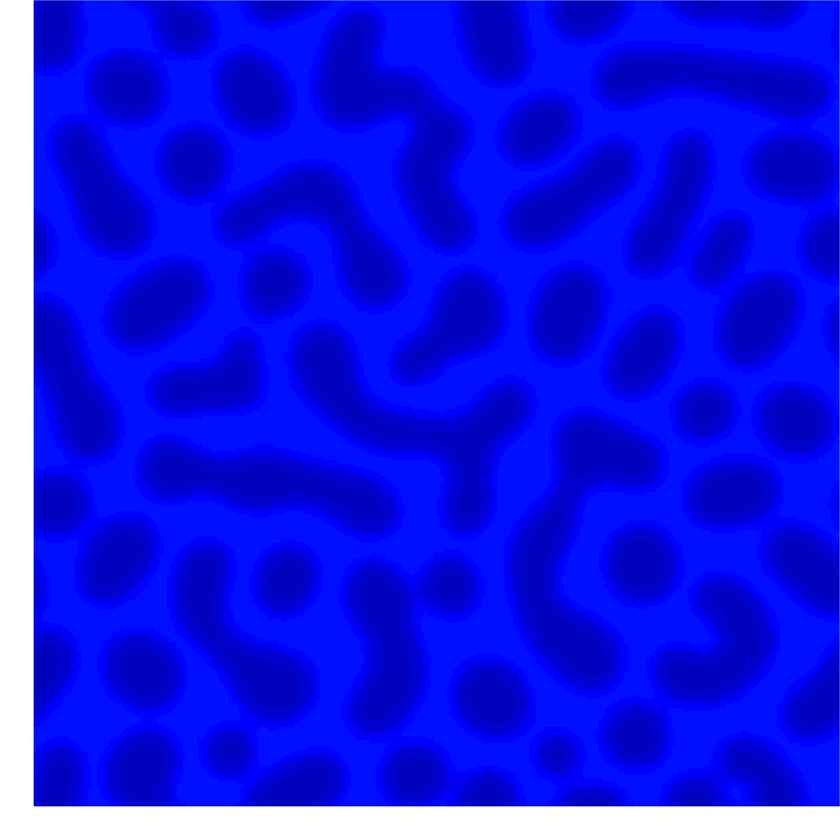}
	\includegraphics[width=2.6cm,height=2.6cm]{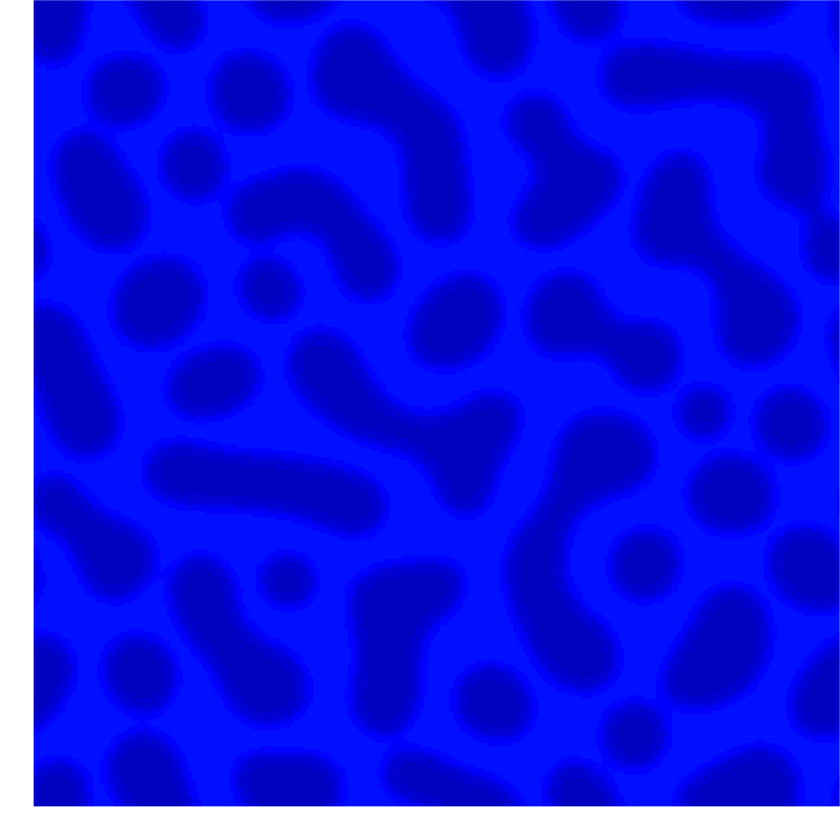}
	\includegraphics[width=2.6cm,height=2.6cm]{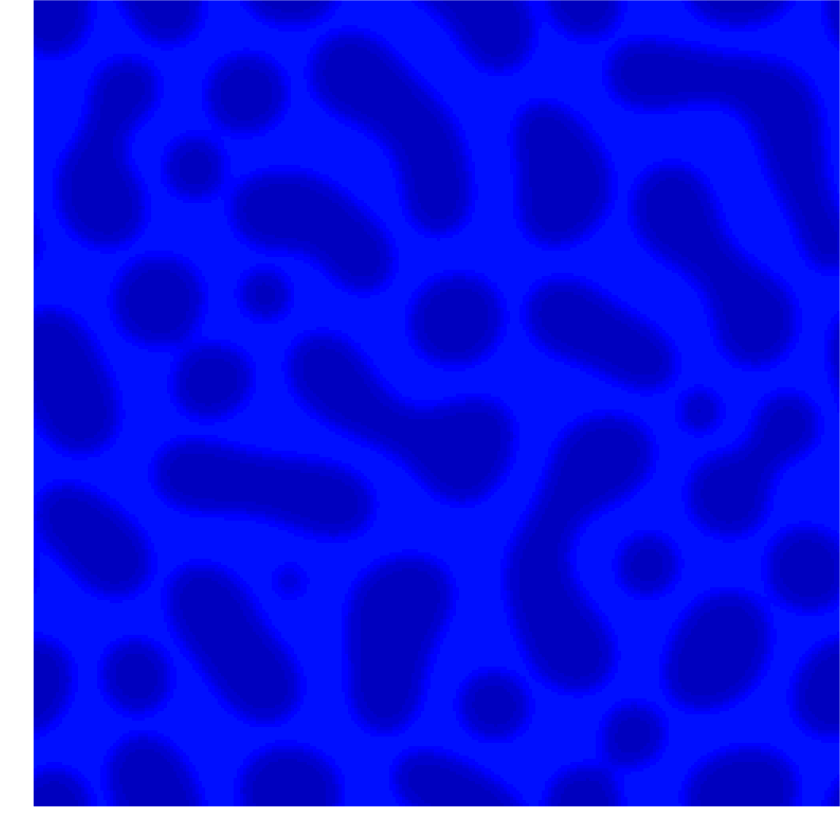}\\
	\includegraphics[width=2.6cm,height=2.6cm]{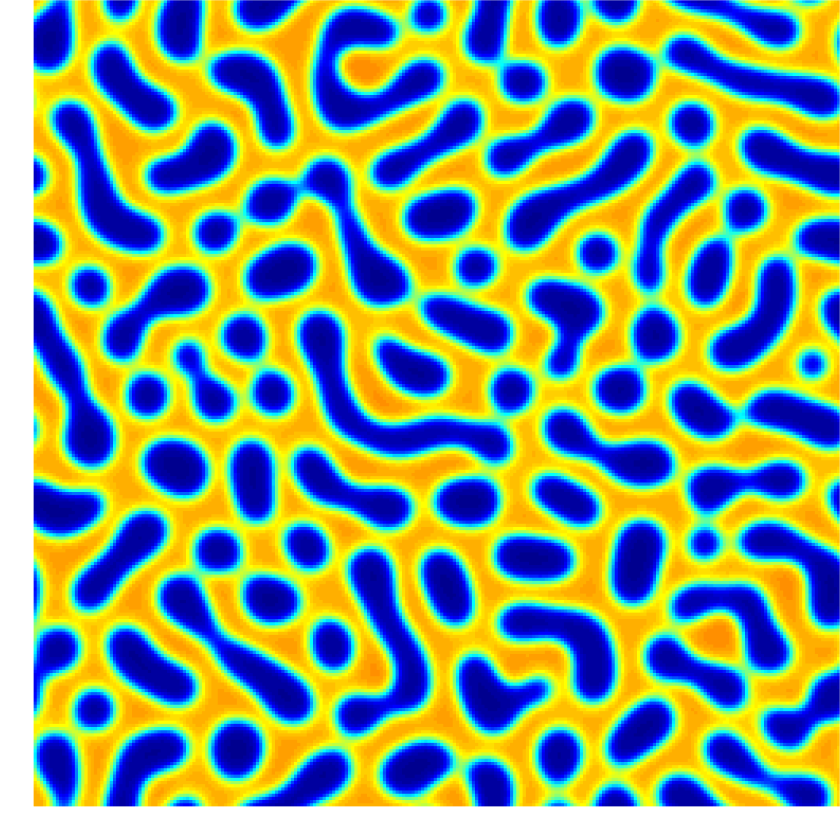}
	\includegraphics[width=2.6cm,height=2.6cm]{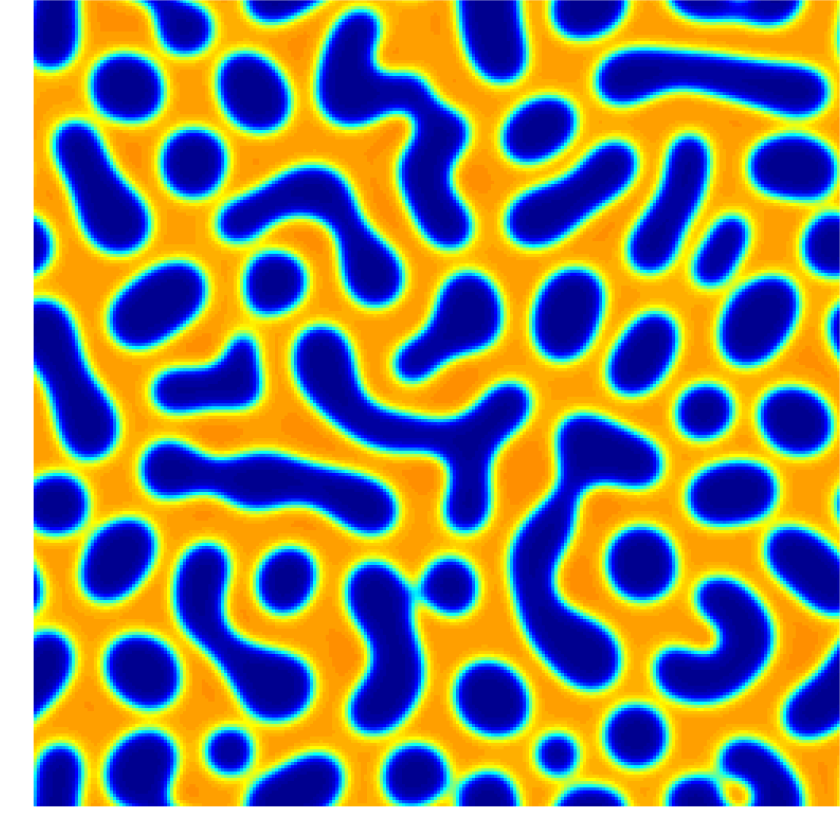}
	\includegraphics[width=2.6cm,height=2.6cm]{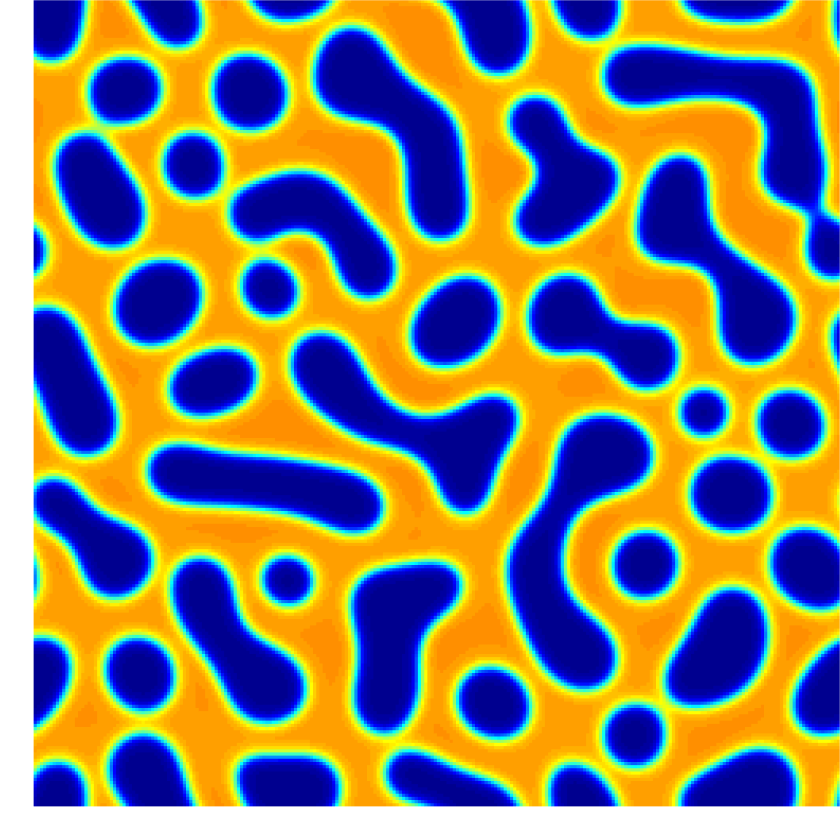}
	\includegraphics[width=2.6cm,height=2.6cm]{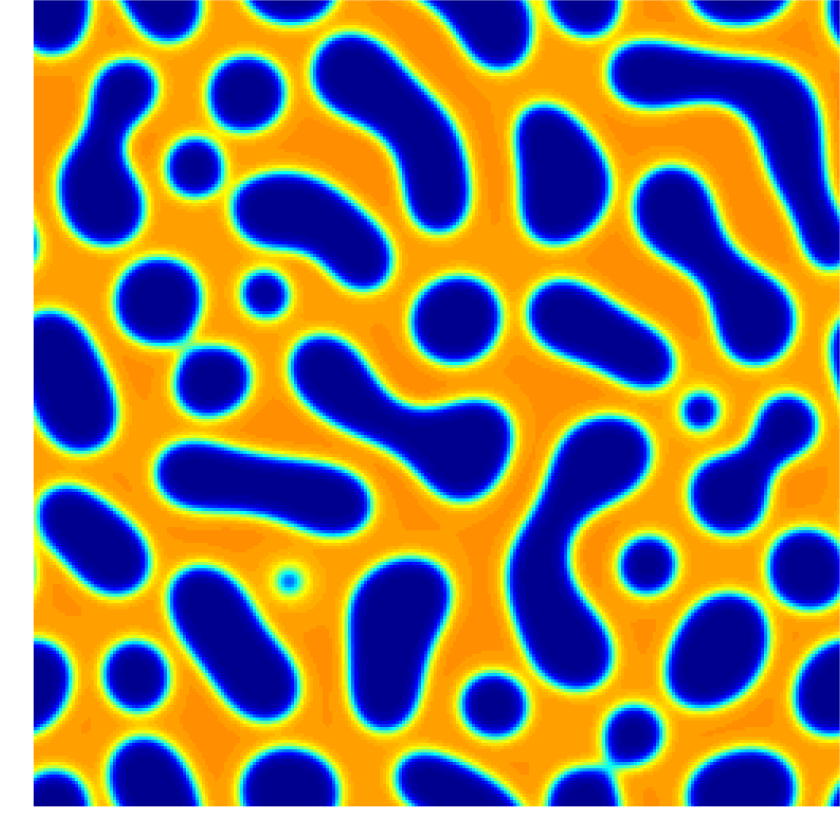}\\
	\includegraphics[width=2.6cm,height=2.6cm]{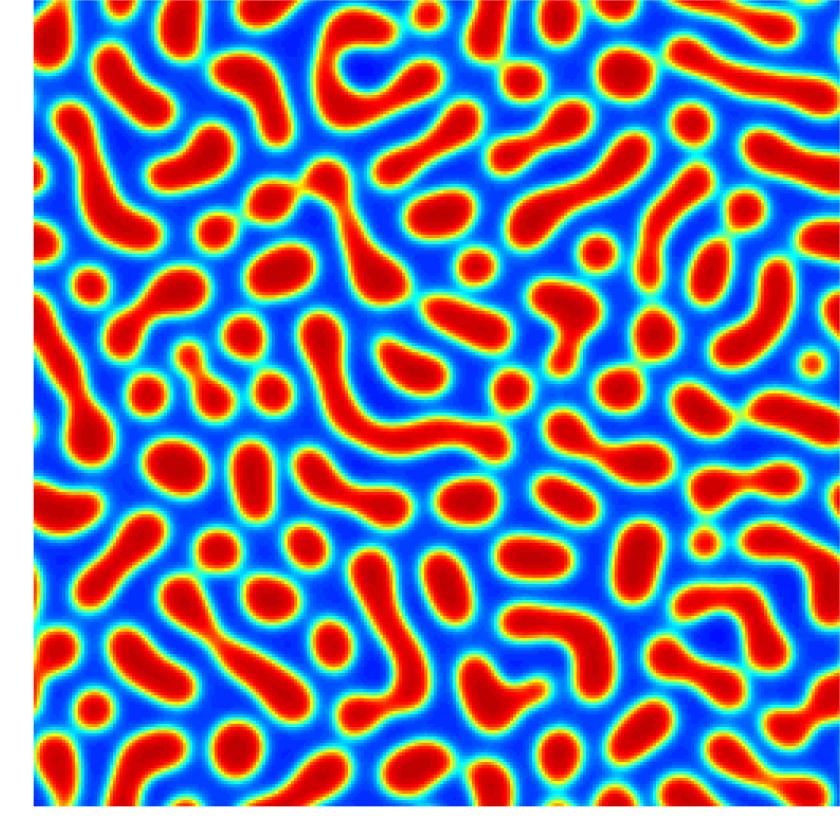}
	\includegraphics[width=2.6cm,height=2.6cm]{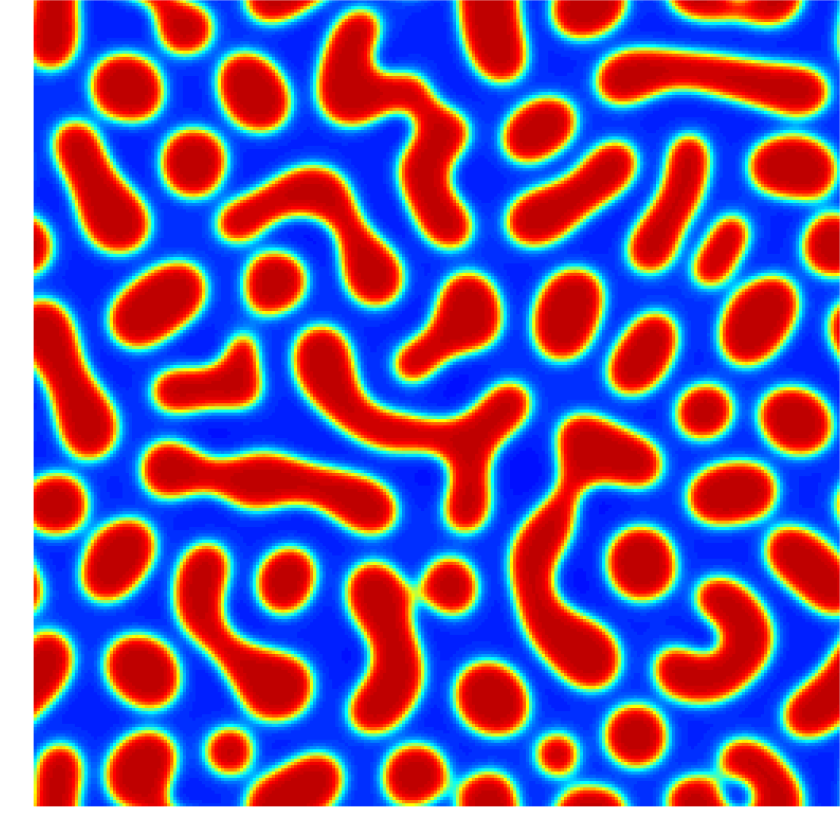}
	\includegraphics[width=2.6cm,height=2.6cm]{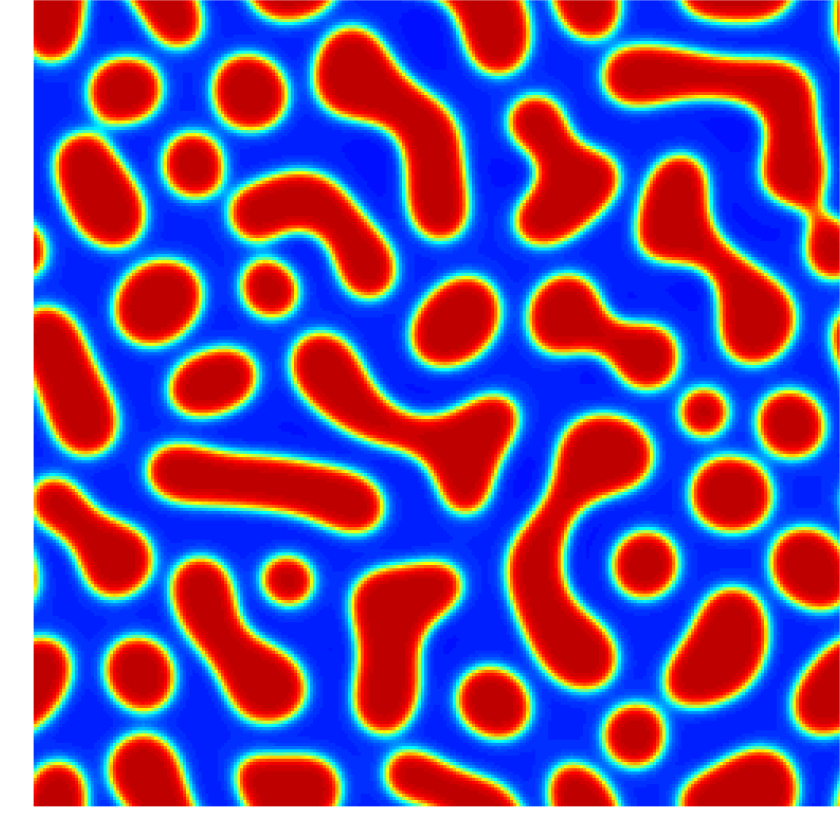}
	\includegraphics[width=2.6cm,height=2.6cm]{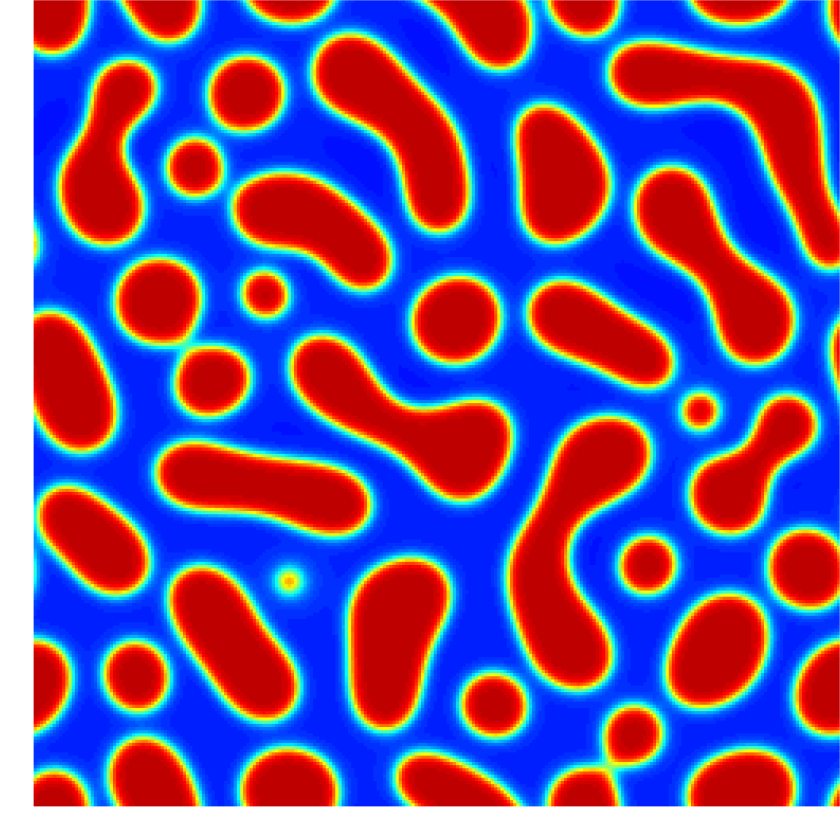}\\
    $t=10$\hspace{1.6cm}$t=20$\hspace{1.6cm}$t=30$\hspace{1.6cm}$t=40$
\caption{Example \ref{example 2}: Evolution of three phase variables at $t = 10, 20, 30$ and $40$. The first line is for $\phi_1$, the second line is for $\phi_2$ and the last line is for $\phi_3$. The time step size is taken as $\dt = 1.0 \times 10^{-4}$. }\label{fig:coslong}
\end{figure}
\begin{figure}[ht]
	\begin{center}
		\begin{subfigure}{}
			\includegraphics[width=2.2in]{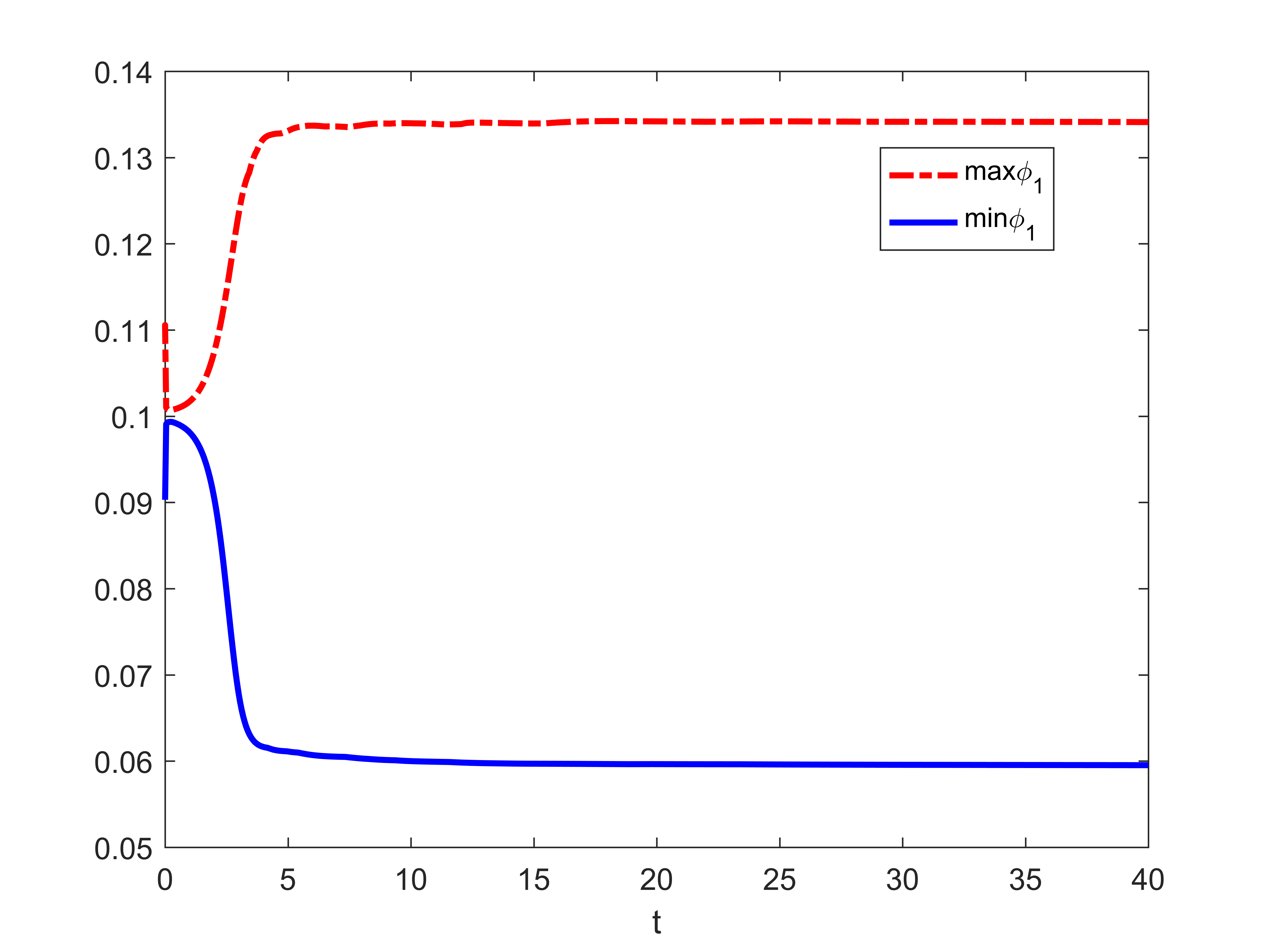}
		\end{subfigure}
		\begin{subfigure}{}
			\includegraphics[width=2.2in]{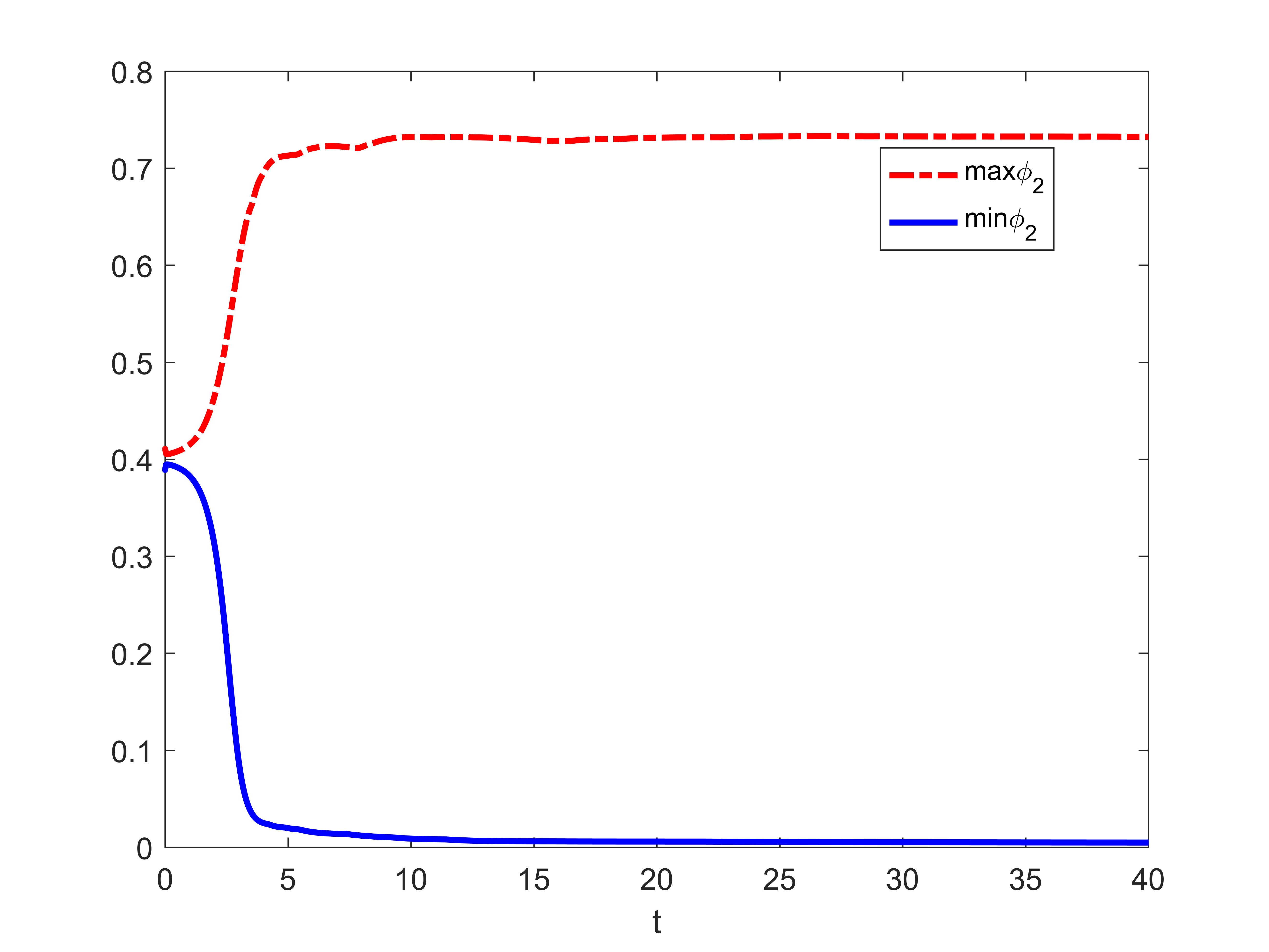}
		\end{subfigure}
	\end{center}
	\caption{Example \ref{example 2}: The time evolution of the maximum and minimum values for $\phi_1$ and $\phi_2$, respectively. }
	\label{fig:cosmaxmin}
\end{figure}
\begin{figure}[!htp]
	\begin{center}	
		\includegraphics[width=2.5in]{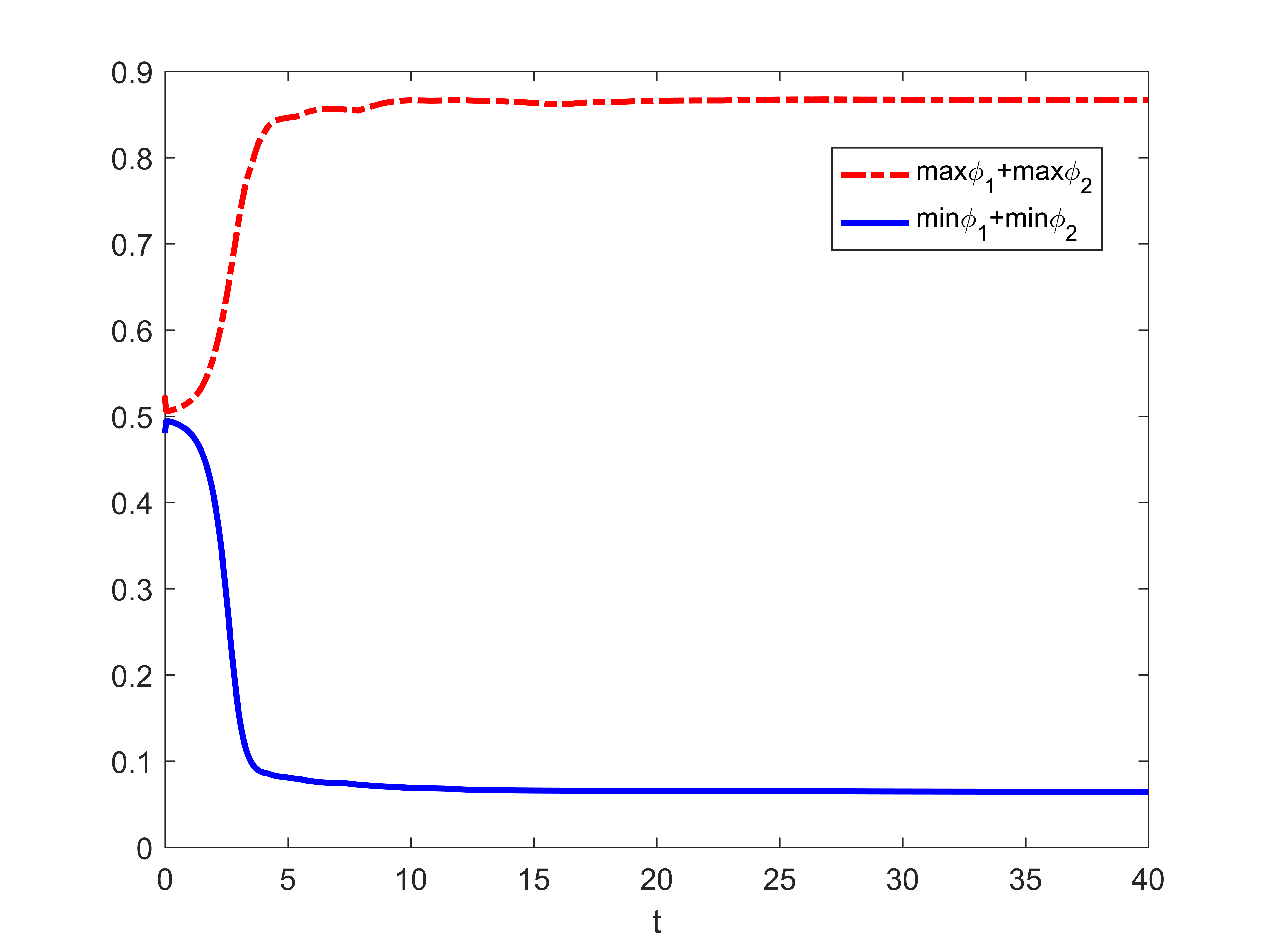}
		\caption{Example \ref{example 2}: The time evolutions of the maximum and minimum values for $\phi_1+\phi_2$, with $\dt = 1.0 \times 10^{-4}$.}\label{fig:coslong_maxminadd}
	\end{center}
\end{figure}
The energy evolution of the numerical solution with $\dt = 1.0\times 10^{-4}$ is illustrated in Figure \ref{fig:coslong_energy}, which indicates a clear energy decay. In Figure \ref{fig:cosmass}, we also present the error evolutions of the total mass of $\phi_1$ and $\phi_2$. In Figure \ref{fig:coslong}, the snapshot plots of $\phi_1$, $\phi_2$ and $\phi_3$ at a sequence of time instants are displayed, to make a comparison with the existing ternary MMC results. Moreover, the maximum values and minimum values of $\phi_1$, $\phi_2$ and $\phi_1+\phi_2$ are presented in Figure \ref{fig:cosmaxmin} and Figure \ref{fig:coslong_maxminadd}. In summary, our numerical tests further confirm that the proposed numerical scheme respects mass conservation, energy dissipation, and positivity at discrete level.
\section{Conclusions} \label{sec:conclusion}

A second order finite difference numerical scheme is proposed and analyzed for the ternary MMC system. The BDF temporal discrete and second-order Adams-Bashforth extrapolation formula has been used to construct the full discrete scheme.
In the proposed numerical algorithmic,  a unique solvability and positivity-preserving property turn to be available. Combined Douglas-Dupont regularization term, the energy stability property is estimated. Moreover, the second order convergence analysis are available in the theoretical level. To overcome a well-known difficulty associated with the highly nonlinear and singular nature of the surface diffusion coefficients, a rough error estimate has to be performed, so that the $\ell^{\infty}$ bound for $\phi_i$ could be derived. This $\ell^\infty$ estimate yields the upper and lower bounds of the three variables, and these bounds play a crucial role in the subsequent analysis. Finally, the refined error estimate is carried out to accomplish the desired convergence result. In addition, mass conservation, energy stability, bound of the numerical solution and the second order accurate are demonstrated in the numerical experiments. 

	\section*{Acknowledgements}

L.X.~Dong is partially supported by the National Natural Science Foundation of China (No. 12201051, 12371396). C.~Wang is partially supported by the National Science Foundation (No. DMS-2012269, DMS-2309548). Z.R.~Zhang is partially supported by the National Natural Science Foundation of China (No. 11871105, No. 12231003).  
\newpage
\bibliographystyle{plain}
\bibliography{draft2}

\end{document}